\documentclass{article}
\usepackage[utf8]{inputenc}
\usepackage[T1]{fontenc} 
\usepackage{amsmath}
\usepackage{amsfonts}
\usepackage{amssymb}
\usepackage{amsthm}
\usepackage{color}
\usepackage{esint}
\usepackage{enumerate}
\usepackage{dsfont}
\usepackage{psfrag}
\usepackage{stmaryrd}
\usepackage{hyperref}
\usepackage{graphicx}
\usepackage{graphics}
\usepackage{epstopdf}
\usepackage{epsfig}
\usepackage{setspace}
\usepackage{tikz}
\usetikzlibrary{positioning,arrows}
\usetikzlibrary{calc,arrows}
\usetikzlibrary{er,positioning,bayesnet}

\usepackage{pgfplots}
\usepackage{pgfplotstable}
\pgfplotsset{compat=1.11} 

\usepackage{bm}
\usepackage{booktabs} 
\usepackage{caption} 
\usepackage{subcaption} 
\usepackage[all]{nowidow}
\usepackage{multicol}
\usepackage{mathabx}
\usepackage{algpseudocode,algorithm,algorithmicx}
\usepackage{hyperref}
\usepackage[inline]{enumitem} 

\newtheorem{theorem}{Theorem}
\newtheorem{lemma}[theorem]{Lemma}
\newtheorem{remark}[theorem]{Remark}
\newtheorem{coro}[theorem]{Corollary}

\newcommand{\RR}{\mathbb{R}}

\newcommand{\NN}{\mathbb{N}}

\newcommand{\eps}{\varepsilon}

\newcommand{\dps}{\displaystyle}

\definecolor{blue}{HTML}{1F77B4}
\definecolor{orange}{HTML}{FF7F0E}
\definecolor{green}{HTML}{2CA02C}

\begin{document}
\date{\today}
\author{Olga Gorynina, Claude Le Bris and Frédéric Legoll
\\
{\footnotesize École Nationale des Ponts et Chaussées and Inria Paris,} \\
{\footnotesize 6 et 8 avenue Blaise Pascal, 77455 Marne-La-Vall\'ee Cedex 2, France} \\
{\footnotesize \tt \{olga.gorynina,claude.le-bris,frederic.legoll\}@enpc.fr}\\
}
\title{Mathematical analysis of a coupling method for the practical computation of homogenized coefficients}

\maketitle

\begin{abstract}
We present the mathematical study of a computational approach originally introduced by R.~Cottereau in~\cite{cottereau}. The approach aims at evaluating the effective (a.k.a. homogenized) coefficient of a medium with some fine-scale structure. It combines, using the Arlequin coupling method, the original fine-scale description of the medium with an effective description and optimizes upon the coefficient of the effective medium to best fit the response of an equivalent purely homogeneous medium. We prove here that the approach is mathematically well-posed and that it provides, under suitable assumptions, the actual value of the homogenized coefficient of the original medium in the limit of asymptotically infinitely fine structures. The theory presented here therefore usefully complements our numerical developments of~\cite{ref_olga_comp}.
\end{abstract}

\section{Introduction}

The work~\cite{cottereau} has introduced a domain decomposition approach for the specific purpose of approximating the homogenized coefficient of a heterogeneous medium. In short, the approach consists in dividing the computational domain in two overlapping subdomains (see Figure~\ref{fig:decompo_left} below). The first, inner subdomain explicitly accounts for the fine-scale structure. In the second, outer subdomain, an effective medium is considered. The two subdomains overlap, typically over an annular layer, where both models, the fine-scale model and the effective model, coexist. Suitable boundary conditions are imposed on the outer boundary of the domain. The bottom line of the approach then consists in optimizing upon the coefficient of the effective medium in order to best fit the response that would be obtained if the effective coefficient employed were the actual homogenized coefficient corresponding to the fine scale structure. The approach thus provides a computational strategy to approximate the homogenized coefficient which is an alternative to standard homogenization techniques. In particular, and in the same vein as some other approaches previously proposed in the literature~\cite{BouQW:88,Durlo:91,cocv,cras_kun_li}, it does not require computing the usual ingredients of homogenization such as corrector functions before providing an approximation of the homogenized coefficient. We refer to the companion article~\cite{ref_olga_comp} for more details motivating the approach of~\cite{cottereau}.

\medskip

Let us at once say that, in short, the conclusion of the mathematical study conducted herein is that the method introduced in~\cite{cottereau} from a purely computational perspective is mathematically sound.

\medskip

In some more details, our study follows the following pattern. We work, as in~\cite{cottereau,ref_olga_comp}, on the simple, linear, in divergence form diffusion equation
\begin{equation} \label{eq:diffusion}
  -{\rm div} \left( k_\eps \nabla u_\eps \right) = f,
\end{equation}
which is posed in $\Omega$, a bounded domain of $\RR^d$. We assume for simplicity of exposition that $d=2$, but our mathematical study carries over to a higher dimensional setting in a straightforward way. The practically relevant case is of course $d=3$, but we wish to spare the reader the required adjustments of our arguments. In~\eqref{eq:diffusion}, the coefficient $k_\eps$ models the fine-scale structure of the actual medium (typically a complex material) considered, the effective coefficient of which we aim at approaching. The parameter $\eps>0$, presumably small, encodes the size of the fine-scale structure, supposedly tiny: $\eps \ll 1$. The coefficient $k_\eps$ may be scalar-valued, or matrix-valued. Our study is actually insensitive to this distinction, and we hence assume throughout our article that $k_\eps$ is matrix-valued. Furthermore, we assume that $k_\eps$ is a {\em symmetric} matrix. This assumption will be used e.g. to write the Euler-Lagrange equations of the optimization problem~\eqref{eq:pb_min} in the form of~\eqref{eq:arlequin_var}, and when writing the estimate~\eqref{eq:vraiment_energie}.

In addition, we assume the following classical boundedness and coercivity conditions:
\begin{equation} \label{eq:boundedness+coercivity}
\forall \xi \in \RR^2, \quad k_\eps(x) \xi \cdot \xi \geq c_1 |\xi|^2 \quad \text{and} \quad | k_\eps(x) \xi | \leq c_2 |\xi| \qquad \text{a.e. on $\Omega$},
\end{equation}
for two constants $c_1>0$, $c_2>0$ independent from $\eps$.

The computational approach introduced in~\cite{cottereau} and briefly outlined above can then be put in action for any such coefficient and any value of the parameter $\eps$. In theory though, the existence of a homogenized coefficient $k^\star$, and, foremost, the existence of a coefficient $k^\star$ \emph{amenable to practical computations}, so that equation~\eqref{eq:diffusion} converges in the limit of asymptotically small parameters $\eps$ to a homogenized equation of the type
\begin{equation} \label{eq:homogenized-equation}
  -{\rm div}\left(k^\star \nabla u^\star\right) = f,
\end{equation}
for some homogenized coefficient $k^\star$, requires more stringent assumptions on $k_\eps$. In our study, we will assume $k_\eps$ is of the form
\begin{equation} \label{eq:structure-k}
  k_\eps(x) = k_{\rm per}(x/\eps),
\end{equation}
for some fixed coefficient $k_{\rm per}$, which we further assume \emph{periodic}, as a prototypical case of a large class of adequate structures for which quantitative homogenization holds. For instance, the theoretical setting and the computational approach (see~\cite{cottereau,ref_olga_comp}) carry over to the case of a \emph{random stationary} coefficient $k_\eps$. The homogenized coefficient $k^\star$, which in the case~\eqref{eq:structure-k} of a periodic coefficient is \emph{constant}, may be, like the coefficient $k_\eps$, scalar-valued or matrix-valued, the latter case being even possible despite the fact that $k_\eps$ is scalar-valued. In the most part of our mathematical study below, we assume for simplicity that $k^\star$ is scalar-valued. In Appendix~\ref{seq:matrix}, we make precise how our arguments should be modified to, for the most part, carry over to the case of a matrix-valued homogenized coefficient.

\medskip

The approach of~\cite{cottereau} considers, to begin with, a coupling of the two equations~\eqref{eq:diffusion} for the actual fine-scale coefficient $k_\eps$, and
\begin{equation} \label{eq:constant-equation}
  -{\rm div}\left(\overline{k} \nabla \overline{u} \right) = f,
\end{equation}
which corresponds to the homogenized equation~\eqref{eq:homogenized-equation} for a tentative value $\overline{k}$ of the, beforehand unknown, coefficient $k^\star$. As said above, the equations are respectively posed in an inner and an outer subdomain that surrounds the former (see Figure~\ref{fig:decompo_left} below). The coupling is performed in an overlapping region, and encodes the fact that the solution $u_\eps$ to~\eqref{eq:diffusion} agrees ``on average'' (the meaning of that term is made precise in~\eqref{eq:def_C}--\eqref{eq:constraint} below) with the solution $\overline{u}$ to~\eqref{eq:constant-equation} within the overlapping region. More specifically, the computational implementation of this coupling is performed using the now classical Arlequin method, a popular approach in computational mechanics, which has been introduced in~\cite{dhia1998multiscale,2005arlequin,cottereau_sto,rateau} and which we recall in Section~\ref{sec:arlequin} below. Suitable boundary conditions are imposed on the outer boundary of the domain. These boundary conditions are typically linear Dirichlet boundary conditions. The solution to the coupled system is consequently computed, using an adequate finite element type discretization. The necessary details are presented in Section~\ref{sec:disc}.

A cost function is then evaluated. It measures to which extent the solution obtained differs from the solution obtained for an entirely homogeneous medium. The tentative value of $\overline{k}$ is updated correspondingly (by minimizing this cost function, see~\eqref{eq:optim_J}--\eqref{eq:def_J} below) and the above process is repeated until consistency is obtained, within the desired degree of accuracy.

\medskip

Placing the computational approach described above on a sound mathematical grounding requires to successively establish the following properties:
\begin{description}
\item[(i)] for a fixed value of $\eps$, there exists an optimized value of $\overline{k}$, denoted by $\overline{k}^{\rm opt}_\eps$, where the cost function attains its minimum.
\item[(ii)] as $\eps\to 0$, the optimal value $\overline{k}^{\rm opt}_\eps$ converges to the homogenized coefficient $k^\star$.
\end{description}
In addition, the uniqueness of the optimal value $\overline{k}^{\rm opt}_\eps$ in~{\bf (i)} may be studied.

\medskip

After presenting the computational approach in Section~\ref{sec:introductionArlequin}, we turn in Section~\ref{sec:scalar} to the analysis of the scalar case. We study each of the two above properties, respectively in Section~\ref{sec:opt} and Section~\ref{sec:homogenization} below, under suitable, somehow classical and relatively mild assumptions. And we indeed establish they both hold true, in Theorems~\ref{th:optimization} and~\ref{th:homogenization} respectively. We next establish, in Section~\ref{sec:uniqueness}, the uniqueness of the optimal coefficient $\overline{k}^{\rm opt}_\eps$ for sufficiently small values of $\eps$ (see Theorem~\ref{th:uniqueness}).

Consequently, the approach does provide, for each size of the fine-scale structure fixed, an optimal effective coefficient. In the limit of a vanishing size of the fine-scale structure, this optimal coefficient is unique and allows one to identify the actual homogenized coefficient of the medium considered.

We next turn in Appendix~\ref{seq:matrix} to the matrix-valued case, where our main results (existence of at least one optimal coefficient $\overline{k}^{\rm opt}_\eps$ and convergence of that coefficient to $k^\star$) are Theorems~\ref{th:optimization_matrix} and~\ref{th:homogenization_matrix}.

\medskip

We conclude this introductory section by mentioning that, despite the fact that the mathematical study presented here is restricted to the two-dimensional periodic setting, we believe that the arguments introduced (most of which are variational in nature) are likely to carry over to a large variety of settings: random stationary coefficients, nonlinear monotone equations, non constant slowly varying homogenized coefficients, etc. However, we have not pursued in those many directions, and definite conclusions are yet to be obtained. 

\section{Presentation of the computational approach}\label{sec:introductionArlequin}

The purpose of this section is to present in full details our mathematical framework. We start with a coupling strategy of an oscillating model with an effective one and briefly recall the basics of the Arlequin approach in Section~\ref{sec:arlequin}. We next turn to its discretized finite element formulation in Section~\ref{sec:disc}.

We assume that the sequence of oscillatory functions $k_\eps$ is uniformly coercive and bounded (see~\eqref{eq:boundedness+coercivity}). For technical reasons in the mathematical arguments below, we also assume that
\begin{equation}\label{eq:holder_continuity}
  \text{$k_{\rm per}$ is H\"{o}lder continuous}.
\end{equation}
We use this assumption to have some regularity on the correctors and on some related quantities, see the discussion below~\eqref{eq:alpha}. 

\subsection{Mathematical setting and formal description of the coupling method at the continuous level} \label{sec:arlequin}

Throughout this article, we assume that the computational domain is the two-dimensional square $\Omega = (-L,L)^2$ for some $L>0$ (see Figure~\ref{fig:decompo_left}). 
 
To begin with, we choose $L_c$ and $L_f$ such that $0 < L_f < L_c < L$ and introduce three disjoint subdomains $D$, $D_c$ and $D_f$ of the computational domain $\Omega$ such that $\overline{\Omega} = \overline{D} \cup \overline{D_c} \cup \overline{D_f}$. The inner subdomain $D_f = (-L_f,L_f)^2$ explicitly accounts for the fine-scale structure (modelled by an oscillatory coefficient $k_\eps$). This first subdomain is surrounded by a second subdomain $D_c = (-L_c,L_c)^2 \setminus [-L_f,L_f]^2$, where both models are simultaneously considered: the fine-scale structure, and the effective medium (modelled by a coefficient $\overline{k}$ chosen to be constant, since the homogenized coefficient $k^\star$ is constant in view of~\eqref{eq:structure-k}). In that second subdomain, the two models are coupled so that, in a sense made precise below, they are consistent with one another. The second subdomain is surrounded by a third subdomain $D$, where only the effective medium is considered. We denote $\Gamma$ the exterior boundary of $D$ (see Figure~\ref{fig:decompo_left}), on which we impose (non-homogeneous) Dirichlet boundary conditions, see~\eqref{eq:pb_min}. We also introduce the boundaries $\Gamma_f = \partial D_f$ and $\Gamma_c = \partial(\overline{D_c} \cup \overline{D_f})$.

We note that, for the sake of simplicity, we work with square domains $D$, $D_c$ and $D_f$. However, one could consider more general cases with polygonal domains. In the same spirit, $D_f$ does not need to lie exactly in the center of $D_c$, and $D_c$ does not need to lie exactly in the center of $D$, nor be exactly equally thick on each side of $D_f$, \dots

In what follows, we assume that the boundary $\Gamma_f$ (resp. $\Gamma_c$) is the union of $N_f$ straight edges (resp. $N_c$ straight edges), where the number $N_f$ (resp. $N_c$) of edges is independent of $\eps$:
\begin{equation}\label{eq:boundaries}
  \Gamma_f = \mathop{\bigcup}_{1 \leq j \leq N_f} \widetilde{e}_j, \qquad \qquad
  \Gamma_c = \mathop{\bigcup}_{N_f+1 \leq j \leq N_f+N_c} \widetilde{e}_j.  
\end{equation}
This assumption is of course satisfied for any polygonal domains $D_c$ and $D_f$.

\begin{figure}[htbp]
  \vspace{-3pt}
  \centering
  \includegraphics[width=0.8\textwidth]{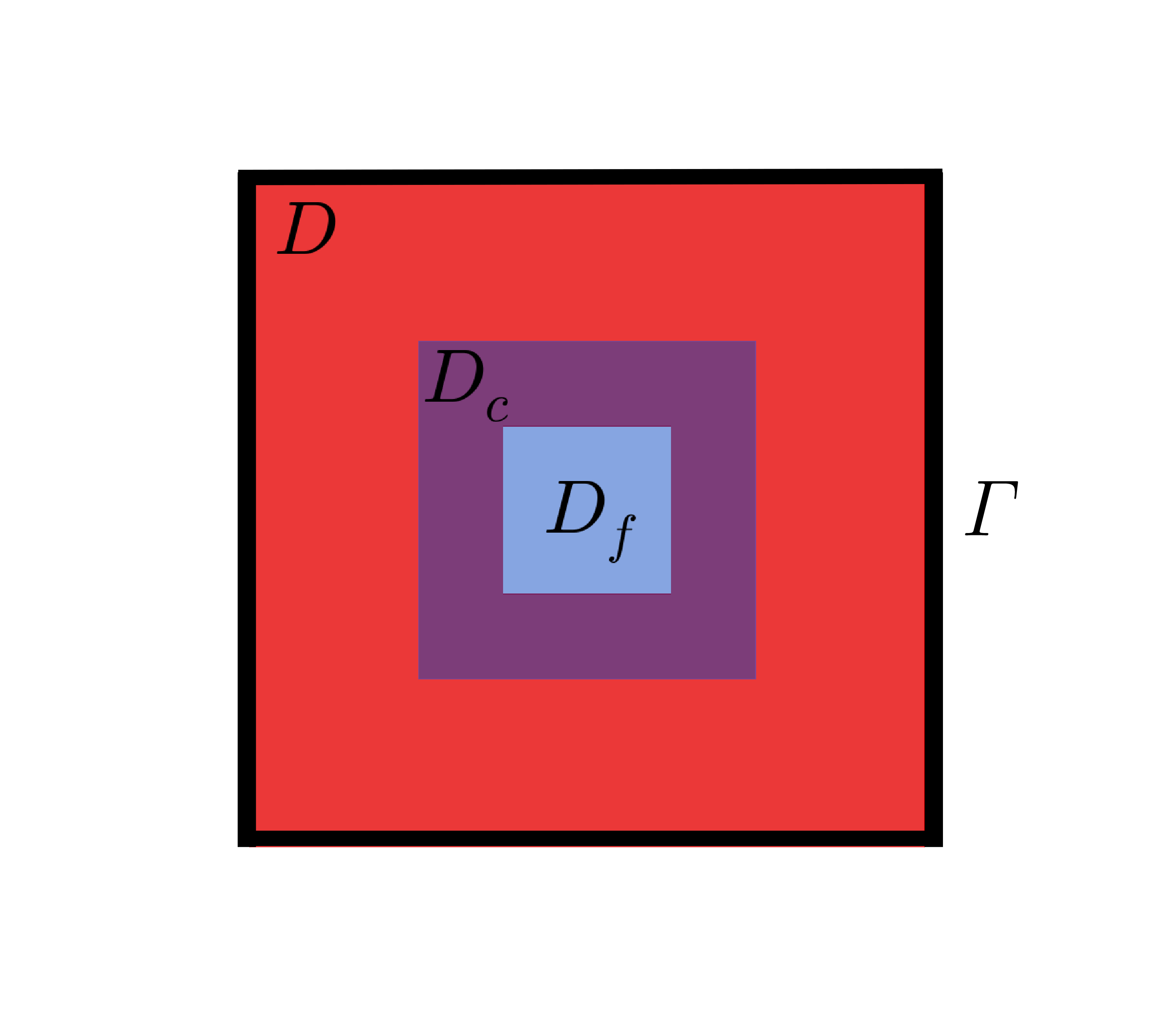}
  \caption{Decomposition of the computational domain into three disjoint subdomains: a subdomain $D$ where only the effective model is defined, a subdomain $D_f$ where only the fine model is defined and a subdomain $D_c$ where both models are defined and over which they are coupled (the subscripts $f$ and $c$ obviously stand for ``fine'' and ``coupled''). Dirichlet boundary conditions are imposed on the exterior boundary of $D$, which is denoted $\Gamma$ and which is represented by black thick lines. \label{fig:decompo_left}}
  \vspace{10pt}
\end{figure}

\bigskip

If it were to be formulated at the continuous level, the Arlequin method applied to~\eqref{eq:diffusion} and~\eqref{eq:constant-equation} would consist in considering the following minimization problem:
\begin{equation} \label{eq:pb_min}
  \inf \left\{ \begin{array}{c} {\cal E}(\overline{u},\widecheck{u}_\eps), \quad \overline{u} \in H^1(D \cup D_c), \quad \overline{u}(x) = x_1 \ \text{on $\Gamma$}, \\ \noalign{\vskip 2pt} \widecheck{u}_\eps \in H^1(D_c \cup D_f), \qquad {\cal C}(\overline{u}-\widecheck{u}_\eps,\phi) = 0 \ \ \text{for any $\phi \in H^1(D_c)$} \end{array} \right\},
\end{equation}
where the constraint function ${\cal C}$ and the energy ${\cal E}$ are defined as follows. The energy ${\cal E}$ is the sum of the contributions of each of the three subdomains:
\begin{multline}\label{eq:def_E}
  {\cal E}(\overline{u},\widecheck{u}_\eps) = \frac{1}{2} \int_D \overline{k} \, \nabla \overline{u}(x) \cdot  \nabla \overline{u}(x) + \frac{1}{2} \int_{D_f} k_\eps(x) \, \nabla \widecheck{u}_\eps(x) \cdot \nabla \widecheck{u}_\eps(x) \\ + \frac{1}{2} \int_{D_c} \Big( \frac{1}{2} \, \overline{k} \, \nabla \overline{u}(x) \cdot \nabla \overline{u}(x) + \frac{1}{2} \, k_\eps(x) \, \nabla \widecheck{u}_\eps(x) \cdot \nabla \widecheck{u}_\eps(x) \Big).
\end{multline}
The last term in ${\cal E}$ accounts for the energy in the domain $D_c$, where the two models co-exist and are equally weighted (thus the factor $1/2$ in the integrand). Other choices of weights are possible, as discussed in~\cite{ref_olga_comp}. We will not consider them hereafter.

In~\eqref{eq:pb_min}, the constraint function ${\cal C}$ is defined by 
\begin{equation} \label{eq:def_C}
  \forall u \in H^1(D_c), \quad \forall \phi \in H^1(D_c), \quad {\cal C}(u,\phi) = \int_{D_c} \nabla u \cdot \nabla \phi + u \, \phi.
\end{equation}
We hence see that the constraint amounts to $\overline{u} = \widecheck{u}_\eps$ on $D_c$ (and that the precise expression of ${\cal C}$ does not matter).

However, the Arlequin approach (in this context, as well as in more general contexts) is to be put in action at the discretized level, as we will see in Section~\ref{sec:disc}. In the latter context, the constraint in~\eqref{eq:pb_min} is transformed in a milder constraint that only imposes that $\overline{u}$ and $\widecheck{u}_\eps$ agree on average (see~\eqref{eq:constraint} below). 

\begin{remark}
  Alternative choices for the boundary conditions on $\Gamma$ in~\eqref{eq:pb_min} could be made (see Remark~\ref{rmk:reference} below for instance).
\end{remark}

\begin{remark}
  In~\eqref{eq:pb_min} and throughout this article, the notation $H^1(D \cup D_c)$ actually stands for the space $H^1(D \cup D_c \cup \Gamma_c)$ (where we recall that $\Gamma_c = \partial(\overline{D_c} \cup \overline{D_f})$), so that the trace on $\Gamma_c$ of a function of that space has the same value on both sides of $\Gamma_c$. Likewise, the notation $H^1(D_c \cup D_f)$ stands for the space $H^1(D_c \cup \overline{D_f}) = H^1(D_c \cup D_f \cup \Gamma_f)$ (where we recall that $\Gamma_f = \partial D_f$).
\end{remark}

It is easy to show (upon considering minimizing sequences and using the strong convexity of ${\cal E}$) that, for any positive definite symmetric matrix $\overline{k}$, problem~\eqref{eq:pb_min} has a unique minimizer.

Solving the minimization problem~\eqref{eq:pb_min} is equivalent to solving the following variational formulation: find $\overline{u} \in H^1(D \cup D_c)$ with $\overline{u}(x) = x_1$ on $\Gamma$, $\widecheck{u}_\eps \in H^1(D_c \cup D_f)$ and $\psi \in H^1(D_c)$ such that
\begin{equation} \label{eq:arlequin_var}
  \begin{cases}
  \forall \overline{v} \in V^0, \quad & \overline{A}_{\overline{k}}(\overline{u},\overline{v}) + {\cal C}(\overline{v},\psi) = 0,
  \\ \noalign{\vskip 3pt}
  \forall \widecheck{v} \in H^1(D_c \cup D_f), \quad & \widecheck{A}_{k_\eps}(\widecheck{u}_\eps,\widecheck{v}) - {\cal C}(\widecheck{v},\psi) = 0,
  \\ \noalign{\vskip 3pt}
  \forall \phi \in H^1(D_c), \quad & {\cal C}(\overline{u}-\widecheck{u}_\eps,\phi) = 0,
  \end{cases}
\end{equation}
where
$$
V^0 = \left\{ v \in H^1(D \cup D_c), \quad \text{$v(x)=0$ on $\Gamma$} \right\},
$$
and where the bilinear forms $\overline{A}_{\overline{k}}$ and $\widecheck{A}_{k_\eps}$ are respectively defined by
\begin{align}
  \label{eq:def_Abar}
  \overline{A}_{\overline{k}}(\overline{u},\overline{v}) &= \int_D \overline{k} \, \nabla \overline{u}(x) \cdot \nabla \overline{v}(x) + \frac{1}{2} \int_{D_c} \overline{k} \, \nabla \overline{u}(x) \cdot \nabla \overline{v}(x),
  \\
  \label{eq:def_Aeps}
  \widecheck{A}_{k_\eps}(\widecheck{u},\widecheck{v}) &= \frac{1}{2}\int_{D_c} k_\eps(x) \, \nabla \widecheck{u}(x) \cdot \nabla \widecheck{v}(x) + \int_{D_f} k_\eps(x) \, \nabla \widecheck{u}(x) \cdot \nabla \widecheck{v}(x).
\end{align}
Of course, the three components $\overline{u}$, $\widecheck{u}_\eps$ and $\psi$ of the solution to~\eqref{eq:arlequin_var} all depend on $\eps$. To keep the notation light, we have made this dependency explicit only for $\widecheck{u}_\eps$ to recall that this function oscillates at the scale $\eps$ (in contrast to $\overline{u}$ and $\psi$, which are meant to be coarse-scale functions, and that will be discretized on a {\em coarse} mesh, see Section~\ref{sec:disc} below).

\medskip

Although we will not specifically use this form, it is illustrative to write the strong form of the optimality system~\eqref{eq:arlequin_var}:
$$
  \begin{cases}
    -{\rm div} \left( \overline{k} \nabla \overline{u} \right) = 0 & \text{in $D$},
    \\
    -{\rm div} \left( \overline{k} \nabla \overline{u} \right) - \Delta \psi + \psi = 0 & \text{in $D_c$},
    \\
    -{\rm div} \left( k_\eps \nabla \widecheck{u}_\eps \right) + \Delta \psi - \psi  = 0 & \text{in $D_c$},
    \\
    -{\rm div} \left( k_\eps \nabla \widecheck{u}_\eps \right) = 0 & \text{in $D_f$},
    \\
    \overline{u} = \widecheck{u}_\eps & \text{in $D_c$},
  \end{cases}
$$  
with the boundary conditions
$$
  \begin{cases}
    \dps \left(\frac{1}{2} \overline{k} \left( \nabla \overline{u} \right) \big|_{D_c} + \left( \nabla \psi  \right) \big|_{D_c} \right) \cdot n_{\Gamma_c} = \left( \overline{k} \left( \nabla \overline{u} \right) \big|_D \right) \cdot n_{\Gamma_c} & \text{on $\Gamma_c$},
    \\ \noalign{\vskip 3pt}
    \dps \left(\frac{1}{2} k_\eps \left( \nabla \widecheck{u}_\eps \right) \big|_{D_c} - \left( \nabla \psi \right) \big|_{D_c} \right) \cdot n_{\Gamma_c} = 0 & \text{on $\Gamma_c$},
    \\ \noalign{\vskip 3pt}
    \dps \left(\frac{1}{2} \overline{k} \left( \nabla \overline{u} \right) \big|_{D_c} + \left( \nabla \psi  \right) \big|_{D_c} \right) \cdot n_{\Gamma_f} = 0 & \text{on $\Gamma_f$},
    \\ \noalign{\vskip 3pt}
    \dps \left(\frac{1}{2} k_\eps \left( \nabla \widecheck{u}_\eps \right) \big|_{D_c} - \left( \nabla \psi \right) \big|_{D_c} \right) \cdot n_{\Gamma_f} = \left( k_\eps \left( \nabla \widecheck{u}_\eps \right) \big|_{D_f} \right) \cdot n_{\Gamma_f} & \text{on $\Gamma_f$},
    \\
    \overline{u}(x) = x_1 & \text{on $\Gamma$}.
  \end{cases}
$$
Here we have denoted by $n_{\Gamma_c}$ (resp. $n_{\Gamma_f}$) the unit normal vector outward to $D_c$ on the boundary $\Gamma_c$ (resp. $\Gamma_f$). We have also denoted by $\left( \nabla \widecheck{u}_\eps \right) \big|_{D_f} \cdot n_{\Gamma_f}$ the normal trace on the boundary $\Gamma_f$ of $\nabla \widecheck{u}_\eps$ seen as a function defined in the domain $D_f$.

\medskip

We now turn to the discretization of the above problem. We will work throughout this article with the discretized form, which is the practically relevant version of the problem.
We emphasize that, in the absence of any discretization, the approach of~\cite{cottereau} {\em does not} yield the value of the homogenized coefficient, as shown in~\cite[Section~2.2]{ref_olga_comp}. If we couple the heterogeneous and the homogeneous models as in~\eqref{eq:pb_min} and next optimize upon $\overline{k}$ as explained in Section~\ref{sec:opt_heuristics} below, we indeed obtain a value which is {\em different} from the homogenized coefficient $k^\star$ we seek, even after passing to the limit $\eps \to 0$.

\subsection{Discretization} \label{sec:disc}

We introduce a coarse mesh ${\cal T}_H$ (of mesh size $H>0$) in the subdomains $D$ and $D_c$ and a fine mesh ${\cal T}_h$ (of mesh size $h>0$) in the subdomains $D_c$ and $D_f$ (see Figure~\ref{fig:decompo_right}). We assume that the coarse meshes of $D$ and $D_c$ are consistent with one another on $\Gamma_c$, namely that they match on the interface (and likewise for the fine meshes of $D_c$ and $D_f$ on the interface $\Gamma_f$). We also assume that, in $D_c$, the fine mesh is a submesh of the coarse mesh. We next introduce the corresponding finite element spaces:
\begin{gather*}
  V_H = \left\{ u^H \in H^1(D \cup D_c), \quad \text{$u^H$ is piecewise affine on the coarse mesh ${\cal T}_H$} \right\},
  \\
  V_h = \left\{ u^h \in H^1(D_c \cup D_f), \quad \text{$u^h$ is piecewise affine on the fine mesh ${\cal T}_h$} \right\},
  \\
  W_H = \left\{ \phi^H \in H^1(D_c), \quad \text{$\phi^H$ is piecewise affine on the coarse mesh ${\cal T}_H$} \right\}.
\end{gather*}
For simplicity of the exposition, we work with $\mathbb{P}_1$ finite elements. Other choices of finite element spaces are however possible \emph{mutatis mutandis} for our arguments below.

The fine mesh size $h$ is assumed to be adjusted so that the discretization in $h$ accurately captures the oscillations of $k_\eps$ (a typical choice is $h \approx \eps/10$). The corresponding discrete problem is therefore expensive to solve. In contrast, the coarse mesh size $H$ can be chosen independent of $\eps$, and therefore satisfies $H \gg h$. The corresponding cost of the macro problem can thus be neglected.

\begin{figure}[htbp]
  \vspace{-1pt}
  \centering
  \includegraphics[width=0.8\textwidth]{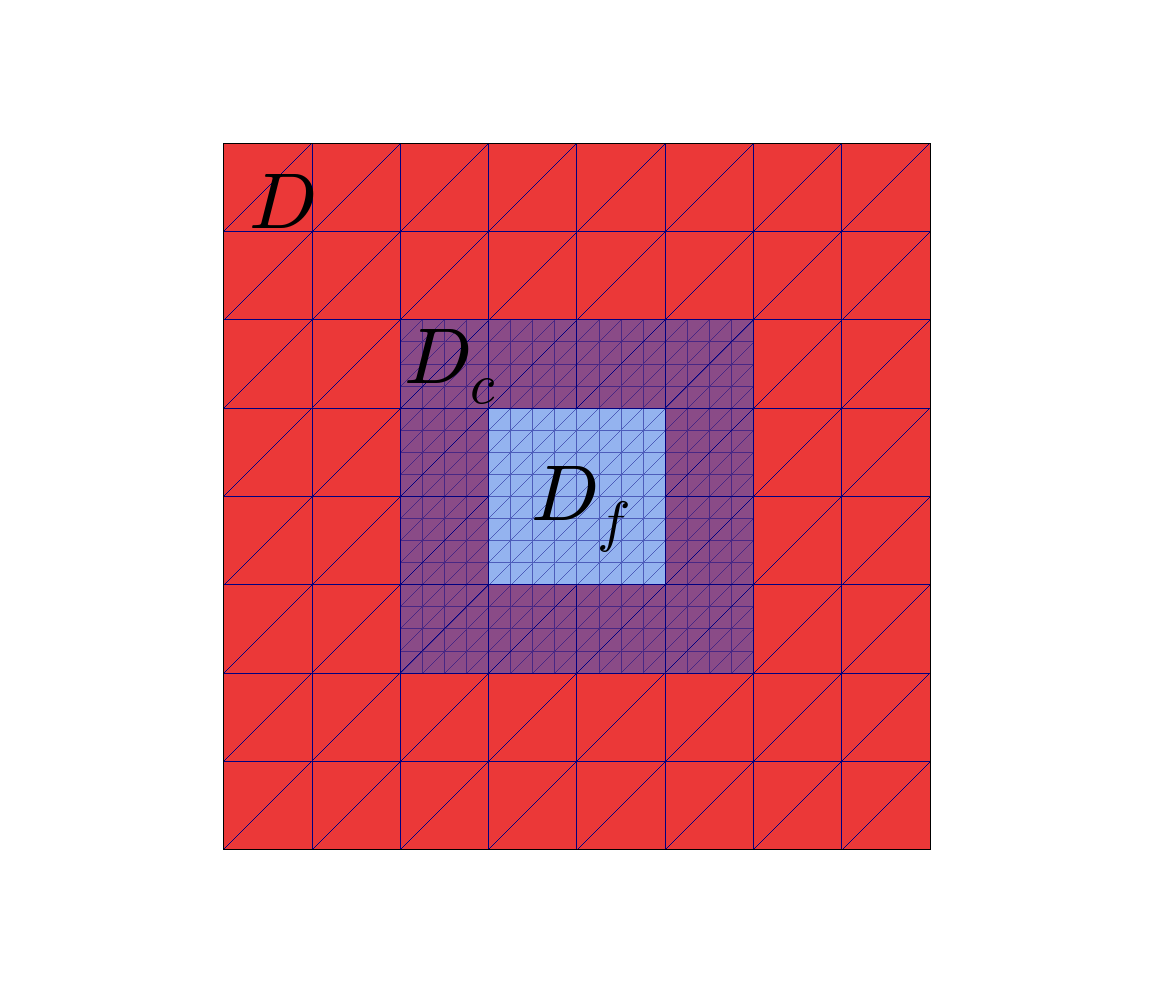}
  \vspace{-5pt}
  \caption{A coarse (resp. fine) mesh is used in $D \cup D_c$ (resp. $D_c \cup D_f$). \label{fig:decompo_right}}
  \vspace{10pt}
\end{figure}

The minimization problem~\eqref{eq:pb_min} is then approximated by
\begin{equation} \label{eq:pb_min_H}
  \inf \left\{ \begin{array}{c} {\cal E}(\overline{u}^H,\widecheck{u}^h_\eps), \quad \overline{u}^H \in V_H, \quad \overline{u}^H(x) = x_1 \ \text{on $\Gamma$}, \\ \noalign{\vskip 2pt} \widecheck{u}^h_\eps\in V_h, \qquad {\cal C}(\overline{u}^H-\widecheck{u}^h_\eps,\phi^H) = 0 \ \ \text{for any $\phi^H \in W_H$} \end{array} \right\},
\end{equation}
where the energy ${\cal E}$ and the constraint function ${\cal C}$ are defined by~\eqref{eq:def_E} and~\eqref{eq:def_C}. Similarly to~\eqref{eq:pb_min}, problem~\eqref{eq:pb_min_H} has a unique minimizer.

In sharp contrast to our observation on~\eqref{eq:pb_min} above, we now observe that the constraint 
\begin{equation}\label{eq:constraint}
  {\cal C}(\overline{u}^H-\widecheck{u}^h_\eps,\phi^H) = 0 \quad \text{for any $\phi^H \in W_H$}
\end{equation}
encodes that, on $D_c$, $\overline{u}^H$ is the projection (in the sense of the scalar product of $H^1(D_c)$, in view of the expression of ${\cal C}$) of $\widecheck{u}^h_\eps$ (itself a piecewise affine function on the fine mesh ${\cal T}_h$) on piecewise affine functions on the coarse mesh ${\cal T}_H$. In that sense, on $D_c$, $\overline{u}^H$ and $\widecheck{u}^h_\eps$ agree with one another {\em on average}. 

\begin{remark} \label{rem:harmful}
  Of course, if the constraint in~\eqref{eq:pb_min_H} is enforced for any $\phi^h \in W_h$, namely any piecewise affine function on the {\em fine} mesh ${\cal T}_h$, then the constraint implies $\overline{u}^H = \widecheck{u}^h_\eps$ in $D_c$. We recover a strong (and harmful) constraint as if we were to consider the continuous Arlequin problem~\eqref{eq:pb_min}. 
\end{remark}

Solving the minimization problem~\eqref{eq:pb_min_H} is equivalent to solving the following variational formulation: find $\overline{u}^H \in V_H^{\rm Dir BC}$, $\widecheck{u}^h_\eps \in V_h$ and $\psi^H \in W_H$ such that
\begin{equation} \label{eq:arlequin0}
  \begin{cases}
    \forall \overline{v}^H \in V_H^0, \quad & \overline{A}_{\overline{k}}(\overline{u}^H,\overline{v}^H) + {\cal C}(\overline{v}^H,\psi^H) = 0,
    \\ \noalign{\vskip 3pt}
    \forall \widecheck{v}^h \in V_h, \quad & \widecheck{A}_{k_\eps}(\widecheck{u}^h_\eps,\widecheck{v}^h) - {\cal C}(\widecheck{v}^h,\psi^H) = 0,
    \\ \noalign{\vskip 3pt}
    \forall \phi^H \in W_H, \quad & {\cal C}(\overline{u}^H-\widecheck{u}^h_\eps,\phi^H) = 0,
  \end{cases}
\end{equation}
where the bilinear forms $\overline{A}_{\overline{k}}$ and $\widecheck{A}_{k_\eps}$ are respectively defined by~\eqref{eq:def_Abar} and~\eqref{eq:def_Aeps} and where
\begin{gather*}
  V_H^{\rm Dir BC} = \left\{ v \in V_H, \quad \text{$v(x)=x_1$ on $\Gamma$} \right\},
  \\
  V_H^0 = \left\{ v \in V_H, \quad \text{$v=0$ on $\Gamma$} \right\}.
\end{gather*}

In general, the solution $(\overline{u}^H,\widecheck{u}^h_\eps,\psi^H)$ to~\eqref{eq:arlequin0} is not analytically known. In the limit $h \to 0$, $\eps \to 0$ and $H \to 0$, however, and if we temporarily assume that $k^\star$ is a scalar and $\overline{k} = k^\star$, then the Lagrange multiplier may be explicitly determined (see~\cite[Section~3.1]{ref_olga_comp}). It is given by $\psi = k^\star \, \psi_0$ where $\psi_0$ is the solution to
\begin{equation} \label{eq:psi_0}
\begin{cases}
  - \Delta \psi_0 + \psi_0 = 0 \quad \text{in $D_c$},
  \\ \noalign{\vskip 3pt}
  \dps \nabla \psi_0 \cdot n_{\Gamma_c} = \frac{1}{2} \, e_1 \cdot n_{\Gamma_c} \ \ \text{on $\Gamma_c$}, \qquad \nabla \psi_0 \cdot n_{\Gamma_f} = -\frac{1}{2} \, e_1 \cdot n_{\Gamma_f} \ \ \text{on $\Gamma_f$}.
\end{cases}
\end{equation}
Following~\cite{ref_olga_comp} and our detailed comments therein regarding consistency of the computational approach, we enrich the classical finite element space $W_H$ and consider
\begin{equation} \label{eq:def_W_H_enrich}
W_H^{\rm enrich} = W_H + \text{Span $\psi_0$}
\end{equation}
instead of $W_H$. In practice, $\psi_0$ is of course not analytically known for general domains and we therefore use an accurate finite element approximation of $\psi_0$. More precisely, the solution $\psi_0 \in H^1(D_c)$ to~\eqref{eq:psi_0} satisfies the following variational formulation:
\begin{equation} \label{eq:psi_0_FV}
\forall \phi \in H^1(D_c), \qquad {\cal C}(\psi_0,\phi) = \frac{1}{2} \int_{\Gamma_c} (e_1 \cdot n_{\Gamma_c}) \, \phi - \frac{1}{2} \int_{\Gamma_f} (e_1 \cdot n_{\Gamma_f}) \, \phi.
\end{equation}
We introduce the finite element space
$$
W_h = \left\{ \phi^h \in H^1(D_c), \quad \text{$\phi^h$ is piecewise affine on the fine mesh ${\cal T}_h$} \right\},
$$
and define $\psi_0^h \in W_h$ as the solution to the following variational formulation:
\begin{equation} \label{eq:psi_0_FV_h}
\forall \phi^h \in W_h, \qquad {\cal C}(\psi_0^h,\phi^h) = \frac{1}{2} \int_{\Gamma_c} (e_1 \cdot n_{\Gamma_c}) \, \phi^h - \frac{1}{2} \int_{\Gamma_f} (e_1 \cdot n_{\Gamma_f}) \, \phi^h.
\end{equation}
For technical reasons that will be apparent below (see~\eqref{eq:tilde_psi}), it is convenient to manipulate an approximation of $\psi_0$ which is a piecewise affine function on the fine mesh ${\cal T}_h$ used to discretize $\widecheck{u}_\eps$. This motivates the choice of approximating $\psi_0$ in $W_h$, and not in another finite dimensional space. Of course, standard finite element arguments show that $\psi_0^h$ converges in $H^1(D_c)$ to $\psi_0$ when $h \to 0$. Note also that the computation of $\psi_0^h$ has only to be performed once, independently of the number of iterations to solve the minimization problem described in Section~\ref{sec:opt_heuristics} below. The additional cost can thus be neglected. In what follows, we therefore consider the enriched space
$$
W_{H,h}^{\rm enrich} = W_H + \text{Span $\psi_0^h$}
$$
instead of $W_H^{\rm enrich}$.

\medskip

As shown in~\cite[Section~3.1]{ref_olga_comp}, the enrichment of the Lagrange multiplier space turns out to be very beneficial from the computational point of view. From the theoretical point of view, and as mentioned above, the enriched approach is now consistent, in the sense that we are enlarging the discretization space so that the exact solution (at convergence $\eps \to 0$) of the problem belongs to that space (see also Remark~\ref{rem:consistency} below). Moreover, we underline that our analysis below (see in particular Lemma~\ref{lemma:optimization} and Section~\ref{sec:homogenization}) critically uses the fact that we work with the enriched space $W_{H,h}^{\rm enrich}$ (or $W_H^{\rm enrich}$ if we set $h=0$). Our current arguments do not go through if we were to work with $W_H$.

\medskip

Thus, instead of discretizating~\eqref{eq:pb_min} in the form of~\eqref{eq:pb_min_H}, we consider hereafter the following enriched minimization problem:
\begin{equation} \label{eq:pb_min_H_enriched}
  \inf \left\{ \begin{array}{c} {\cal E}(\overline{u}^H,\widecheck{u}^h_\eps), \quad \overline{u}^H \in V_H, \quad \overline{u}^H(x) = x_1 \ \text{on $\Gamma$}, \\ \noalign{\vskip 2pt} \widecheck{u}^h_\eps \in V_h, \qquad {\cal C}(\overline{u}^H-\widecheck{u}^h_\eps,\phi^H) = 0 \ \ \text{for any $\phi^H \in W_{H,h}^{\rm enrich}$} \end{array} \right\},
\end{equation}
where the energy ${\cal E}$ and the constraint function ${\cal C}$ are defined by~\eqref{eq:def_E} and~\eqref{eq:def_C}. The corresponding variational formulation reads as: find $\overline{u}^H \in V_H^{\rm Dir BC}$, $\widecheck{u}^h_\eps \in V_h$ and $\psi^H \in W_{H,h}^{\rm enrich}$ such that
\begin{equation} \label{eq:arlequin}
  \begin{cases}
  \forall \overline{v}^H \in V_H^0, \quad &\overline{A}_{\overline{k}}(\overline{u}^H,\overline{v}^H) + {\cal C}(\overline{v}^H,\psi^H) = 0,
  \\ \noalign{\vskip 3pt}
  \forall \widecheck{v}^h \in V_h, \quad & \widecheck{A}_{k_\eps}(\widecheck{u}^h_\eps,\widecheck{v}^h) - {\cal C}(\widecheck{v}^h,\psi^H) = 0,
  \\ \noalign{\vskip 3pt}
  \forall \phi^H \in W_{H,h}^{\rm enrich}, \quad & {\cal C}(\overline{u}^H-\widecheck{u}^h_\eps,\phi^H) = 0.
  \end{cases}
\end{equation}
We keep denoting the Lagrange multiplier by $\psi^H$ (and not $\psi^{H,h}$) since it is meant to be a coarse-scale function. Similarly to~\eqref{eq:arlequin0}, system~\eqref{eq:arlequin} has a unique solution for any positive definite symmetric matrix $\overline{k}$.

\section{Case of scalar-valued coefficients $k^\star$ and $\overline{k}$}\label{sec:scalar}

In this section, we address the case where the homogenized coefficient $k^\star$ and the tentative coefficient $\overline{k}$ upon which we optimize are scalar quantities. The periodic checkerboard~\cite{le2012}, for instance, falls within this case. The study of the general case (when $k^\star$ is matrix-valued) is postponed until Appendix~\ref{seq:matrix}. We recall that, for the well-posedness of the mathematical problem~\eqref{eq:diffusion} (and similarly of~\eqref{eq:pb_min_H_enriched}), the sequence of oscillatory (matrix-valued) functions $k_\eps$ is assumed to be bounded and bounded away from zero uniformly in $\eps$ (see~\eqref{eq:boundedness+coercivity}).

\medskip

This section is organized as follows. We first recall in Section~\ref{sec:opt_heuristics} the optimization strategy (see~\eqref{eq:optim_J} below) introduced in~\cite{cottereau} and which aims at computing an approximation of the homogenized coefficient $k^\star$ associated to the highly oscillatory coefficient $k_\eps$. Second, in Section~\ref{sec:tech_lemmas}, we state some important technical properties of the coupling system~\eqref{eq:pb_min_H_enriched} that are required for our analysis. Third, we study the optimization problem~\eqref{eq:optim_J}:
\begin{itemize}
\item we first establish the existence of an optimal coefficient $\overline{k}^{\rm opt}_\eps$ for a fixed value of $\eps>0$ (see our first result, Theorem~\ref{th:optimization} in Section~\ref{sec:opt}).
\item we then consider the limit $\eps \to 0$ and show that any optimal coefficient $\overline{k}^{\rm opt}_\eps$ indeed converges to the homogenized coefficient $k^\star$ (see our second result, Theorem~\ref{th:homogenization} in Section~\ref{sec:homogenization}).
\item we eventually show that the minimizer to~\eqref{eq:optim_J} is unique when $\eps$ is sufficiently small (see our third result, Theorem~\ref{th:uniqueness} in Section~\ref{sec:uniqueness}).
\end{itemize}
These results prove, at least in this scalar setting, that the approach is certified theoretically.

\subsection{Optimization upon the coefficient $\overline{k}$} \label{sec:opt_heuristics}

We first recall the optimization strategy introduced in~\cite{cottereau}.
Formally, the heterogeneous coefficient $k_\eps$ in~\eqref{eq:pb_min_H_enriched} can be replaced by its homogenized limit $k^\star$, which is a constant coefficient in view of the periodicity assumption~\eqref{eq:structure-k}. It is then clear that, if $\overline{k} = k^\star$, then the response of the material is linear (i.e. $\overline{u}^H(x) = x_1$ in $D \cup D_c$), because the whole domain is modelled by a constant coefficient and because of the particular boundary conditions considered in~\eqref{eq:pb_min_H_enriched}. As shown in~\cite[Lemma~2.1]{ref_olga_comp} (in the absence of any spatial discretization) and in Lemma~\ref{lemma:optimization} below (when taking into account mesh discretization), the converse is also true: if the response of the material is linear (i.e. if $\overline{u}^H(x) = x_1$ in $D \cup D_c$), then the material is homogeneous and $\overline{k} = k^\star$. This motivates the idea to compare the solution $\overline{u}^H$ to~\eqref{eq:pb_min_H_enriched} with the reference solution $u_{\rm ref}$ defined by $u_{\rm ref}(x) = x_1$. Optimizing upon the coefficient $\overline{k}$ in order to best fit a linear field $u_{\rm ref}(x) = x_1$ is thus a way to enforce that $\overline{k} = k^\star$. 

We may achieve that by considering the minimization problem
\begin{equation} \label{eq:optim_J}
  I_{\eps, H, h} = \inf \left\{ J_{\eps, H, h}(\overline{k}), \quad \overline{k} \in (0,\infty) \right\},
\end{equation}
with
\begin{equation} \label{eq:def_J}
  J_{\eps, H, h}(\overline{k}) = \int_{D \cup D_c} \big| \nabla \overline{u}^H_{\overline{k}, k_\eps} - \nabla u_{\rm ref} \big|^2 = \int_{D \cup D_c} \big| \nabla \overline{u}^H_{\overline{k}, k_\eps} - e_1 \big|^2.
\end{equation}
Here $e_1$ is the first canonical vector and $(\overline{u}^H_{\overline{k}, k_\eps},\widecheck{u}^h_{\eps,\overline{k}, k_\eps}, \psi^H_{\overline{k}, k_\eps})$ is the solution to~\eqref{eq:arlequin} where we have made explicit the dependency of the solution with respect to the tentative coefficient $\overline{k}$ and the oscillating coefficient $k_\eps$. Of course, $\overline{u}^H_{\overline{k}, k_\eps}$ depends on $\eps$ and on the mesh sizes $H$ and $h$ used in~\eqref{eq:arlequin}, thus the subscripts in the notation $J_{\eps, H, h}$.

\begin{remark}\label{rmk:reference}
  The reference solution $u_{\rm ref}$ that we consider in~\eqref{eq:def_J} depends on the boundary conditions imposed in the Arlequin problem. For instance, another possibility is to consider the following minimization problem, analogous to~\eqref{eq:pb_min_H_enriched}:
  \begin{equation} \label{eq:pb_min_H_x2}
    \inf \left\{ \begin{array}{c} {\cal E}(\overline{u}^H,\widecheck{u}^h_\eps), \quad \overline{u}^H \in V_H, \quad \overline{u}^H(x) = x_2 \ \text{on $\Gamma$}, \\ \noalign{\vskip 2pt} \widecheck{u}^h_\eps \in V_h, \qquad {\cal C}(\overline{u}^H-\widecheck{u}^h_\eps,\phi^H) = 0 \ \ \text{for any $\phi^H \in W_{H,h}^{\rm enrich}$} \end{array} \right\},
  \end{equation}
  where we recall that $\Gamma$ is the external boundary of $D$, see Figure~\ref{fig:decompo_left}. Of course, in this case, we should compare the solution $\overline{u}^H$ to~\eqref{eq:pb_min_H_x2} with the reference solution $u_{\rm ref}(x)=x_2$ and thus modify the objective function in~\eqref{eq:optim_J} as
  $$
  J_{\eps, H, h}(\overline{k}) = \int_{D \cup D_c} \big| \nabla \overline{u}^H_{\overline{k}, k_\eps} - e_2 \big|^2.
  $$
  We also note that, for the case of a matrix-valued coefficient $k^\star \in \RR^{2\times2}$, in order to recover all the components of $k^\star$, we {\em have} to consider both problems~\eqref{eq:pb_min_H_enriched} and~\eqref{eq:pb_min_H_x2} (see Remark~\ref{matrix_dirichlet_x_2} below).
\end{remark}

Note that, in~\eqref{eq:optim_J}, we impose neither that $\overline{k}$ is bounded from above nor that $\overline{k}$ is bounded away from zero. Our proof of existence of a minimizer actually shows that we do not need to impose such bounds, provided two conditions (namely~\eqref{eq:J_contradiction} and~\eqref{eq:J_contradiction2} below) are satisfied. These conditions can be checked using minimizing sequences of~\eqref{eq:optim_J}, and we show below that they in particular hold whenever $h$ and $\eps$ are sufficiently small.

It is also possible to enforce on $\overline{k}$ in~\eqref{eq:optim_J} the bounds~\eqref{eq:boundedness+coercivity} satisfied by $k_\eps$ and proceed with the mathematical study (see Remark~\ref{rem:k_bounds} below). A corresponding algorithm can be designed. We do not proceed in this direction.

\subsection{Two useful technical results} \label{sec:tech_lemmas}

We collect here two technical results that are useful for our analysis below.

\subsubsection{Auxiliary homogenization result} \label{sec:tech_lemma_homog}

We temporarily assume that $h = 0$ and consider the following system for some fixed $\overline{k} \in (0,\infty)$: find $\overline{u}^H_{\overline{k}, k_\eps} \in V_H^{\rm Dir BC}$, $\widecheck{u}_{\eps,\overline{k}, k_\eps} \in H^1(D_c \cup D_f)$ and $\psi^H_{\overline{k}, k_\eps} \in W_H^{\rm enrich}$ such that
\begin{equation}\label{eq:system_k_star}
  \begin{cases}
    \forall \overline{v}^H \in V_H^0, \quad & \overline{A}_{\overline{k}}(\overline{u}^H_{\overline{k}, k_\eps},\overline{v}^H) + {\cal C}(\overline{v}^H,\psi^H_{\overline{k}, k_\eps}) = 0,
  \\ \noalign{\vskip 3pt}
  \forall \widecheck{v} \in H^1(D_c \cup D_f), \quad & \widecheck{A}_{k_\eps}(\widecheck{u}_{\eps,\overline{k}, k_\eps},\widecheck{v}) - {\cal C}(\widecheck{v},\psi^H_{\overline{k}, k_\eps}) = 0,
  \\ \noalign{\vskip 3pt}
  \forall \phi^H \in W_H^{\rm enrich}, \quad & {\cal C}(\overline{u}^H_{\overline{k}, k_\eps} - \widecheck{u}_{\eps,\overline{k}, k_\eps},\phi^H) = 0,
  \end{cases}
\end{equation}
where we recall that $W_H^{\rm enrich}$ is defined by~\eqref{eq:def_W_H_enrich}. The system~\eqref{eq:system_k_star} corresponds to~\eqref{eq:arlequin} where we have omitted the fine mesh discretization (and thus formally set $h=0$). On purpose, we have made explicit the dependency of the solution to~\eqref{eq:system_k_star} with respect to the constant coefficient $\overline{k}$ and the oscillatory coefficient $k_\eps$. Our aim is to pass to the limit $\eps \to 0$ in~\eqref{eq:system_k_star}. To that aim, we introduce the bilinear form $\widecheck{A}_{k^\star}$ (compare with~\eqref{eq:def_Aeps}) defined by
$$
\widecheck{A}_{k^\star}(\widecheck{u},\widecheck{v}) = \frac{1}{2} \int_{D_c} k^\star \, \nabla \widecheck{u}(x) \cdot \nabla \widecheck{v}(x) + \int_{D_f} k^\star \, \nabla \widecheck{u}(x) \cdot \nabla \widecheck{v}(x).
$$

\begin{lemma}\label{th:multiplication}
  Let $\overline{k} \in (0,\infty)$ and $k_\eps$ be given by~\eqref{eq:structure-k} for some fixed periodic coefficient $k_{\rm per}$ that satisfies the classical boundedness and coercivity conditions~\eqref{eq:boundedness+coercivity}. We make the regularity assumption~\eqref{eq:holder_continuity} and the geometric assumption~\eqref{eq:boundaries}.

  Consider the solution $(\overline{u}^H_{\overline{k}, k_\eps},\widecheck{u}_{\eps,\overline{k}, k_\eps},\psi^H_{\overline{k}, k_\eps})$ to~\eqref{eq:system_k_star} and assume that there exists $\overline{u}^H_0 \in V_H^{\rm Dir BC}$, $\widecheck{u}_0 \in H^1(D_c \cup D_f)$ and $\psi^H_0 \in W_H^{\rm enrich}$ such that $\overline{u}^H_{\overline{k}, k_\eps} \to \overline{u}^H_0$ in $H^1(D \cup D_c)$, $\widecheck{u}_{\eps,\overline{k}, k_\eps} \rightharpoonup \widecheck{u}_0$ in $H^1(D_c \cup D_f)$ and $\psi^H_{\overline{k}, k_\eps} \to \psi^H_0$ in $H^1(D_c)$ when $\eps \to 0$.

  Then $(\overline{u}^H_0, \widecheck{u}_0, \psi^H_0)$ is actually equal to $(\overline{u}^H_{\overline{k},k^\star}, \widecheck{u}_{\overline{k},k^\star}, \psi^H_{\overline{k},k^\star})$, that is the solution to the following system:
\begin{equation}\label{eq:system_contradiction}
  \begin{cases}
    \forall \overline{v}^H \in V_H^0, \quad & \overline{A}_{\overline{k}}(\overline{u}^H_{\overline{k},k^\star},\overline{v}^H) + {\cal C}(\overline{v}^H,\psi^H_{\overline{k},k^\star}) = 0,
    \\ \noalign{\vskip 3pt}
    \forall \widecheck{v} \in H^1(D_c \cup D_f), \quad & \widecheck{A}_{k^\star}(\widecheck{u}_{\overline{k},k^\star},\widecheck{v}) - {\cal C}(\widecheck{v},\psi^H_{\overline{k},k^\star}) = 0,
    \\ \noalign{\vskip 3pt}
    \forall \phi^H \in W_H^{\rm enrich}, \quad & {\cal C}(\overline{u}^H_{\overline{k},k^\star}-\widecheck{u}_{\overline{k},k^\star},\phi^H) = 0.
  \end{cases}
\end{equation}
\end{lemma}

Note that we have assumed that the Lagrange multiplier $\psi^H_{\overline{k}, k_\eps}$ {\em strongly} converges in $H^1(D_c)$. This is a critical assumption. We indeed face a difficulty if we only assume that it weakly converges in $H^1(D_c)$ (because we use a test function $\widecheck{v}$ which itself depends on $\eps$, see~\eqref{eq:test_function} below, and only weakly converges in $H^1(D_c)$; passing to the limit in the last two terms of~\eqref{eq:arlequin_2nd_line} would then be difficult under the only assumption of weak convergence of $\psi^H_{\overline{k}, k_\eps}$).

\begin{coro} \label{coro:multiplication}
We make the same assumptions as in Lemma~\ref{th:multiplication}, except that we now consider a sequence $\overline{k}_\eps$ of constant coefficients (instead of a constant coefficient $\overline{k}$ independent of $\eps$), with $\overline{k}_\eps \in (0,\infty)$ for any $\eps$ and such that $\overline{k}_\eps$ converges to some $\overline{k}_0 \in [0,\infty)$ when $\eps \to 0$. Then the conclusion of Lemma~\ref{th:multiplication} still holds, with $\overline{k}$ replaced by $\overline{k}_0$ in~\eqref{eq:system_contradiction}. In the case $\overline{k}_0 = 0$, $(\overline{u}^H_{\overline{k}_0,k^\star}, \widecheck{u}_{\overline{k}_0,k^\star}, \psi^H_{\overline{k}_0,k^\star})$ is simply {\em a solution} (rather than {\em the solution}) to~\eqref{eq:system_contradiction}.
\end{coro}

The proof of Corollary~\ref{coro:multiplication} can be performed using exactly the same arguments as in the proof of Lemma~\ref{th:multiplication}. We therefore skip it.

\begin{proof}[Proof of Lemma~\ref{th:multiplication}]
The second line of~\eqref{eq:system_k_star} explicitly reads as
\begin{multline} \label{eq:arlequin_2nd_line}
  \forall \widecheck{v}\in H^1(D_c \cup D_f), \quad \frac{1}{2} \int_{D_c} k_\eps \nabla \widecheck{u}_{\eps,\overline{k}, k_\eps} \cdot \nabla \widecheck{v} + \int_{D_f}  k_\eps \nabla \widecheck{u}_{\eps,\overline{k}, k_\eps} \cdot \nabla \widecheck{v}
  \\
  - \int_{D_c} \nabla \psi^H_{\overline{k}, k_\eps} \cdot \nabla \widecheck{v} - \int_{D_c} \psi^H_{\overline{k}, k_\eps} \widecheck{v} = 0.    
\end{multline}
It is not straightforward to pass to the limit $\eps \to 0$ in the first two terms of~\eqref{eq:arlequin_2nd_line}, because both $k_\eps$ and $\nabla \widecheck{u}_{\eps,\overline{k}, k_\eps}$ only weakly converge. This is a classical homogenization issue (we refer e.g. to~\cite{BeLiP:78,CioDo:99,jikov2012homogenization,Tarta:10} and~\cite[Chapter~1]{Allai:02} for classical textbooks on homogenization). To address this difficulty, we are going to use the oscillating test function method. To simplify the notation, we temporarily denote $\widecheck{u}_{\eps,\overline{k}, k_\eps}$ by $\widecheck{u}_\eps$. We start by choosing an appropriate test function $\widecheck{v}$ in~\eqref{eq:arlequin_2nd_line}, namely 
\begin{equation}\label{eq:test_function}
  \widecheck{v}(x) = \varphi(x) + \eps \sum_{i=1}^2 \partial_i \varphi(x) \, w_i(x/\eps),
\end{equation}
where $\varphi \in C^\infty(\overline{D_c \cup D_f})$ is arbitrary and $w_i$ is the periodic corrector function associated to the $i$-th coordinate vector $e_i$, namely the solution (unique up to the addition of a constant) to  
\begin{equation}\label{eq:corrector}
  -{\rm div} \big( k_{\rm per} (e_i+\nabla w_i) \big) = 0 \quad \text{in $\RR^2$}, \qquad \text{$w_i$ is periodic}.
\end{equation}
With this particular test function~\eqref{eq:test_function} and using integrations by parts, the first term in~\eqref{eq:arlequin_2nd_line} writes
\begin{align}
  \nonumber
  \int_{D_c} k_\eps \nabla \widecheck{u}_\eps \cdot \nabla \widecheck{v}
  &=
  \int_{D_c} k_\eps \nabla \widecheck{u}_\eps \cdot \nabla \varphi + \sum_{i = 1}^2 \int_{D_c} \partial_i\varphi \, k_\eps \, \nabla \widecheck{u}_\eps \cdot \nabla w_i(\cdot/\eps)
  \\
  \nonumber 
  & \quad
  + \eps \sum_{i = 1}^2 \int_{D_c} w_i(\cdot/\eps) \, k_\eps \nabla \widecheck{u}_\eps \cdot \partial_i\nabla\varphi
  \\
  \nonumber
  &=
  -\sum_{i = 1}^2 \int_{D_c} \widecheck{u}_\eps \, \partial_i\varphi\,{\rm div}\left[k_\eps \left(e_i+\nabla w_i(\cdot/\eps)\right)\right]
  \\
  \nonumber 
  & \quad
  -\sum_{i = 1}^2 \int_{D_c} \widecheck{u}_\eps \, k_\eps \left(e_i+\nabla w_i(\cdot/\eps)\right)\cdot \partial_i\nabla\varphi
  \\
  \nonumber
  & \quad
  + \sum_{i = 1}^2 \int_{\Gamma_f} \widecheck{u}_\eps \, \partial_i\varphi \, \left[k_\eps \left(e_i+\nabla w_i(\cdot/\eps)\right)\right] \cdot n_{\Gamma_f}
  \\
  \nonumber 
  &\quad
  + \sum_{i = 1}^2 \int_{\Gamma_c} \widecheck{u}_\eps \, \partial_i\varphi \, \left[k_\eps \left(e_i+\nabla w_i(\cdot/\eps)\right)\right] \cdot n_{\Gamma_c}
  \\
  \label{eq:limit_eps} 
  & \quad
  + \eps \sum_{i = 1}^2 \int_{D_c} w_i(\cdot/\eps) \, k_\eps \nabla \widecheck{u}_\eps \cdot\partial_i\nabla\varphi.
\end{align}
We now successively pass to the limit in each term of the right-hand side of~\eqref{eq:limit_eps}. The first term vanishes because of~\eqref{eq:structure-k} and the corrector equation~\eqref{eq:corrector}. The limit of the second term is identified using the fact that the weak convergence of $\widecheck{u}_\eps$ in $H^1(D_c \cup D_f)$ implies, up to the extraction of a subsequence (that we do not make explicit in the notation), its strong convergence in $L^2(D_c \cup D_f)$. We thus obtain
\begin{equation}\label{eq:second_term}
  \lim_{\eps\to 0} \sum_{i = 1}^2 \int_{D_c} \widecheck{u}_\eps \, k_\eps \left(e_i+\nabla w_i(\cdot/\eps)\right) \cdot \partial_i\nabla\varphi = \sum_{i = 1}^2 \int_{D_c} \widecheck{u}_0 \, k^\star e_i \cdot \partial_i\nabla\varphi,
\end{equation}
where the homogenized coefficient $k^\star$ reads as
$$
k^\star \, e_i = \int_{(0,1)^2} k_{\rm per} \left(e_i+\nabla w_i\right), \qquad i = 1,2.
$$
The last term of the right-hand side of~\eqref{eq:limit_eps} converges to zero, using the Cauchy-Schwarz inequality and the fact that $\nabla \widecheck{u}_\eps$ is bounded in $L^2(D_c)$, $k_\eps$ is bounded in $L^\infty(D_c)$ in view of assumption~\eqref{eq:boundedness+coercivity} and $w_i \in L^2((0,1)^2)$ is periodic.

\medskip

There now remains to show that we can pass to the limit in the two boundary integrals of the right-hand side of~\eqref{eq:limit_eps}. For this purpose, we introduce the two vector fields
$$
G_i = k_{\rm per}\left(e_i+\nabla w_i\right) - k^\star e_i, \qquad i = 1,2,
$$
which are periodic, divergence-free and of zero mean. In the two-dimensional setting, there hence exist (see e.g.~\cite[p.~6]{jikov2012homogenization}) two periodic functions $\alpha_i$ of zero mean such that 
\begin{equation}\label{eq:alpha}
  G_i = \left( \begin{array}{c} \partial_{x_2} \alpha_i \\ -\partial_{x_1} \alpha_i \end{array} \right), \qquad i = 1,2.
\end{equation}
In view of the H\"{o}lder continuity~\eqref{eq:holder_continuity} we have assumed on the coefficient $k_{\rm per}$, we know from~\eqref{eq:corrector} and elliptic regularity theory that $\nabla w_i$ is H\"{o}lder continuous (see~\cite{gilbarg2015elliptic}), thus $G_i$ is also H\"{o}lder continuous and we have that $\alpha_i\in C^1(\RR^2)$, a property that we will use in~\eqref{eq:regul_alpha} below.

Under the assumption~\eqref{eq:boundaries} that the boundaries $\Gamma_f$ and $\Gamma_c$ are unions of straight edges, namely $\dps \Gamma_f = \mathop{\bigcup}_{1\leq j\leq N_f} \widetilde{e}_j$ and $\dps \Gamma_c = \mathop{\bigcup}_{N_f+1\leq j\leq N_f +N_c} \widetilde{e}_j$, we may consider, without loss of generality, the third and fourth terms of the right-hand side of~\eqref{eq:limit_eps} as a boundary integral on an edge $\widetilde{e} \in \{ \widetilde{e}_j \}_{1\leq j\leq N_f+N_c}$. We then write 
\begin{align}
  \nonumber
  \int_{\widetilde{e}} \widecheck{u}_\eps \, \partial_i\varphi \, \left[ k_\eps \left( e_i+\nabla w_i(\cdot/\eps) \right) \right] \cdot n_{\widetilde{e}}
  &=
  \int_{\widetilde{e}} (\widecheck{u}_\eps-\widecheck{u}_0) \, \partial_i\varphi \, \left[ k_\eps \left( e_i+\nabla w_i(\cdot/\eps) \right) \right] \cdot n_{\widetilde{e}}
  \\
  \label{eq:star_boundary}
  &+
  \int_{\widetilde{e}} \widecheck{u}_0 \, \partial_i\varphi \, \left[ k_\eps \left( e_i+\nabla w_i(\cdot/\eps) \right) \right] \cdot n_{\widetilde{e}},
\end{align}
where $n_{\widetilde{e}}$ is a unit vector orthogonal to the edge $\widetilde{e}$. We now use the fact that the trace operator is linear and continuous from $H^1(D_c \cup D_f)$ to $H^{1/2}(\widetilde{e})$ and that the injection $H^{1/2}(\widetilde{e})\subset L^2(\widetilde{e})$ is compact to obtain that $\widecheck{u}_\eps$ strongly converges to $\widecheck{u}_0$ in $L^2(\widetilde{e})$. Since $k_{\rm per} \left( e_i+\nabla w_i \right)$ is continuous and bounded in $\RR^2$, we obtain that the first term of the right-hand side of~\eqref{eq:star_boundary} goes to 0 when $\eps\to 0$.

Using~\eqref{eq:alpha}, we recast the second term of the right-hand side of~\eqref{eq:star_boundary} as
\begin{align}
  & \int_{\widetilde{e}} \widecheck{u}_0 \, \partial_i\varphi \, \left[ k_\eps \left( e_i+\nabla w_i(\cdot/\eps) \right) \right] \cdot n_{\widetilde{e}}
  \nonumber
  \\
  &=
  \int_{\widetilde{e}} \widecheck{u}_0 \, \partial_i \varphi \, k^\star \, e_i \cdot n_{\widetilde{e}} + \int_{\widetilde{e}} \widecheck{u}_0 \, \partial_i\varphi \left( \begin{array}{c} \partial_{x_2} \alpha_i(\cdot/\eps) \\ -\partial_{x_1}\alpha_i(\cdot/\eps) \end{array} \right) \cdot n_{\widetilde{e}}
  \nonumber
  \\
  &=
  \int_{\widetilde{e}} \widecheck{u}_0 \, \partial_i \varphi \, k^\star \, e_i \cdot n_{\widetilde{e}} + \int_{\widetilde{e}} \widecheck{u}_0 \, \partial_i\varphi \, \partial_{\tau_{\widetilde{e}}} \alpha_i(\cdot/\eps),
  \label{eq:eps_boundary}
\end{align}
where $\partial_{\tau_{\widetilde{e}}} \alpha_i$ is the tangential derivative (in the direction of the edge $\widetilde{e}$) of the function $\alpha_i$. 

We now claim that the last term in the right-hand side of~\eqref{eq:eps_boundary} goes to 0 when $\eps \to 0$, that is
\begin{equation}\label{eq:lim_eps_boundary}
  \lim_{\eps \to 0} \int_{\widetilde{e}} \widecheck{u}_0 \, \partial_i\varphi \, \partial_{\tau_{\widetilde{e}}} \alpha_i(\cdot/\eps) = 0.
\end{equation}
We prove this result using an interpolation argument similar to the one used in~\cite[Proof of Lemma~4.6]{lozinski_chinois}. Suppose momentarily that $\widecheck{u}_0 \in H^1(\widetilde{e})$. Using an integration by parts, we have
$$
\int_{\widetilde{e}} \widecheck{u}_0 \, \partial_i\varphi \, \partial_{\tau_{\widetilde{e}}} \alpha_i(\cdot/\eps)
=
\eps \left[ \widecheck{u}_0 \, \partial_i \varphi \, \alpha_i(\cdot/\eps) \right] \big|_{\partial{\widetilde{e}}} - \eps \int_{\widetilde{e}} \alpha_i(\cdot/\eps) \, \partial_{\tau_{\widetilde{e}}}\left( \widecheck{u}_0 \, \partial_i\varphi \right),
$$
and therefore, using that $\alpha_i \in C^1(\RR^2)$ is a periodic function (and thus bounded) and that the injection $H^1(\widetilde{e}) \subset C^0\big(\overline{\widetilde{e}}\big)$ is continuous, we obtain
\begin{align}
  \nonumber
  & \left| \int_{\widetilde{e}} \widecheck{u}_0 \, \partial_i\varphi \, \partial_{\tau_{\widetilde{e}}} \alpha_i(\cdot/\eps) \right|
  \\
  \nonumber
  &\leq
  \eps \, \| \alpha_i \|_{C^0(\RR^2)} \big( 2 \, \| \widecheck{u}_0 \|_{C^0(\overline{\widetilde{e}})} \| \partial_i\varphi\|_{C^0(\overline{\widetilde{e}})} + 2 \, \| \widecheck{u}_0 \|_{H^1(\widetilde{e})} \| \partial_i\varphi\|_{H^1(\widetilde{e})} \big)
  \\
  \label{eq:interp_1}
  &\leq
  C \, \eps \, \|\alpha_i\|_{C^0(\RR^2)} \| \widecheck{u}_0\|_{H^1(\widetilde{e})} \|\partial_i\varphi\|_{H^1(\widetilde{e})},
\end{align}
for some constant $C$ independent of $\eps$. On the other hand, using simply that $\widecheck{u}_0 \in L^2(\widetilde{e})$, we have 
\begin{equation}\label{eq:interp_1.2}
  \left| \int_{\widetilde{e}} \widecheck{u}_0 \, \partial_i\varphi \, \partial_{\tau_{\widetilde{e}}} \alpha_i(\cdot/\eps) \right| \leq \| \nabla\alpha_i \|_{C^0(\RR^2)} \| \widecheck{u}_0 \|_{L^2(\widetilde{e})} \| \partial_i\varphi \|_{L^2(\widetilde{e})}.
\end{equation}
In the statement of Lemma~\ref{th:multiplication}, we have assumed that $\widecheck{u}_0 \in H^1(D_c \cup D_f)$, which implies that $\widecheck{u}_0 \in H^{1/2}(\widetilde{e})$. By interpolation between~\eqref{eq:interp_1} and~\eqref{eq:interp_1.2}, we thus obtain
\begin{equation}\label{eq:regul_alpha}
\left| \int_{\widetilde{e}} \widecheck{u}_0 \, \partial_i\varphi \, \partial_{\tau_{\widetilde{e}}}\alpha_i(\cdot/\eps) \right| \leq C \, \eps^{1/2} \, \| \alpha_i \|_{C^1(\RR^2)} \| \widecheck{u}_0 \|_{H^{1/2}(\widetilde{e})} \| \partial_i\varphi \|_{H^{1/2}(\widetilde{e})},
\end{equation}
which of course implies~\eqref{eq:lim_eps_boundary}.

\medskip

Collecting~\eqref{eq:limit_eps}, \eqref{eq:second_term}, \eqref{eq:star_boundary}, \eqref{eq:eps_boundary} and~\eqref{eq:lim_eps_boundary}, we infer that
\begin{align*}
  & \lim_{\eps\to 0} \int_{D_c} k_\eps \nabla \widecheck{u}_\eps \cdot \nabla \widecheck{v}
  \\
  &= - \sum_{i = 1}^2 \int_{D_c} \widecheck{u}_0 \, k^\star \, e_i \cdot \partial_i\nabla\varphi + \sum_{i = 1}^2 \int_{\Gamma_f} \widecheck{u}_0 \, \partial_i\varphi \, k^\star \, e_i \cdot n_{\Gamma_f} + \sum_{i = 1}^2 \int_{\Gamma_c} \widecheck{u}_0 \, \partial_i\varphi \, k^\star \, e_i \cdot n_{\Gamma_c}
  \\
  &= \int_{D_c} k^\star \nabla \widecheck{u}_0 \cdot \nabla \varphi,
\end{align*}
where we again have used an integration by parts to deduce the last equality. 

We have thus identified the limit when $\eps\to 0$ of the first term of~\eqref{eq:arlequin_2nd_line}. Proceeding similarly for the second term, we have
$$
\int_{D_f} k_\eps \nabla \widecheck{u}_\eps \cdot \nabla \widecheck{v} \xrightarrow[\eps \to 0]{} \int_{D_f} k^\star \nabla \widecheck{u}_0 \cdot \nabla \varphi,
$$
for $\widecheck{v}$ given by~\eqref{eq:test_function}. The last two terms of~\eqref{eq:arlequin_2nd_line} are easy to handle since $\psi^H_{\overline{k}, k_\eps}$ is assumed to converge (in the finite dimensional space $W_H^{\rm enrich}$) to some $\psi^H_0$ and the test function $\widecheck{v}$ (which depends on $\eps$) weakly converges in $H^1(D_c)$ to $\varphi$. Thus, passing to the limit in~\eqref{eq:arlequin_2nd_line} with the test function~\eqref{eq:test_function} (and reinstating our original notation) yields
\begin{multline*}
  \forall \varphi \in C^\infty(\overline{D_c \cup D_f}), \quad \frac{1}{2} \int_{D_c} k^\star \nabla \widecheck{u}_0 \cdot \nabla \varphi + \int_{D_f} k^\star \nabla \widecheck{u}_0 \cdot \nabla \varphi \\ - \int_{D_c} \nabla \psi^H_0 \cdot \nabla \varphi - \int_{D_c} \psi^H_0 \varphi = 0.
\end{multline*}
We finally use the density of $C^\infty(\overline{D_c \cup D_f})$ in $H^1(D_c \cup D_f)$ to extend this equality for any $\varphi \in H^1(D_c \cup D_f)$.

\medskip

Passing to the limit $\eps\to 0$ in the first and third lines of~\eqref{eq:system_k_star} is straightforward. We deduce that $(\overline{u}^H_0,\widecheck{u}_0,\psi^H_0)$ is a solution to~\eqref{eq:system_contradiction}. We conclude the proof of Lemma~\ref{th:multiplication} by noting that~\eqref{eq:system_contradiction} has a unique solution. 
\end{proof}

\subsubsection{The specific case of homogeneous materials} \label{sec:tech_lemma_consistency}

We have pointed out in Section~\ref{sec:opt_heuristics} that, if the material in $D \cup D_c \cup D_f$ is homogeneous, then the response is linear, and that the converse statement also holds true. We now make this assertion precise by studying the following system (note that we again assume, similarly to Section~\ref{sec:tech_lemma_homog}, that $h = 0$): for any constant, scalar-valued coefficients $\overline{k}_a$ and $\overline{k}_b$, find $\overline{u}^H_{\overline{k}_a, \overline{k}_b} \in V_H^{\rm Dir BC}$, $\widecheck{u}_{\overline{k}_a, \overline{k}_b} \in H^1(D_c \cup D_f)$ and $\psi^H_{\overline{k}_a, \overline{k}_b} \in W_H^{\rm enrich}$ such that
\begin{equation} \label{eq:arlequin_limit_lemma}
  \begin{cases}
  \forall \overline{v}^H \in V_H^0, \quad & \overline{A}_{\overline{k}_a}(\overline{u}^H_{\overline{k}_a, \overline{k}_b},\overline{v}^H) + {\cal C}(\overline{v}^H,\psi^H_{\overline{k}_a, \overline{k}_b}) = 0,
  \\ \noalign{\vskip 3pt}
  \forall \widecheck{v} \in H^1(D_c \cup D_f), \quad & \widecheck{A}_{\overline{k}_b}(\widecheck{u}_{\overline{k}_a, \overline{k}_b},\widecheck{v}) - {\cal C}(\widecheck{v},\psi^H_{\overline{k}_a, \overline{k}_b}) = 0,
  \\ \noalign{\vskip 3pt}
  \forall \phi^H \in W_H^{\rm enrich}, \quad & {\cal C}(\overline{u}^H_{\overline{k}_a, \overline{k}_b}-\widecheck{u}_{\overline{k}_a, \overline{k}_b},\phi^H) = 0.
  \end{cases}
\end{equation}
On purpose, we have made explicit the dependency of the solution to~\eqref{eq:arlequin_limit_lemma} with respect to $\overline{k}_a$ and $\overline{k}_b$.

\begin{lemma}\label{lemma:optimization}
  We consider~\eqref{eq:arlequin_limit_lemma} for some constant coefficients $\overline{k}_a \in [0,\infty)$ and $\overline{k}_b \in (0,\infty)$.
  
  If $\overline{k}_a = \overline{k}_b$, then the solution to~\eqref{eq:arlequin_limit_lemma} is $\overline{u}^H_{\overline{k}_a, \overline{k}_b}(x) = x_1$ in $D \cup D_c$, $\widecheck{u}_{\overline{k}_a, \overline{k}_b}(x) = x_1$ in $D_c \cup D_f$ and $\psi^H_{\overline{k}_a, \overline{k}_b} = \overline{k}_a \, \psi_0$ in $D_c$, where $\psi_0$ is the Lagrange multiplier function defined by~\eqref{eq:psi_0}.

  Conversely, if $(\overline{u}^H_{\overline{k}_a, \overline{k}_b},\widecheck{u}_{\overline{k}_a, \overline{k}_b},\psi^H_{\overline{k}_a, \overline{k}_b})$ is a solution to~\eqref{eq:arlequin_limit_lemma} with $\overline{u}^H_{\overline{k}_a, \overline{k}_b}(x) = x_1$ in $D \cup D_c$, then $\widecheck{u}_{\overline{k}_a, \overline{k}_b}(x) = x_1$ in $D_c \cup D_f$, $\psi^H_{\overline{k}_a, \overline{k}_b} = \overline{k}_a \, \psi_0$ in $D_c$ and $\overline{k}_a = \overline{k}_b$.
\end{lemma}

This result is the analogue of~\cite[Lemma~2.1]{ref_olga_comp} when taking into account the discretization on the coarse mesh ${\cal T}_H$. Note also that the fact that we work in the enriched space $W_H^{\rm enrich}$ rather than $W_H$ is pivotal for this lemma.

\begin{remark} \label{rem:consistency}
  We note that the approach using $W_{H,h}^{\rm enrich}$ (or $W_H^{\rm enrich}$) as the Lagrange multiplier approximation space is {\em consistent} in the following sense. Lem\-ma~\ref{th:multiplication} means that, in the limit $\eps \to 0$, the problem~\eqref{eq:system_k_star} is well approximated by its homogenized limit~\eqref{eq:system_contradiction}.

  Considering the choice $\overline{k} = k^\star$, we wish the function $\overline{u}^H_{k^\star,k^\star}(x) = x_1$ to be a solution of that system, which ensures that $\overline{k} = k^\star$ is a minimizer of the optimization problem~\eqref{eq:optim_J}. When working with the enriched space $W_H^{\rm enrich}$, this is indeed the case: in view of the first assertion of Lemma~\ref{lemma:optimization}, the unique solution to~\eqref{eq:system_contradiction} with $\overline{k} = k^\star$ is $(\overline{u}^H_{k^\star,k^\star},\widecheck{u}_{k^\star,k^\star},\psi^H_{k^\star,k^\star}) = (x_1,x_1,k^\star \psi_0)$.
  And conversely, if $\overline{u}^H_{\overline{k},k^\star}(x) = x_1$, that is if we reach a minimum in~\eqref{eq:optim_J}, then $\overline{k} = k^\star$ and we have correctly recovered the homogenized coefficient.
\end{remark}

\begin{proof}[Proof of Lemma~\ref{lemma:optimization}] 
We start by the first assertion and assume $\overline{k}_a = \overline{k}_b > 0$. We immediately get the result, recalling that the system~\eqref{eq:arlequin_limit_lemma} has a unique solution and noticing that $\big( \overline{u}^H_{\overline{k}_a,\overline{k}_b}(x),\widecheck{u}_{\overline{k}_a,\overline{k}_b}(x),\psi^H_{\overline{k}_a,\overline{k}_b}(x) \big) = (x_1,x_1,\overline{k}_a \, \psi_0(x))$, where $\psi_0$ is defined by~\eqref{eq:psi_0}, is a solution to~\eqref{eq:arlequin_limit_lemma}.

We now turn to the second assertion and hence assume that $\overline{u}^H_{\overline{k}_a,\overline{k}_b}(x) = x_1$ in $D \cup D_c$. The first line of~\eqref{eq:arlequin_limit_lemma} reads as
\begin{equation}\label{eq:first_line}
  \forall \overline{v}^H \in V_H^0, \qquad \overline{A}_{\overline{k}_a}(x_1,\overline{v}^H) + {\cal C}(\overline{v}^H,\psi^H_{\overline{k}_a,\overline{k}_b}) = 0.
\end{equation}
Since $\psi^H_{\overline{k}_a,\overline{k}_b} \in W_H^{\rm enrich} = W_H + \text{Span $\psi_0$}$, we can represent it as 
\begin{equation}\label{eq:lagrange_multiplier_expansion}
  \psi^H_{\overline{k}_a,\overline{k}_b} = \tau \, \psi_0 + \widetilde{\psi}^H,
\end{equation} 
for some $\tau \in \RR$ and some $\widetilde{\psi}^H \in W_H$. We infer from~\eqref{eq:first_line} and~\eqref{eq:lagrange_multiplier_expansion} that
$$
\forall \overline{v}^H \in V_H^0, \qquad {\cal C}(\overline{v}^H,\widetilde{\psi}^H) = - \tau \, {\cal C}(\overline{v}^H,\psi_0) - \overline{A}_{\overline{k}_a}(x_1,\overline{v}^H),
$$
that provides us with an expression of $\widetilde{\psi}^H$ in terms of $\tau$ and $\overline{k}_a$, using the linearity of the problem and the fact that $\overline{k}_a$ is a scalar: 
\begin{equation} \label{eq:covid6}
\widetilde{\psi}^H = - \tau \, \widetilde{\psi}_1^H + \overline{k}_a \, \widetilde{\psi}_2^H,
\end{equation}
with $\widetilde{\psi}_1^H \in W_H$ and $\widetilde{\psi}_2^H \in W_H$ uniquely defined by
\begin{equation} \label{eq:LMtech}
  \begin{array}{l}
    \forall \overline{v}^H \in V_H^0, \qquad {\cal C}(\overline{v}^H,\widetilde{\psi}_1^H) = {\cal C}(\overline{v}^H,\psi_0),
    \\ \noalign{\vskip 3pt}
    \forall \overline{v}^H \in V_H^0, \qquad {\cal C}(\overline{v}^H,\widetilde{\psi}_2^H) = -\overline{A}_1(x_1,\overline{v}^H),
  \end{array}
\end{equation}
where the bilinear form $\overline{A}_1$ is defined by
\begin{equation} \label{eq:def_overline_A_1}
\overline{A}_1(\overline{u},\overline{v}) = \int_D \nabla \overline{u}(x) \cdot \nabla \overline{v}(x) + \frac{1}{2} \int_{D_c} \nabla \overline{u}(x) \cdot \nabla \overline{v}(x).
\end{equation}

Let us introduce the $H^1$-orthogonal projection operators to the coarse finite element spaces $\Pi_H : H^1(D_c) \to W_H$ and $\Pi_H^{\rm enrich} : H^1(D_c) \to W^{\rm enrich}_H$ defined as follows: for any $v \in H^1(D_c)$, $\Pi_H(v) \in W_H$ is such that
\begin{equation}\label{eq:Pih_a}
  \forall \varphi^H \in W_H, \qquad \left(\Pi_H(v),\varphi^H\right)_{H^1(D_c)} = \left(v,\varphi^H\right)_{H^1(D_c)},
\end{equation}
and $\Pi_H^{\rm enrich}(v) \in W^{\rm enrich}_H$ is such that
\begin{equation}\label{eq:Pih_b}
  \forall \varphi^H \in W_H^{\rm enrich}, \qquad \left(\Pi_H^{\rm enrich}(v),\varphi^H\right)_{H^1(D_c)} = \left(v,\varphi^H\right)_{H^1(D_c)}, 
\end{equation}
where $\left(\cdot,\cdot\right)_{H^1(D_c)}$ is the $H^1$ scalar product in $D_c$.

We observe that the Lagrange multiplier $\psi_0$ defined by~\eqref{eq:psi_0} (the variational formulation of which is~\eqref{eq:psi_0_FV}) satisfies
\begin{equation}\label{eq:LMvariational}
  \forall \overline{v} \in H^1(D \cup D_c) \ \ \text{with $\overline{v} = 0$ on $\Gamma$}, \quad \overline{A}_1(x_1,\overline{v}) = - {\cal C}(\overline{v},\psi_0).
\end{equation}
In view of the first line of~\eqref{eq:LMtech} and of the definition of $\Pi_H$, we have that $\widetilde{\psi}^H_1 = \Pi_H(\psi_0)$. In view of~\eqref{eq:LMvariational} and the second line of~\eqref{eq:LMtech}, we see that $\widetilde{\psi}^H_2$ satisfies the same equation as $\widetilde{\psi}^H_1$. We thus have $\widetilde{\psi}^H_2 = \widetilde{\psi}^H_1 = \Pi_H(\psi_0)$. Inserting this relation in~\eqref{eq:covid6}, we deduce that $\widetilde{\psi}^H = (\overline{k}_a - \tau) \, \Pi_H(\psi_0)$, and thus
\begin{equation}\label{eq:psi_decomposition}
  \psi^H_{\overline{k}_a,\overline{k}_b} = \tau \, \psi_0 + (\overline{k}_a - \tau) \, \Pi_H(\psi_0),
\end{equation}
where the constant $\tau$ will be determined later.

\medskip

We now turn to the second line of~\eqref{eq:arlequin_limit_lemma}. Let us introduce $\widecheck{u}_1$ and $\widecheck{u}_2$ in $H^1(D_c \cup D_f)$ such that $\dps \int_{D_c} \widecheck{u}_1 = \int_{D_c} \widecheck{u}_2 = 0$ and 
\begin{equation}\label{eq:u1_u2}
  \begin{array}{ll}
    \forall \widecheck{v} \in H^1(D_c \cup D_f), \qquad & \widecheck{A}_1(\widecheck{u}_1,\widecheck{v}) = {\cal C}(\widecheck{v},\psi_0),
    \\ \noalign{\vskip 3pt}
    \forall \widecheck{v} \in H^1(D_c \cup D_f), \qquad & \widecheck{A}_1(\widecheck{u}_2,\widecheck{v}) = {\cal C}(\widecheck{v},\Pi_H(\psi_0)),
  \end{array}
\end{equation}
where the bilinear form $\widecheck{A}_1$ is defined by
$$
\widecheck{A}_1(\widecheck{u},\widecheck{v}) = \frac{1}{2} \int_{D_c} \nabla \widecheck{u}(x) \cdot \nabla \widecheck{v}(x) + \int_{D_f} \nabla \widecheck{u}(x) \cdot \nabla \widecheck{v}(x).
$$
Using the definition~\eqref{eq:psi_0} of the Lagrange multiplier $\psi_0$, we obtain from the first line of~\eqref{eq:u1_u2} that $\widecheck{u}_1(x) = x_1$ in $D_c \cup D_f$.

Inserting the expression~\eqref{eq:psi_decomposition} for $\psi^H_{\overline{k}_a, \overline{k}_b}$ in the second line of~\eqref{eq:arlequin_limit_lemma}, we obtain
\begin{equation}\label{eq:u_star}
  \widecheck{u}_{\overline{k}_a, \overline{k}_b} = \lambda + \frac{\tau}{\overline{k}_b} \, \widecheck{u}_1 + \frac{\overline{k}_a - \tau}{\overline{k}_b} \, \widecheck{u}_2,
\end{equation}
where $\lambda \in \RR$ is an arbitrary constant.

\medskip

To identify the constants $\lambda$ and $\tau$, we use the third line of~\eqref{eq:arlequin_limit_lemma}, that reads as
\begin{equation}\label{eq:constrain}
  \forall \phi^H \in W_H^{\rm enrich}, \quad {\cal C}\left(\lambda + \frac{\tau-\overline{k}_b}{\overline{k}_b} \, \widecheck{u}_1 + \frac{\overline{k}_a - \tau}{\overline{k}_b} \, \widecheck{u}_2, \phi^H \right) = 0,
\end{equation}
where we have used that $\overline{u}^H_{\overline{k}_a,\overline{k}_b}(x) = x_1 = \widecheck{u}_1(x)$ in $D_c$. Taking $\phi^H = 1$ and using that the mean over $D_c$ of $\widecheck{u}_1$ and $\widecheck{u}_2$ vanishes, we get $\lambda = 0$.

We claim that $\widecheck{u}_1$ and $\Pi^{\rm enrich}_H (\widecheck{u}_2)$ are linearly independent functions on $D_c$, a fact that will be useful below. In order to prove this claim, we argue by contradiction. Since $\widecheck{u}_1$ does not identically vanish on $D_c$, we assume that there exists $\alpha \in \RR$ such that 
\begin{equation}\label{eq:collinearity}
  \Pi^{\rm enrich}_H (\widecheck{u}_2) = \alpha \, \widecheck{u}_1.
\end{equation} 
For any $\widecheck{v} \in H^1(D_c \cup D_f)$, we compute, using~\eqref{eq:u1_u2}, that
\begin{align*}
  \widecheck{A}_1(\widecheck{u}_2 - \alpha \, \widecheck{u}_1,\widecheck{v})
  &=
  \widecheck{A}_1(\widecheck{u}_2,\widecheck{v}) - \alpha \, \widecheck{A}_1(\widecheck{u}_1,\widecheck{v})
  \\
  &=
  {\cal C}(\widecheck{v},\Pi_H(\psi_0)) - \alpha \, {\cal C}(\widecheck{v},\psi_0)
  \\
  &=
  {\cal C}(\widecheck{v},\Pi_H(\psi_0) - \alpha \, \psi_0).
\end{align*}
Taking $\widecheck{v} = \widecheck{u}_2 - \alpha \, \widecheck{u}_1$ in the equation above, we obtain that
\begin{align} 
  \widecheck{A}_1\left(\widecheck{u}_2 - \alpha \, \widecheck{u}_1,\widecheck{u}_2 - \alpha \, \widecheck{u}_1\right)
  &=
  {\cal C}(\widecheck{u}_2 - \alpha \, \widecheck{u}_1,\Pi_H(\psi_0)-\alpha \, \psi_0)
  \nonumber
  \\
  &=
  {\cal C}(\Pi_H^{\rm enrich}\left(\widecheck{u}_2-\alpha \, \widecheck{u}_1\right),\Pi_H(\psi_0)-\alpha \, \psi_0),
  \label{eq:colliniearity2}
\end{align}
where the last equality stems from the definition of the projection operator $\Pi_H^{\rm enrich}$. We next observe that $\Pi_H^{\rm enrich}\left(\widecheck{u}_2-\alpha \, \widecheck{u}_1\right) = 0$, because of~\eqref{eq:collinearity} and the fact that $\Pi_H^{\rm enrich}(\widecheck{u}_1) = \widecheck{u}_1$ (recall that $\widecheck{u}_1(x) = x_1$ in $D_c \cup D_f$ and thus $\widecheck{u}_1 \in W_H^{\rm enrich}$). The right-hand side of~\eqref{eq:colliniearity2} thus vanishes. By definition of the bilinear form $\widecheck{A}_1$, this implies that $\widecheck{u}_2 = \alpha \, \widecheck{u}_1 + \widehat{\lambda}$ on $D_c \cup D_f$ for some constant $\widehat{\lambda}$. Since $\widecheck{u}_2$ and $\widecheck{u}_1$ are functions the average over $D_c$ of which vanishes, we obtain $\widehat{\lambda}=0$ and thus $\widecheck{u}_2 = \alpha \, \widecheck{u}_1$.

We thus infer from~\eqref{eq:u1_u2} that, for any $\widecheck{v} \in H^1(D_c \cup D_f)$,
$$
{\cal C}(\widecheck{v},\Pi_H(\psi_0)) = \widecheck{A}_1(\widecheck{u}_2,\widecheck{v}) = \alpha \, \widecheck{A}_1(\widecheck{u}_1,\widecheck{v}) = \alpha \, {\cal C}(\widecheck{v},\psi_0).
$$
This yields $\Pi_H(\psi_0) = \alpha \, \psi_0$. If $\alpha \neq 0$, this implies that $\psi_0 \in W_H$, a fact that is obviously wrong. We then get $\alpha = 0$, hence $\Pi_H(\psi_0) = 0$, and thus, for any $\varphi^H \in W_H$,
\begin{multline*}
0
=
\left(\Pi_H(\psi_0),\varphi^H\right)_{H^1(D_c)} = \left(\psi_0,\varphi^H\right)_{H^1(D_c)}
=
{\cal C}(\psi_0,\varphi^H)
\\
=
\frac{1}{2} \int_{\Gamma_c} (e_1 \cdot n_{\Gamma_c}) \, \varphi^H - \frac{1}{2} \int_{\Gamma_f} (e_1 \cdot n_{\Gamma_f}) \, \varphi^H,
\end{multline*}
where we have used~\eqref{eq:psi_0_FV} in the last equality. Since the value of $\varphi^H$ can be chosen independently on $\Gamma_c$ and $\Gamma_f$, this implies that $\dps \int_{\Gamma_c} \varphi^H \, e_1 \cdot n_{\Gamma_c} = 0 = \int_{\Gamma_f} \varphi^H \, e_1 \cdot n_{\Gamma_f}$, which leads to a contradiction.
This concludes the proof of our claim.

\medskip

We now return to~\eqref{eq:constrain}, recall that $\lambda=0$ and take
$$
\phi^H = \Pi^{\rm enrich}_H\left(\frac{\tau-\overline{k}_b}{\overline{k}_b} \, \widecheck{u}_1 + \frac{\overline{k}_a - \tau}{\overline{k}_b} \, \widecheck{u}_2\right).
$$
We hence obtain that
$$
\Pi^{\rm enrich}_H\left(\frac{\tau-\overline{k}_b}{\overline{k}_b} \, \widecheck{u}_1 + \frac{\overline{k}_a - \tau}{\overline{k}_b} \, \widecheck{u}_2\right) = 0,
$$
which reads as 
$$
\frac{\tau-\overline{k}_b}{\overline{k}_b} \, \widecheck{u}_1 + \frac{\overline{k}_a - \tau}{\overline{k}_b} \, \Pi^{\rm enrich}_H\left(\widecheck{u}_2\right) = 0.
$$
Using the linear independence of $\widecheck{u}_1$ and $\Pi_H^{\rm enrich}(\widecheck{u}_2)$, we obtain $\overline{k}_b = \tau = \overline{k}_a$. This concludes the proof of Lemma~\ref{lemma:optimization}.
\end{proof}

\subsection{Well-posedness of the optimization problem upon $\overline{k}$ for a fixed value of $\eps$} \label{sec:opt}

In this section, we investigate the existence of a minimizer to the optimization problem~\eqref{eq:optim_J}--\eqref{eq:def_J}. More precisely, we show the following theorem, which is our first main result.

\begin{theorem}\label{th:optimization}
  Let $k_\eps$ be given by~\eqref{eq:structure-k} for some fixed periodic coefficient $k_{\rm per}$ that satisfies the classical boundedness and coercivity conditions~\eqref{eq:boundedness+coercivity}. We make the regularity assumption~\eqref{eq:holder_continuity} and the geometric assumption~\eqref{eq:boundaries}. 
  
  For any fixed $\eps>0$, $H > 0$ and $h > 0$, the optimization problem~\eqref{eq:optim_J} has at least one solution $\overline{k}^{\rm opt}_\eps(H,h)$, provided the two variational conditions made precise in~\eqref{eq:J_contradiction} and~\eqref{eq:J_contradiction2} below are satisfied. These conditions in particular hold true in the limit $(h \to 0,\eps \to 0)$.
\end{theorem}
The uniqueness of the minimizer to~\eqref{eq:optim_J} is investigated in Section~\ref{sec:uniqueness} below, for $\eps$ sufficiently small.

\begin{remark} \label{rem:k_bounds}
If we impose that $\overline{k}$ is bounded from above (resp. bounded away from zero) by some finite positive constant in the minimization problem~\eqref{eq:optim_J}, then Theorem~\ref{th:optimization} holds true without the additional assumption~\eqref{eq:J_contradiction} (resp.~\eqref{eq:J_contradiction2}).
\end{remark}

The remainder of this Section~\ref{sec:opt} is devoted to the proof of Theorem~\ref{th:optimization}. To that aim, we consider a minimizing sequence $\{ \overline{k}^n \}_{n \in \NN}$ of the optimization problem~\eqref{eq:optim_J}, that is a sequence that satisfies the inequality 
\begin{equation} \label{eq:quasiminimizers} 
  I_{\eps, H, h} \leq J_{\eps, H, h}(\overline{k}^n) = \int_{D \cup D_c} \big| \nabla \overline{u}^H_{\overline{k}^n, k_\eps} - e_1 \big|^2 \leq I_{\eps, H, h} + \frac{1}{n},
\end{equation}
where $(\overline{u}^H_{\overline{k}^n, k_\eps},\widecheck{u}^h_{\eps,\overline{k}^n, k_\eps}, \psi^H_{\overline{k}^n, k_\eps}) \in V_H^{\rm Dir BC} \times V_h \times W_{H,h}^{\rm enrich}$ is the solution to~\eqref{eq:arlequin} for the tentative constant coefficient $\overline{k} = \overline{k}^n$, namely
\begin{equation} \label{eq:arlequin_n}
  \begin{cases}
  \forall \overline{v}^H \in V_H^0, \quad &\overline{A}_{\overline{k}^n}(\overline{u}^H_{\overline{k}^n, k_\eps},\overline{v}^H) + {\cal C}(\overline{v}^H,\psi^H_{\overline{k}^n, k_\eps}) = 0,
  \\ \noalign{\vskip 3pt}
  \forall \widecheck{v}^h \in V_h, \quad & \widecheck{A}_{k_\eps}(\widecheck{u}^h_{\eps,\overline{k}^n, k_\eps},\widecheck{v}^h) - {\cal C}(\widecheck{v}^h,\psi^H_{\overline{k}^n, k_\eps}) = 0,
  \\ \noalign{\vskip 3pt}
  \forall \phi^H \in W_{H,h}^{\rm enrich}, \quad & {\cal C}(\overline{u}^H_{\overline{k}^n, k_\eps}-\widecheck{u}^h_{\eps,\overline{k}^n, k_\eps},\phi^H) = 0.
  \end{cases}
\end{equation}
The proof of Theorem~\ref{th:optimization} falls in three steps:
\begin{itemize}
\item we first show a priori bounds on $(\overline{u}^H_{\overline{k}^n, k_\eps},\widecheck{u}^h_{\eps,\overline{k}^n, k_\eps}, \psi^H_{\overline{k}^n, k_\eps})$ in Section~\ref{sec:bounds}.
\item we next show that, up to a subsequence extraction, $\overline{k}^n$ converges to some limit $\overline{k}^\infty < \infty$ in Section~\ref{sec:k_n}, under assumption~\eqref{eq:J_contradiction}.
\item we then show that $\overline{k}^\infty > 0$ (under assumption~\eqref{eq:J_contradiction2}) in Section~\ref{sec:non_0}. Since the function $\overline{k} \mapsto J_{\eps, H, h}(\overline{k})$ is continuous on $(0,\infty)$, this shows that $\overline{k}^\infty$ is a minimizer of~\eqref{eq:optim_J}.
\end{itemize}

\subsubsection{Bounds on $\overline{u}^H_{\overline{k}^n, k_\eps}$, $\widecheck{u}^h_{\eps,\overline{k}^n, k_\eps}$ and $\psi^H_{\overline{k}^n, k_\eps}$} \label{sec:bounds}

\paragraph{Bound on $\overline{u}^H_{\overline{k}^n, k_\eps}$.}

The optimization problem~\eqref{eq:optim_J} has been designed so that the homogenized coefficient $k^\star$ is an admissible test coefficient in~\eqref{eq:optim_J}. Hence, by definition of $I_{\eps, H, h}$, we have in particular
\begin{equation}\label{eq:u_bound_eps}
  I_{\eps, H, h} \leq J_{\eps, H, h}(k^\star) = \int_{D \cup D_c} \Big| \nabla \overline{u}^H_{k^\star, k_\eps} - e_1 \Big|^2,
\end{equation}
where $\left(\overline{u}^H_{k^\star, k_\eps},\widecheck{u}^h_{\eps,k^\star, k_\eps}, \psi^H_{k^\star, k_\eps}\right)$ is the solution to~\eqref{eq:arlequin} with the constant coefficient $\overline{k} = k^\star$. Since we know that $\left( \overline{u}^H_{k^\star, k_\eps}, \widecheck{u}^h_{\eps, k^\star, k_\eps}\right)$ is the minimizer of~\eqref{eq:pb_min_H_enriched} with $\overline{k} = k^\star$, we can compare its energy with that of the particular choice $\left(\overline{u}^H(x) = x_1, \widecheck{u}^h_\eps(x) = x_1\right)$. Writing that
$$
{\cal E}\left( \overline{u}^H_{k^\star, k_\eps}, \widecheck{u}^h_{\eps, k^\star, k_\eps}\right) \leq {\cal E}\left(\overline{u}^H(x) = x_1, \widecheck{u}^h_\eps(x) = x_1\right),
$$
we obtain
\begin{multline} \label{eq:vraiment_energie}
  \int_D k^\star \left| \nabla \overline{u}^H_{k^\star, k_\eps} \right|^2 + \frac{1}{2} \int_{D_c} \Big( k^\star \left|\nabla \overline{u}^H_{k^\star, k_\eps} \right|^2 + k_\eps \nabla \widecheck{u}^h_{\eps, k^\star, k_\eps} \cdot \nabla \widecheck{u}^h_{\eps, k^\star, k_\eps} \Big) \\ + \int_{D_f} k_\eps \nabla \widecheck{u}^h_{\eps, k^\star, k_\eps} \cdot \nabla \widecheck{u}^h_{\eps, k^\star, k_\eps} \leq k^\star \, | D | + \frac{k^\star}{2} \, | D_c | + \frac{1}{2} \int_{D_c} k_\eps e_1 \cdot e_1 + \int_{D_f} k_\eps e_1 \cdot e_1.
\end{multline}
Using that $k_\eps$ is uniformly bounded and coercive, we obtain 
\begin{equation}\label{eq:u_k_bound}
  \int_D \, |\nabla \overline{u}^H_{k^\star, k_\eps}|^2 + \frac{1}{2} \int_{D_c} |\nabla \overline{u}^H_{k^\star, k_\eps}|^2 \leq C,
\end{equation}
for some constant $C$ independent of $\eps$, $H$ and $h$ (this independence with respect to $h$ and $\eps$ is important since we will later on use this bound and the subsequent ones in the regime $h, \eps \to 0$).

Collecting~\eqref{eq:quasiminimizers}, \eqref{eq:u_bound_eps} and~\eqref{eq:u_k_bound} and using the boundary conditions on $\Gamma$ for $\overline{u}^H_{\overline{k}^n, k_\eps}$, we immediately obtain that the sequence $\overline{u}^H_{\overline{k}^n, k_\eps}$ is bounded in $H^1(D \cup D_c)$: there exists some constant $C$ independent of $n$, $\eps$, $H$ and $h$ such that
\begin{equation}\label{eq:u_k^n_bound}
  \| \overline{u}^H_{\overline{k}^n, k_\eps} \|_{H^1(D \cup D_c)} \leq C.
\end{equation}
Since $\overline{u}^H_{\overline{k}^n, k_\eps}$ belongs to the finite dimensional space $V_H^{\rm Dir BC}$, we deduce that there exists a subsequence, which we still denote by $\overline{u}^H_{\overline{k}^n, k_\eps}$, that converges in $V_H$ (when $n \to \infty$) to some $\overline{u}^H_{\infty,k_\eps} \in V_H^{\rm Dir BC}$.

\paragraph{Bound on $\widecheck{u}^h_{\eps,\overline{k}^n, k_\eps}$.}

To bound this function, we first extend the function $\overline{u}^H_{\overline{k}^n, k_\eps}$ (which is defined in $D \cup D_c$) inside the domain $D_f$, in order to build an appropriate test function for the second line of~\eqref{eq:arlequin_n}. There are several ways to perform this extension and we have chosen to proceed as follows. We first build a function $\overline{u}^H_{f, \overline{k}^n} \in H^1(D_f)$ that satisfies $\left( \overline{u}^H_{f, \overline{k}^n} \right) \big|_{D_f} = \left( \overline{u}^H_{\overline{k}^n, k_\eps} \right) \big|_{D_c}$ on the boundary $\Gamma_f$ and $\| \overline{u}^H_{f, \overline{k}^n} \|_{H^1(D_f)} \leq C \, \| \overline{u}^H_{\overline{k}^n, k_\eps} \|_{H^{1/2}(\Gamma_f)}$ for some $C$ independent of $n$, $\eps$, $H$ and $h$. For instance, we can define $\overline{u}^H_{f, \overline{k}^n}$ as the harmonic extension of $\overline{u}^H_{\overline{k}^n, k_\eps} \big|_{\Gamma_f}$ in $D_f$. 

We now pass from $\overline{u}^H_{f,\overline{k}^n}$ to a piecewise affine function $\widetilde{u}^H_{\overline{k}^n}$ belonging to the space
\begin{equation} \label{eq:def_V_tilde_h}
\widetilde{V}_h = \left\{ u^h \in H^1(D_f), \quad \text{$u^h$ is piecewise affine on the fine mesh ${\cal T}_h$} \right\}
\end{equation}
using a Scott-Zhang type interpolation, which has the advantage of being defined for functions that are not necessarily continuous (in contrast to nodal interpolation) and of preserving boundary conditions (in contrast to Clément interpolation). More precisely, using~\cite{scott_zhang90} (see also~\cite[Theorem~3.4]{melenk2003hp}), we know that there exists a linear and continuous operator $I^{\rm SZ} : H^1(D_f) \to \widetilde{V}_h$ such that, if $v \in H^1(D_f)$ is continuous and piecewise affine on $\partial D_f$, then $I^{\rm SZ} v = v$ on $\partial D_f$. We hence set
\begin{equation}\label{eq:u_help}
  \widetilde{u}^H_{\overline{k}^n} = I^{\rm SZ} \, \overline{u}^H_{f,\overline{k}^n},
\end{equation}
which satisfies $\widetilde{u}^H_{\overline{k}^n} \in \widetilde{V}_h$, $\widetilde{u}^H_{\overline{k}^n} = \overline{u}^H_{f,\overline{k}^n}$ on $\Gamma_f$ and $\| \widetilde{u}^H_{\overline{k}^n} \|_{H^1(D_f)} \leq C \, \| \overline{u}_{f,\overline{k}^n}^H \|_{H^1(D_f)}$ for some $C$ independent of $n$, $\eps$, $H$ and $h$ (it actually only depends on $D_f$).

\medskip

We hence have built some $\widetilde{u}^H_{\overline{k}^n} \in \widetilde{V}_h$ satisfying $\widetilde{u}^H_{\overline{k}^n} = \overline{u}^H_{\overline{k}^n, k_\eps}$ on $\Gamma_f$ and such that
\begin{equation}\label{eq:covid1}
\| \widetilde{u}^H_{\overline{k}^n} \|_{H^1(D_f)} \leq C \, \| \overline{u}_{f,\overline{k}^n}^H \|_{H^1(D_f)} \leq C \, \| \overline{u}^H_{\overline{k}^n, k_\eps} \|_{H^{1/2}(\Gamma_f)} \leq C \, \| \overline{u}^H_{\overline{k}^n, k_\eps} \|_{H^1(D_c)},
\end{equation}
for some $C$ independent of $n$, $\eps$, $H$ and $h$. We next introduce the extension of the function $\overline{u}^H_{\overline{k}^n, k_\eps}$ inside the domain $D_f$ defined by
$$
  \widetilde{u}^h_{\overline{k}^n} = 
  \begin{cases}
    \overline{u}^H_{\overline{k}^n, k_\eps} \text{ in $D_c$},
    \\ \noalign{\vskip 3pt}
    \widetilde{u}^H_{\overline{k}^n} \text{ in $D_f$}.
  \end{cases}
$$
Note that $\widetilde{u}^h_{\overline{k}^n}\in V_h$, since the fine mesh ${\cal T}_h$ is assumed to be a submesh of the coarse mesh ${\cal T}_H$ in $D_c$. In addition, using~\eqref{eq:covid1} and~\eqref{eq:u_k^n_bound}, we observe that
\begin{equation} \label{eq:covid2}
  \| \widetilde{u}^h_{\overline{k}^n} \|_{H^1(D_c \cup D_f)} \leq C,
\end{equation}
for some $C$ independent of $n$, $\eps$, $H$ and $h$. 

\medskip

We are now in position to bound the sequence $\widecheck{u}_{\eps,\overline{k}^n,k_\eps}^h$. Taking $\widecheck{v}^h = \widecheck{u}_{\eps,\overline{k}^n, k_\eps}^h - \widetilde{u}^h_{\overline{k}^n}$ in the second line of~\eqref{eq:arlequin_n} (which is a possible choice since both $\widecheck{u}_{\eps,\overline{k}^n, k_\eps}^h$ and $\widetilde{u}^h_{\overline{k}^n}$ belong to $V_h$), we have 
\begin{multline}\label{eq:third_line}
  \frac{1}{2} \int_{D_c} k_\eps \nabla \widecheck{u}_{\eps,\overline{k}^n, k_\eps}^h \cdot \nabla \left( \widecheck{u}_{\eps,\overline{k}^n, k_\eps}^h - \widetilde{u}^h_{\overline{k}^n} \right) + \int_{D_f} k_\eps \nabla \widecheck{u}_{\eps,\overline{k}^n, k_\eps}^h \cdot \nabla \left( \widecheck{u}_{\eps,\overline{k}^n, k_\eps}^h - \widetilde{u}^h_{\overline{k}^n} \right) \\ - \int_{D_c} \nabla \psi^H_{\overline{k}^n, k_\eps} \cdot \nabla \left( \widecheck{u}_{\eps,\overline{k}^n, k_\eps}^h - \widetilde{u}^h_{\overline{k}^n} \right) - \int_{D_c} \psi^H_{\overline{k}^n, k_\eps} \left( \widecheck{u}_{\eps,\overline{k}^n, k_\eps}^h - \widetilde{u}^h_{\overline{k}^n} \right) = 0.
\end{multline}
Using the third line of~\eqref{eq:arlequin_n} with $\phi^H = \psi^H_{\overline{k}^n, k_\eps}$ and recalling that $\widetilde{u}^h_{\overline{k}^n} = \overline{u}^H_{\overline{k}^n, k_\eps}$ in $D_c$, we see that the sum of the last two terms in~\eqref{eq:third_line} vanishes. We next use the Cauchy-Schwarz inequality and obtain
\begin{align*}
  & \frac{1}{2} \int_{D_c} k_\eps \, \nabla \widecheck{u}_{\eps,\overline{k}^n, k_\eps}^h \cdot \nabla \widecheck{u}_{\eps,\overline{k}^n, k_\eps}^h + \int_{D_f} k_\eps \, \nabla \widecheck{u}_{\eps,\overline{k}^n, k_\eps}^h \cdot \nabla \widecheck{u}_{\eps,\overline{k}^n, k_\eps}^h
  \\
  & =
  \frac{1}{2} \int_{D_c} k_\eps \, \nabla \widecheck{u}_{\eps,\overline{k}^n, k_\eps}^h \cdot \nabla \widetilde{u}^h_{\overline{k}^n} + \int_{D_f} k_\eps \, \nabla \widecheck{u}_{\eps,\overline{k}^n, k_\eps}^h \cdot \nabla \widetilde{u}^h_{\overline{k}^n}
  \\
  & \leq
  \frac{1}{2} \| k_\eps \|_{L^\infty(D_c)} \| \nabla \widecheck{u}_{\eps,\overline{k}^n, k_\eps}^h \|_{L^2(D_c)} \| \nabla \widetilde{u}^h_{\overline{k}^n} \|_{L^2(D_c)}
  \\
  & \quad
  + \| k_\eps \|_{L^\infty(D_f)} \| \nabla \widecheck{u}_{\eps,\overline{k}^n, k_\eps}^h \|_{L^2(D_f)} \| \nabla \widetilde{u}^h_{\overline{k}^n} \|_{L^2(D_f)}.
\end{align*}
Using that the oscillating coefficient $k_\eps$ is bounded and bounded away from zero (see~\eqref{eq:boundedness+coercivity}) and that $\widetilde{u}^h_{\overline{k}^n}$ is bounded in $H^1(D_c \cup D_f)$ (see~\eqref{eq:covid2}), we obtain that there exists a constant $C$ independent of $n$, $\eps$, $H$ and $h$ such that
\begin{equation} \label{eq:bound_u_eps}
  \forall n \in \NN, \quad \| \nabla \widecheck{u}_{\eps,\overline{k}^n, k_\eps}^h \|_{L^2(D_c \cup D_f)} \leq C.
\end{equation}
Testing the third line of~\eqref{eq:arlequin_n} with $\phi^H = 1$ (which indeed belongs to $W_{H,h}^{\rm enrich}$), we obtain that 
$$
\int_{D_c} \widecheck{u}_{\eps,\overline{k}^n, k_\eps}^h = \int_{D_c} \overline{u}_{\overline{k}^n, k_\eps}^H,
$$
and hence, using~\eqref{eq:u_k^n_bound}, we obtain that $\dps \int_{D_c} \widecheck{u}_{\eps,\overline{k}^n,k_\eps}^h$ is bounded. Thus, by the Poincar\'e-Wirtinger inequality, we deduce from~\eqref{eq:bound_u_eps} that the sequence $\widecheck{u}^h_{\eps,\overline{k}^n, k_\eps}$ is bounded in $H^1(D_c \cup D_f)$, independently of $n$, $\eps$, $H$ and $h$. Since $\widecheck{u}^h_{\eps, \overline{k}^n, k_\eps}$ belongs to the finite dimensional space $V_h$, we deduce that there exists a subsequence, which we still denote by $\widecheck{u}^h_{\eps, \overline{k}^n, k_\eps}$, that converges in $V_h$ (when $n \to \infty$) to some $\widecheck{u}^h_{\eps, \infty, k_\eps} \in V_h$ (in the particular case when $h=0$, the convergence is strong in $L^2(D_c \cup D_f)$ and weak in $H^1(D_c \cup D_f)$).

\paragraph{Bound on $\psi^H_{\overline{k}^n, k_\eps}$.}

We are now left with showing that the sequence of Lagrange multipliers is also bounded. To that aim, we proceed as above and first extend $\psi^H_{\overline{k}^n, k_\eps}$ (which is defined in $D_c$) inside the domain $D_f$, in order to again build an appropriate test function for the second line of~\eqref{eq:arlequin_n}. This extension is built following the same steps as above (see~\eqref{eq:u_help} and~\eqref{eq:covid1}), which thus allow to introduce some $\widetilde{\psi}^H_{\overline{k}^n} \in \widetilde{V}_h$ (where $\widetilde{V}_h \subset H^1(D_f)$ is defined by~\eqref{eq:def_V_tilde_h}) satisfying $\widetilde{\psi}^H_{\overline{k}^n} = \psi^H_{\overline{k}^n, k_\eps}$ on $\Gamma_f$ and such that
\begin{equation}\label{eq:covid3}
\| \widetilde{\psi}^H_{\overline{k}^n} \|_{H^1(D_f)} \leq C \, \| \psi^H_{\overline{k}^n, k_\eps} \|_{H^1(D_c)},
\end{equation}
for some $C$ independent of $n$, $\eps$, $H$ and $h$.

We next extend $\psi^H_{\overline{k}^n, k_\eps}$ inside the domain $D_f$ by introducing 
\begin{equation}\label{eq:tilde_psi}
  \widetilde{\psi}^h_{\overline{k}^n} = 
  \begin{cases}
    \psi^H_{\overline{k}^n, k_\eps} \text{ in $D_c$},
    \\ \noalign{\vskip 3pt}
    \widetilde{\psi}^H_{\overline{k}^n} \text{ in $D_f$}. 
  \end{cases}
\end{equation}
We note that $\widetilde{\psi}^h_{\overline{k}^n} \in H^1(D_c \cup D_f)$. In addition, both $\psi^H_{\overline{k}^n, k_\eps}$ and $\widetilde{\psi}^H_{\overline{k}^n}$ are piecewise affine functions on the fine mesh ${\cal T}_h$ (for $\psi^H_{\overline{k}^n, k_\eps}$, this is a consequence of the fact that the fine mesh ${\cal T}_h$ is assumed to be a submesh of the coarse mesh ${\cal T}_H$ in $D_c$ and of the specific approximation $\psi_0^h$, defined by~\eqref{eq:psi_0_FV_h}, of the Lagrange multiplier $\psi_0$). We hence deduce that $\widetilde{\psi}^h_{\overline{k}^n} \in V_h$.

Considering the test function $\widecheck{v}^h = \widetilde{\psi}^h_{\overline{k}^n}$ in the second line of~\eqref{eq:arlequin_n}, we obtain
\begin{multline*}
  \frac{1}{2} \int_{D_c} k_\eps \nabla \widecheck{u}_{\eps,\overline{k}^n, k_\eps}^h \cdot \nabla \widetilde{\psi}^h_{\overline{k}^n} + \int_{D_f} k_\eps \nabla \widecheck{u}_{\eps,\overline{k}^n, k_\eps}^h \cdot \nabla\widetilde{\psi}^h_{\overline{k}^n} \\ - \int_{D_c} \nabla \psi^H_{\overline{k}^n, k_\eps} \cdot \nabla\widetilde{\psi}^h_{\overline{k}^n} - \int_{D_c} \psi^H_{\overline{k}^n, k_\eps} \widetilde{\psi}^h_{\overline{k}^n} = 0.
\end{multline*}
Since $\widetilde{\psi}^h_{\overline{k}^n} = \psi^H_{\overline{k}^n, k_\eps}$ in $D_c$, we deduce that
\begin{multline} \label{eq:psi_bound}
  \| \psi^H_{\overline{k}^n, k_\eps} \|_{H^1(D_c)}^2 \leq \frac{1}{2} \| k_\eps \|_{L^\infty(D_c)} \| \nabla \widecheck{u}_{\eps,\overline{k}^n, k_\eps}^h \|_{L^2(D_c)} \| \nabla \psi^H_{\overline{k}^n, k_\eps} \|_{L^2(D_c)} \\ + \| k_\eps \|_{L^\infty(D_f)} \| \nabla \widecheck{u}_{\eps,\overline{k}^n, k_\eps}^h \|_{L^2(D_f)} \| \nabla \widetilde{\psi}^H_{\overline{k}^n} \|_{L^2(D_f)}.
\end{multline}
Using that $k_\eps$ is uniformly bounded and the bounds~\eqref{eq:covid3} and~\eqref{eq:bound_u_eps}, we infer from~\eqref{eq:psi_bound} that there exists a constant $C$ independent of $n$, $\eps$, $H$ and $h$ such that
$$
\forall n \in\NN, \quad \| \psi^H_{\overline{k}^n, k_\eps} \|_{H^1(D_c)} \leq C.
$$
Since $\psi^H_{\overline{k}^n, k_\eps}$ belongs to the finite dimensional space $W_{H,h}^{\rm enrich}$, we deduce that there exists a subsequence, which we still denote by $\psi^H_{\overline{k}^n, k_\eps}$, that converges in $W_{H,h}^{\rm enrich}$ (when $n \to \infty$) to some $\overline{\psi}^H_{\infty,k_\eps} \in W_{H,h}^{\rm enrich}$.

\subsubsection{Convergence of the minimizing sequence $\overline{k}^n$} \label{sec:k_n}

We now show that the minimizing sequence $\overline{k}^n$ converges to some limit $\overline{k}^\infty$. We introduce $\widetilde{u}_0^H \in V_H^{\rm Dir BC}$ such that 
\begin{equation}\label{eq:contradiction}
  \forall v^H \in V^0_H, \quad \int_D \nabla \widetilde{u}_0^H \cdot \nabla v^H + \frac{1}{2} \int_{D_c} \nabla \widetilde{u}_0^H \cdot \nabla v^H = 0.
\end{equation}
Taking $\overline{v}^H = \overline{u}^H_{\overline{k}^n, k_\eps} -\widetilde{u}_0^H \in V^0_H$ in the first line of~\eqref{eq:arlequin_n}, we have 
\begin{multline*}
  \int_D \overline{k}^n \nabla \overline{u}^H_{\overline{k}^n, k_\eps} \cdot \nabla \left( \overline{u}^H_{\overline{k}^n, k_\eps} - \widetilde{u}_0^H \right) + \frac{1}{2} \int_{D_c} \overline{k}^n \nabla \overline{u}^H_{\overline{k}^n, k_\eps} \cdot \nabla \left( \overline{u}^H_{\overline{k}^n, k_\eps} - \widetilde{u}_0^H \right)
  \\
  + \int_{D_c} \nabla \psi^H_{\overline{k}^n, k_\eps} \cdot \nabla \left( \overline{u}^H_{\overline{k}^n, k_\eps} - \widetilde{u}_0^H \right) + \int_{D_c} \psi^H_{\overline{k}^n, k_\eps} \left( \overline{u}^H_{\overline{k}^n, k_\eps} - \widetilde{u}_0^H \right) = 0.
\end{multline*}
Thus, from the definition of $\widetilde{u}_0^H$ and the fact that $\overline{k}^n$ is a scalar, we obtain
\begin{multline}\label{eq:k_infty}
  \overline{k}^n \left\| \nabla \left( \overline{u}^H_{\overline{k}^n, k_\eps} - \widetilde{u}_0^H \right) \right\|^2_{L^2(D)} + \frac{1}{2} \overline{k}^n \left\| \nabla \left( \overline{u}^H_{\overline{k}^n, k_\eps} - \widetilde{u}_0^H \right) \right\|^2_{L^2(D_c)}
  \\
  + \int_{D_c} \nabla \psi^H_{\overline{k}^n, k_\eps} \cdot \nabla \left( \overline{u}^H_{\overline{k}^n, k_\eps} - \widetilde{u}_0^H \right) + \int_{D_c} \psi^H_{\overline{k}^n, k_\eps} \left( \overline{u}^H_{\overline{k}^n, k_\eps} - \widetilde{u}_0^H \right) = 0.    
\end{multline}
All the terms in~\eqref{eq:k_infty} converge when $n \to \infty$, {\em except} possibly $\overline{k}^n$. The only case when we cannot deduce from~\eqref{eq:k_infty} that $\overline{k}^n$ converges is that when the limit $\overline{u}^H_{\infty,k_\eps}$ of $\overline{u}^H_{\overline{k}^n, k_\eps}$ identically satisfies $\nabla \overline{u}^H_{\infty,k_\eps} = \nabla \widetilde{u}_0^H$ in $D \cup D_c$. Since $\overline{u}_{\infty,k_\eps}^H(x) = \widetilde{u}_0^H(x) = x_1$ on $\Gamma$, this would imply that $\overline{u}^H_{\infty,k_\eps} = \widetilde{u}_0^H$ in $D \cup D_c$. Passing to the limit $n \to \infty$ in~\eqref{eq:quasiminimizers} yields 
$$
I_{\eps, H, h} = \int_{D \cup D_c} \left| \nabla \overline{u}_{\infty,k_\eps}^H - e_1 \right|^2.
$$
We are left with showing the condition 
\begin{equation}\label{eq:J_contradiction}
  I_{\eps, H, h} < \int_{D \cup D_c} \left| \nabla \widetilde{u}_0^H - e_1 \right|^2,
\end{equation}
if we want to rule out this case and conclude that $\overline{k}^n$ converges up to an extraction (to some coefficient that we denote $\overline{k}^\infty$). We recall that $\widetilde{u}_0^H \in V_H^{\rm Dir BC}$ in~\eqref{eq:J_contradiction} is defined by~\eqref{eq:contradiction}.

\medskip

We note that the right-hand side of~\eqref{eq:J_contradiction} is positive (and of course independent of $\eps$ and $h$ by construction). Indeed, if $\dps \int_{D \cup D_c} \left| \nabla \widetilde{u}^H_0 - e_1 \right|^2$ were vanishing, we would have $\widetilde{u}^H_0(x) = x_1$ on $D \cup D_c$ (recall $\widetilde{u}^H_0(x) = x_1$ on $\Gamma$), which however does not satisfy~\eqref{eq:contradiction}.

\medskip

Investigating whether~\eqref{eq:J_contradiction} holds in full generality is delicate, and this is why we have {\em assumed} this condition in Theorem~\ref{th:optimization}. It can be investigated {\em numerically}. There are also a few situations where~\eqref{eq:J_contradiction} can be established {\em mathematically}. One such case is when we suppose that the fine mesh parameter $h$ and the oscillating parameter $\eps$ are sufficiently small. We indeed claim that 
\begin{equation}\label{eq:limit_contradiction}
  \lim_{\eps \to 0} \lim_{h \to 0} J_{\eps, H, h}(k^\star) = 0,
\end{equation}
where $J_{\eps, H, h}$ is defined by~\eqref{eq:def_J}, which obviously implies 
\begin{equation}\label{eq:limit_contradiction_I}
  \lim_{\eps \to 0} \lim_{h \to 0} I_{\eps, H, h} = 0,
\end{equation}
and thus~\eqref{eq:J_contradiction} (in the regime $h \ll \eps \ll 1$) since we have pointed out above that the right-hand side of~\eqref{eq:J_contradiction} is positive and independent of $\eps$ and $h$.

\medskip

In order to prove~\eqref{eq:limit_contradiction}, we consider~\eqref{eq:arlequin} with $\overline{k} = k^\star$, the solution of which is denoted $(\overline{u}^{H,h}_{k^\star,k_\eps},\widecheck{u}^h_{\eps,k^\star,k_\eps},\psi^{H,h}_{k^\star,k_\eps})$ (where we have on purpose made explicit the dependency of the three components of the solution with respect to $h$). Using standard finite element arguments, we can pass to the limit $h \to 0$. We thus have $\dps (\overline{u}^{H,h}_{k^\star,k_\eps},\widecheck{u}^h_{\eps,k^\star,k_\eps},\psi^{H,h}_{k^\star,k_\eps}) \xrightarrow[h \to 0]{} (\overline{u}^H_{k^\star,k_\eps},\widecheck{u}_{\eps,k^\star,k_\eps},\psi^H_{k^\star,k_\eps})$ in $H^1(D \cup D_c) \times H^1(D_c \cup D_f) \times H^1(D_c)$, where $(\overline{u}^H_{k^\star,k_\eps},\widecheck{u}_{\eps,k^\star,k_\eps},\psi^H_{k^\star,k_\eps})$ is the solution to~\eqref{eq:system_k_star} with $\overline{k} = k^\star$. Furthermore, $(\overline{u}^H_{k^\star,k_\eps},\widecheck{u}_{\eps,k^\star,k_\eps},\psi^H_{k^\star,k_\eps})$ is bounded in $H^1(D \cup D_c) \times H^1(D_c \cup D_f) \times H^1(D_c)$ by a constant independent of $\eps$ and $H$. This bound on $\overline{u}^H_{k^\star,k_\eps}$ has indeed been shown above (see~\eqref{eq:u_k_bound}), and it implies a bound on $\widecheck{u}_{\eps,k^\star,k_\eps}$ and $\psi^H_{k^\star,k_\eps}$ using the same arguments as those used in Section~\ref{sec:bounds}. 

We now refer to Lemma~\ref{th:multiplication} for $\overline{k} = k^\star$ (the bounds that we have just discussed obviously implying the convergences stated as assumptions in that lemma) and obtain that, when $\eps\to 0$, the solution $(\overline{u}^H_{k^\star,k_\eps},\widecheck{u}_{\eps,k^\star,k_\eps},\psi^H_{k^\star,k_\eps})$ to~\eqref{eq:system_k_star} converges to $(\overline{u}^H_{k^\star,k^\star},\widecheck{u}_{k^\star,k^\star},\psi^H_{k^\star,k^\star})$, solution to~\eqref{eq:system_contradiction} with $\overline{k} = k^\star$. We eventually note (as stated in the first assertion of Lemma~\ref{lemma:optimization}) that the unique solution to the system~\eqref{eq:system_contradiction} with $\overline{k} = k^\star$ is $\overline{u}^H_{k^\star,k^\star}(x) = x_1$, $\widecheck{u}_{k^\star,k^\star}(x) = x_1$ and $\psi^H_{k^\star,k^\star} = k^\star \, \psi_0$.

We have therefore shown that $\overline{u}^{H,h}_{k^\star,k_\eps}$ converges (when $h \to 0$ and $\eps \to 0$) to $\overline{u}^H_{k^\star,k^\star}(x) = x_1$, strongly in $H^1(D \cup D_c)$ since $\overline{u}^{H,h}_{k^\star,k_\eps}$ belongs to the finite dimensional space $V_H^{\rm Dir BC}$. We have therefore obtained that 
$$
\lim_{\eps \to 0} \lim_{h \to 0} J_{\eps, H, h}(k^\star) = \int_{D \cup D_c} \big| \nabla \overline{u}^H_{k^\star,k^\star} - e_1 \big|^2 = 0,
$$
that is exactly~\eqref{eq:limit_contradiction}. 

\subsubsection{Existence of an optimal coefficient} \label{sec:non_0}

We have shown above that, under assumption~\eqref{eq:J_contradiction}, the minimizing sequence $\overline{k}^n \in (0,\infty)$ converges (up to a subsequence extraction) to some $\overline{k}^\infty \in [0,\infty)$. We now show that $\overline{k}^\infty \neq 0$, provided we impose an additional condition (see~\eqref{eq:J_contradiction2} below) to the problem. 

To state that additional condition, we consider~\eqref{eq:arlequin} and we formally set $\overline{k} = 0$. We hence look for $\overline{u}^H_{0,k_\eps} \in V_H^{\rm Dir BC}$, $\widecheck{u}^h_{\eps,0,k_\eps} \in V_h$ and $\psi^H_{0,k_\eps} \in W_{H,h}^{\rm enrich}$ such that
\begin{equation} \label{eq:contradiction_arlequin}
  \begin{cases}
  \forall \overline{v}^H \in V_H^0, \quad & {\cal C}(\overline{v}^H,\psi^H_{0,k_\eps}) = 0,
  \\ \noalign{\vskip 3pt}
  \forall \widecheck{v}^h \in V_h, \quad & \widecheck{A}_{k_\eps}(\widecheck{u}^h_{\eps,0,k_\eps},\widecheck{v}^h) - {\cal C}(\widecheck{v}^h,\psi_{0,k_\eps}^H) = 0,
  \\ \noalign{\vskip 3pt}
  \forall \phi^H \in W_{H,h}^{\rm enrich}, \quad & {\cal C}(\overline{u}^H_{0,k_\eps} - \widecheck{u}^h_{\eps,0,k_\eps},\phi^H) = 0,
  \end{cases}
\end{equation}
which is equivalent to solving the minimization problem
\begin{equation} \label{eq:pb_min_0}
  \inf \left\{ \begin{array}{c} {\cal E}_0(\widecheck{u}^h_\eps), \quad \overline{u}^H \in V_H, \quad \overline{u}^H(x) = x_1 \ \text{on $\Gamma$}, \\ \noalign{\vskip 2pt} \widecheck{u}^h_\eps \in V_h, \qquad {\cal C}(\overline{u}^H-\widecheck{u}^h_\eps,\phi^H) = 0 \ \ \text{for any $\phi^H \in W_{H,h}^{\rm enrich}$} \end{array} \right\},
\end{equation}
where the constraint function ${\cal C}$ is defined by~\eqref{eq:def_C} and where the energy ${\cal E}_0$ is obtained from the energy ${\cal E}$ defined in~\eqref{eq:def_E} by formally setting $\overline{k} = 0$:
\begin{equation}\label{eq:def_E0}
  {\cal E}_0(\widecheck{u}_\eps) = \frac{1}{2} \int_{D_f} k_\eps(x) \, \nabla \widecheck{u}_\eps(x) \cdot \nabla \widecheck{u}_\eps(x) + \frac{1}{4} \int_{D_c} k_\eps(x) \, \nabla \widecheck{u}_\eps(x) \cdot \nabla \widecheck{u}_\eps(x).
\end{equation}
Note that ${\cal E}_0$ does not depend on $\overline{u}^H$.

The minimizers of~\eqref{eq:pb_min_0} are simple to characterize: they satisfy $\widecheck{u}^h_{\eps,0,k_\eps} = \lambda$ in $D_c \cup D_f$ for some constant $\lambda$ (which indeed minimizes the energy~\eqref{eq:def_E0}) and $\overline{u}^H_{0,k_\eps} = \lambda$ in $D_c$ (which is obtained by considering $\phi^H = (\overline{u}^H_{0,k_\eps} - \lambda) \big|_{D_c}$ in the constraint). The value of $\overline{u}^H_{0,k_\eps}$ in $D$ is free, besides the fact that it should satisfy the boundary condition $\overline{u}^H_{0,k_\eps}(x) = x_1$ on $\Gamma$ and the trace condition $\overline{u}^H_{0,k_\eps} = \lambda$ on $\Gamma_c$.

It is next easy to see that, for any solution to~\eqref{eq:contradiction_arlequin}, we have
\begin{align*}
  \int_{D \cup D_c} \left| \nabla \overline{u}_{0,k_\eps}^H - e_1 \right|^2
  &=
  \int_D \left| \nabla \overline{u}_{0,k_\eps}^H - e_1 \right|^2 + \int_{D_c} \left| \nabla \overline{u}_{0,k_\eps}^H - e_1 \right|^2
  \\
  &=
  \int_D \left| \nabla \overline{u}_{0,k_\eps}^H - e_1 \right|^2 + |D_c|.
\end{align*}
We are now going to compute the minimum of the above quantity over all solutions to~\eqref{eq:contradiction_arlequin}.

Let us introduce the unique function $\widetilde{u}_{0,a}^H \in V_H^{\rm Dir BC}$ satisfying $\widetilde{u}^H_{0,a} = 0$ in $D_c$ and 
\begin{equation}\label{eq:covid4a}
  \forall v^H \in \widetilde{V}_H, \quad \int_D \left( \nabla \widetilde{u}_{0,a}^H - e_1 \right) \cdot \nabla v^H = 0,
\end{equation}
where
$$
\widetilde{V}_H = \left\{ v \in V^0_H, \quad \text{$v=0$ on $D_c$} \right\}.
$$
Let us also introduce the unique function $\widetilde{u}_{0,b}^H \in V_H^0$ satisfying $\widetilde{u}^H_{0,b} = 1$ in $D_c$ and 
\begin{equation}\label{eq:covid4b}
  \forall v^H \in \widetilde{V}_H, \quad \int_D \left( \nabla \widetilde{u}_{0,b}^H - e_1 \right) \cdot \nabla v^H = 0.
\end{equation}
We then have
$$
\underset{\small \begin{array}{c} \text{solutions to~\eqref{eq:contradiction_arlequin}} \\ \text{which are equal to $\lambda$ on $D_c$} \end{array}}{\inf} \int_D \left| \nabla \overline{u}_{0,k_\eps}^H - e_1 \right|^2 \\ = \int_D \left| \nabla \widetilde{u}_{0,a}^H + \lambda \, \nabla \widetilde{u}_{0,b}^H - e_1 \right|^2.
$$
Next, by minimizing with respect to $\lambda$, we obtain
\begin{equation}\label{eq:covid5}
  \inf_{\text{solutions to~\eqref{eq:contradiction_arlequin}}} \int_{D \cup D_c} \left| \nabla \overline{u}_{0,k_\eps}^H - e_1 \right|^2 = \widetilde{I}_{0,H},
\end{equation}
where $\widetilde{I}_{0,H}$ is defined in terms of the solution $\widetilde{u}_{0,a}^H$ to~\eqref{eq:covid4a} and $\widetilde{u}_{0,b}^H$ to~\eqref{eq:covid4b} by
$$
\widetilde{I}_{0,H} = |D_c| + \int_D \left| \nabla \widetilde{u}_{0,a}^H - e_1 \right|^2 - \frac{\left( \int_D \nabla \widetilde{u}_{0,b}^H \cdot \left( \nabla \widetilde{u}_{0,a}^H - e_1 \right) \right)^2}{\int_D \left| \nabla \widetilde{u}_{0,b}^H \right|^2}.
$$
Note that the sum of the last two terms of $\widetilde{I}_{0,H}$ is non-negative, in view of the Cauchy-Schwarz inequality, and hence $\widetilde{I}_{0,H} \geq |D_c| > 0$.

\medskip

We now assume that 
\begin{equation}\label{eq:J_contradiction2}
  I_{\eps, H, h} < \widetilde{I}_{0,H}.
\end{equation}
Under that assumption, we claim that the limit $\overline{k}^\infty$ of the minimizing sequence satisfies $\overline{k}^\infty \neq 0$. We argue by contradiction and assume that $\overline{k}^\infty = 0$. Taking the limit $n \to \infty$ in~\eqref{eq:arlequin_n}, we thus have that the limit $(\overline{u}^H_{\infty, k_\eps},\widecheck{u}^h_{\eps,\infty, k_\eps}, \psi^H_{\infty, k_\eps})$ of $(\overline{u}^H_{\overline{k}^n, k_\eps},\widecheck{u}^h_{\eps,\overline{k}^n, k_\eps},\psi^H_{\overline{k}^n, k_\eps})$ is a solution to~\eqref{eq:contradiction_arlequin}. Passing to the limit $n \to \infty$ in~\eqref{eq:quasiminimizers} yields 
$$
I_{\eps, H, h} = \int_{D \cup D_c} \left| \nabla \overline{u}_{\infty,k_\eps}^H - e_1 \right|^2,
$$
which is in contradiction with~\eqref{eq:covid5} and~\eqref{eq:J_contradiction2}. This proves that $\overline{k}^\infty > 0$.

\medskip

One of the ways to establish~\eqref{eq:J_contradiction2} mathematically is to assume that the parameters $h$ and $\eps$ are sufficiently small. We indeed recall that, in that regime, the left-hand side of~\eqref{eq:J_contradiction2} converges to 0 (see~\eqref{eq:limit_contradiction_I}). In contrast, the right-hand side of~\eqref{eq:J_contradiction2} is independent of $h$ and $\eps$ and is positive.

\medskip

We have thus shown that, under assumptions~\eqref{eq:J_contradiction} and~\eqref{eq:J_contradiction2}, the minimizing sequence $\overline{k}^n$ converges (up to a subsequence extraction) to some $\overline{k}^\infty \in (0,\infty)$. Since the function $\overline{k} \mapsto J_{\eps, H, h}(\overline{k})$ is continuous on $(0,\infty)$, this shows that $\overline{k}^\infty$ is a minimizer of~\eqref{eq:optim_J}. We denote by $\overline{k}^{\rm opt}_\eps$ such an optimal coefficient, to emphasize its dependency with respect to $\eps$. This concludes the proof of Theorem~\ref{th:optimization}.

\subsection{Homogenized limit}\label{sec:homogenization}

For each $\eps>0$, and under assumptions~\eqref{eq:J_contradiction} and~\eqref{eq:J_contradiction2}, we know from Theorem~\ref{th:optimization} that there exists at least one optimal constant coefficient $\overline{k}^{\rm opt}_\eps$ minimizing~\eqref{eq:optim_J} with a corresponding solution $(\overline{u}^H_{\overline{k}^{\rm opt}_\eps, k_\eps},\widecheck{u}^h_{\eps,\overline{k}^{\rm opt}_\eps, k_\eps},\psi^H_{\overline{k}^{\rm opt}_\eps,k_\eps})$ to the system~\eqref{eq:arlequin} for $\overline{k} = \overline{k}^{\rm opt}_\eps$. We now aim at studying the limit of $\overline{k}^{\rm opt}_\eps$ when $\eps \to 0$. We recall that, in that limit, assumptions~\eqref{eq:J_contradiction} and~\eqref{eq:J_contradiction2} are satisfied. 

\medskip

We have assumed in~\eqref{eq:structure-k} that the sequence $k_\eps$ is such that $k_\eps = k_{\rm per}(\cdot/\eps)$ for some fixed periodic function $k_{\rm per}$. This periodicity assumption implies that the homogenized coefficient $k^\star$ exists and is constant (a fact that we have already used above, see e.g. the first arguments of Section~\ref{sec:bounds}). Our aim in this section is to show that the optimal coefficient $\overline{k}^{\rm opt}_\eps$ converges to the homogenized coefficient $k^\star$ when $\eps$ goes to $0$, as stated in the following theorem, which is our second main result. Although we perform our analysis in the periodic setting, we believe that it actually carries over to more general cases (random stationary setting, \dots).

\begin{theorem}\label{th:homogenization}
Let $k_\eps$ be given by~\eqref{eq:structure-k} for some fixed periodic coefficient $k_{\rm per}$ that satisfies the classical boundedness and coercivity conditions~\eqref{eq:boundedness+coercivity}. We make the regularity assumption~\eqref{eq:holder_continuity} and the geometric assumption~\eqref{eq:boundaries}. 

Then, any optimal coefficient $\overline{k}_\eps^{\rm opt}(H,h)$ (the existence of which is provided by Theorem~\ref{th:optimization}) converges to the homogenized coefficient $k^\star$ when $h$ and $\eps$ go to $0$: for any $H>0$, we have
\begin{equation} \label{eq:main_result2}
\lim_{\eps \to 0} \lim_{h \to 0} \overline{k}_\eps^{\rm opt}(H,h) = k^\star.
\end{equation}
\end{theorem}

We discuss in Remark~\ref{rem:H_independant} below the fact that $H$ is kept fixed in~\eqref{eq:main_result2}. We do not need to finally take the limit $H \to 0$ to recover $k^\star$. This is a clear advantage from the computational viewpoint, since this property allows to work with values of $H$ that are not asymptotically small (see~\cite[Section~3.1]{ref_olga_comp} for some numerical results).  

\medskip

Before proceeding, we note that considering the regime $\eps \to 0$ implies that we also have $h \to 0$ (since $h$ has to be chosen much smaller than $\eps$). For simplicity and brevity of exposition, we therefore \emph{fix} $h = 0$, and point out that Theorem~\ref{th:optimization} still holds true. Taking the limit $h \to 0$ is just an additional, technical ingredient. We spare the reader with this unnecessary technicality.

We thus consider the following ``partially'' discretized system (which is~\eqref{eq:system_k_star} with $\overline{k} = \overline{k}^{\rm opt}_\eps$): find $\overline{u}^H_{\overline{k}^{\rm opt}_\eps,k_\eps} \in V_H^{\rm Dir BC}$, $\widecheck{u}_{\eps,\overline{k}^{\rm opt}_\eps,k_\eps} \in H^1(D_c \cup D_f)$ and $\psi^H_{\overline{k}^{\rm opt}_\eps,k_\eps} \in W_H^{\rm enrich}$ such that
\begin{equation} \label{eq:arlequin2}
  \begin{cases}
  \forall \overline{v}^H \in V_H^0, \quad & \overline{A}_{\overline{k}^{\rm opt}_\eps}(\overline{u}^H_{\overline{k}^{\rm opt}_\eps,k_\eps},\overline{v}^H) + {\cal C}(\overline{v}^H,\psi^H_{\overline{k}^{\rm opt}_\eps,k_\eps}) = 0,
  \\ \noalign{\vskip 3pt}
  \forall \widecheck{v} \in H^1(D_c \cup D_f), \quad & \widecheck{A}_{k_\eps}(\widecheck{u}_{\eps,\overline{k}^{\rm opt}_\eps,k_\eps},\widecheck{v}) - {\cal C}(\widecheck{v},\psi^H_{\overline{k}^{\rm opt}_\eps,k_\eps}) = 0,
  \\ \noalign{\vskip 3pt}
  \forall \phi^H \in W_H^{\rm enrich}, \quad & {\cal C}(\overline{u}^H_{\overline{k}^{\rm opt}_\eps, k_\eps} - \widecheck{u}_{\eps,\overline{k}^{\rm opt}_\eps, k_\eps},\phi^H) = 0.
  \end{cases}
\end{equation}
In the above system and in the proof below, we do not make explicit in the notation the fact that the optimal coefficient $\overline{k}^{\rm opt}_\eps$ depends on $H$ (in particular because $H$ is kept fixed and we do not need to take the limit $H \to 0$ to recover $k^\star$). 

\medskip

The proof of Theorem~\ref{th:homogenization} falls in two steps:
\begin{itemize}
\item as in Section~\ref{sec:bounds}, we first establish some bounds independent of $\eps$ on the solution $(\overline{u}^H_{\overline{k}^{\rm opt}_\eps, k_\eps},\widecheck{u}_{\eps,\overline{k}^{\rm opt}_\eps, k_\eps},\psi^H_{\overline{k}^{\rm opt}_\eps, k_\eps}$) to~\eqref{eq:arlequin2} (see Section~\ref{sec:bounds_eps}).
  \item we next pass to the limit $\eps \to 0$ in~\eqref{eq:arlequin2} (see Section~\ref{sec:k_opt}).
\end{itemize}

\subsubsection{Bounds on $\overline{u}^H_{\overline{k}^{\rm opt}_\eps, k_\eps}$, $\widecheck{u}_{\eps,\overline{k}^{\rm opt}_\eps, k_\eps}$ and $\psi^H_{\overline{k}^{\rm opt}_\eps, k_\eps}$} \label{sec:bounds_eps}

We pass to the limit $n \to \infty$ in~\eqref{eq:u_k^n_bound} (where we recall that the constant $C$ is in particular independent of $n$, $h$ and $\eps$), using that $\overline{k}^n \xrightarrow[n \to \infty]{} \overline{k}^{\rm opt}_\eps$ (which implies that $\overline{u}^H_{\overline{k}^n,k_\eps} \xrightarrow[n \to \infty]{} \overline{u}^H_{\overline{k}^{\rm opt}_\eps,k_\eps}$), and immediately obtain that the sequence $\overline{u}^H_{\overline{k}^{\rm opt}_\eps, k_\eps} \in V_H^{\rm Dir BC}$ is bounded in $H^1(D \cup D_c)$, independently of $\eps$.

We next simply repeat the steps of Section~\ref{sec:bounds} and obtain that the sequence $\widecheck{u}_{\eps,\overline{k}^{\rm opt}_\eps, k_\eps}$ (resp. the sequence $\psi^H_{\overline{k}^{\rm opt}_\eps, k_\eps} \in W_H^{\rm enrich}$) is bounded in $H^1(D_c \cup D_f)$ (resp. in $H^1(D_c)$), independently of $\eps$.

We thus deduce that there exist $\overline{u}^H_\star \in V_H^{\rm Dir BC}$, $\widecheck{u}_\star \in H^1(D_c \cup D_f)$ and $\psi^H_\star \in W_H^{\rm enrich}$ such that, up to the extraction of a subsequence, $\overline{u}^H_{\overline{k}^{\rm opt}_\eps, k_\eps}$ converges to $\overline{u}^H_\star$ in $H^1(D \cup D_c)$, $\widecheck{u}_{\eps,\overline{k}^{\rm opt}_\eps, k_\eps}$ weakly converges to $\widecheck{u}_\star$ in $H^1(D_c \cup D_f)$ and $\psi^H_{\overline{k}^{\rm opt}_\eps, k_\eps}$ converges to $\psi^H_\star$ in $H^1(D_c)$ when $\eps \to 0$.

\subsubsection{Limit system with $\overline{k}^{\rm opt}_\eps$} \label{sec:k_opt}

Passing to the limit $\eps\to 0$ in the first line of~\eqref{eq:arlequin2} is not straightforward since, unfortunately, we have little information on $\overline{k}^{\rm opt}_\eps$ so far. Since we do not seem to have an obvious bound independent of $\eps$ on the coefficient $\overline{k}^{\rm opt}_\eps$, we circumvent this difficulty as follows.

We have
$$
\lim_{h \to 0} I_{\eps,H,h} = \int_{D \cup D_c} \left| \nabla \overline{u}^H_{\overline{k}^{\rm opt}_\eps,k_\eps} - e_1 \right|^2,
$$
where $(\overline{u}^H_{\overline{k}^{\rm opt}_\eps, k_\eps},\widecheck{u}_{\eps,\overline{k}^{\rm opt}_\eps, k_\eps},\psi^H_{\overline{k}^{\rm opt}_\eps, k_\eps})$ is the solution to~\eqref{eq:arlequin2}. Passing to the limit $\eps \to 0$ and using~\eqref{eq:limit_contradiction_I}, we deduce that
$$
0 = \lim_{\eps \to 0} \lim_{h \to 0} I_{\eps,H,h} = \int_{D \cup D_c} \left| \nabla \overline{u}^H_\star - e_1 \right|^2
$$
where $\overline{u}^H_\star \in V_H^{\rm Dir BC}$ is the limit of $\overline{u}^H_{\overline{k}^{\rm opt}_\eps,k_\eps}$. This implies that
\begin{equation}\label{eq:lim_u_k^opt_bis}
  \overline{u}^H_\star(x) = x_1 \quad \text{in $D \cup D_c$}.
\end{equation}

We are now in position to show that the sequence $\overline{k}_\eps^{\rm opt}$ converges to some limit $\overline{k}^{\rm opt}_0$ when $\eps \to 0$. Repeating the steps from the beginning of Section~\ref{sec:k_n}, we take $\overline{v}^H = \overline{u}^H_{\overline{k}_\eps^{\rm opt},k_\eps} -\widetilde{u}_0^H \in V^0_H$ in the first line of~\eqref{eq:arlequin2}, where $\widetilde{u}_0^H \in V_H^{\rm Dir BC}$ is defined by~\eqref{eq:contradiction}. Using that $\overline{k}_\eps^{\rm opt}$ is a scalar, we thus obtain
\begin{multline}\label{eq:k_infty_hom}
  \overline{k}_\eps^{\rm opt} \left\| \nabla \left(\overline{u}^H_{\overline{k}_\eps^{\rm opt}, k_\eps} - \widetilde{u}_0^H \right) \right\|^2_{L^2(D)} + \frac{1}{2} \overline{k}_\eps^{\rm opt} \left\| \nabla \left( \overline{u}^H_{\overline{k}_\eps^{\rm opt}, k_\eps} - \widetilde{u}_0^H \right) \right\|^2_{L^2(D_c)}
  \\
  + \int_{D_c} \nabla \psi^H_{\overline{k}_\eps^{\rm opt}, k_\eps} \cdot \nabla \left( \overline{u}^H_{\overline{k}_\eps^{\rm opt}, k_\eps} - \widetilde{u}_0^H \right) + \int_{D_c} \psi^H_{\overline{k}_\eps^{\rm opt}, k_\eps} \left( \overline{u}^H_{\overline{k}_\eps^{\rm opt}, k_\eps} - \widetilde{u}_0^H \right) = 0.    
\end{multline}
All the terms in~\eqref{eq:k_infty_hom} converge when $\eps\to 0$, {\em except} possibly $\overline{k}^{\rm opt}_\eps$. The only case when we cannot deduce from~\eqref{eq:k_infty_hom} that $\overline{k}_\eps^{\rm opt}$ converges is that when the limit $\overline{u}^H_\star$ of $\overline{u}^H_{\overline{k}_\eps^{\rm opt},k_\eps}$ satisfies $\nabla \overline{u}^H_\star = \nabla \widetilde{u}_0^H$ in $D \cup D_c$. Since $\overline{u}^H_\star$ and $\widetilde{u}_0^H$ satisfy the same boundary condition on $\Gamma$, this would imply that $\overline{u}^H_\star = \widetilde{u}_0^H$ in $D \cup D_c$, and thus, in view of~\eqref{eq:lim_u_k^opt_bis}, that $\widetilde{u}_0^H(x) = x_1$ in $D \cup D_c$. We have already pointed out in Section~\ref{sec:k_n} that $\widetilde{u}_0^H$ is not equal to the function $x_1$ in $D \cup D_c$. We have thus obtained a contradiction, which shows that the sequence $\overline{k}^{\rm opt}_\eps$ converges (up to a subsequence extraction) to some coefficient $\overline{k}^{\rm opt}_0 \in [0,\infty)$ when $\eps \to 0$. It remains to prove that this coefficient $\overline{k}^{\rm opt}_0$ is equal to the homogenized coefficient $k^\star$. This will prove that the optimization problem~\eqref{eq:optim_J} indeed provides an approximation of the homogenized coefficient, and additionally that the {\em whole} sequence $\overline{k}^{\rm opt}_\eps$ (and not only a subsequence) converges.

\medskip

To show that $\overline{k}^{\rm opt}_0 = k^\star$, we pass to the limit $\eps \to 0$ in~\eqref{eq:arlequin2}. Using Corollary~\ref{coro:multiplication} (as established in Section~\ref{sec:bounds_eps}, the sequence $(\overline{u}^H_{\overline{k}^{\rm opt}_\eps, k_\eps},\widecheck{u}_{\eps,\overline{k}^{\rm opt}_\eps, k_\eps},\psi^H_{\overline{k}^{\rm opt}_\eps, k_\eps})$ indeed satisfies the appropriate convergence properties stated as assumptions in that corollary),
we observe that the limit $(\overline{u}^H_\star,\widecheck{u}_\star,\psi^H_\star)$ of $(\overline{u}^H_{\overline{k}^{\rm opt}_\eps, k_\eps},\widecheck{u}_{\eps,\overline{k}^{\rm opt}_\eps, k_\eps},\psi^H_{\overline{k}^{\rm opt}_\eps, k_\eps})$ is actually the solution (or {\em a solution}, if $\overline{k}^{\rm opt}_0=0$) to the following system: find $\overline{u}^H_{\overline{k}^{\rm opt}_0,k^\star} \in V_H^{\rm Dir BC}$, $\widecheck{u}_{\overline{k}^{\rm opt}_0,k^\star} \in H^1(D_c \cup D_f)$ and $\psi^H_{\overline{k}^{\rm opt}_0,k^\star} \in W_H^{\rm enrich}$ such that
\begin{equation} \label{eq:arlequin_limit2}
  \begin{cases}
  \forall \overline{v}^H \in V_H^0, \quad & \overline{A}_{\overline{k}^{\rm opt}_0}(\overline{u}^H_{\overline{k}^{\rm opt}_0, k^\star},\overline{v}^H) + {\cal C}(\overline{v}^H,\psi^H_{\overline{k}^{\rm opt}_0, k^\star}) = 0,
  \\ \noalign{\vskip 3pt}
  \forall \widecheck{v} \in H^1(D_c \cup D_f), \quad & \widecheck{A}_{k^\star}(\widecheck{u}_{\overline{k}^{\rm opt}_0, k^\star},\widecheck{v}) - {\cal C}(\widecheck{v},\psi^H_{\overline{k}^{\rm opt}_0, k^\star}) = 0,
  \\ \noalign{\vskip 3pt}
  \forall \phi^H \in W_H^{\rm enrich}, \quad & {\cal C}(\overline{u}^H_{\overline{k}^{\rm opt}_0, k^\star} - \widecheck{u}_{\overline{k}^{\rm opt}_0, k^\star},\phi^H) = 0.
  \end{cases}
\end{equation}
In addition, we know from~\eqref{eq:lim_u_k^opt_bis} that $\overline{u}^H_{\overline{k}^{\rm opt}_0,k^\star}(x) = \overline{u}^H_\star(x) = x_1$ in $D \cup D_c$. 

We now use the second assertion of Lemma~\ref{lemma:optimization} and obtain that $\overline{k}^{\rm opt}_0 = k^\star$ and $\psi^H_{\overline{k}^{\rm opt}_0, k^\star} = k^\star \, \psi_0$, which concludes the proof of Theorem~\ref{th:homogenization}. 

\begin{remark} \label{rem:H_independant}
We note that, in Theorem~\ref{th:homogenization}, the parameter $H$ is fixed and that it does not affect the value of the limit of $\overline{k}^{\rm opt}_\eps$. This is directly related to the consistency of the approach as discussed in Remark~\ref{rem:consistency}: for any $H>0$, the fact that $\overline{u}^H_{\overline{k},k^\star}(x) = x_1$ implies that $\overline{k} = k^\star$.
\end{remark}

\subsection{Uniqueness of the optimal coefficient for sufficiently small $\eps$} \label{sec:uniqueness}

Now that we have shown that any minimizer $\overline{k}^{\rm opt}_\eps(H,h)$ converges to $k^\star$ when $h$ and $\eps$ go to 0 (in the sense of Theorem~\ref{th:homogenization}), we are in position to show the uniqueness of the optimal coefficient when $\eps$ is sufficiently small. As in Section~\ref{sec:homogenization}, we hereafter fix $h=0$, and thus consider $J_{\eps,H}$ defined (compare with~\eqref{eq:def_J} and~\eqref{eq:arlequin}) by
\begin{equation} \label{eq:def_J_no-h}
J_{\eps, H}(\overline{k}) = \int_{D \cup D_c} \big| \nabla \overline{u}^H_{\overline{k}, k_\eps} - \nabla u_{\rm ref} \big|^2 = \int_{D \cup D_c} \big| \nabla \overline{u}^H_{\overline{k}, k_\eps} - e_1 \big|^2,
\end{equation}
where $(\overline{u}^H_{\overline{k}, k_\eps},\widecheck{u}_{\eps,\overline{k}, k_\eps}, \psi^H_{\overline{k}, k_\eps})$ is the solution to~\eqref{eq:system_k_star}. In the sequel, we thus consider $\overline{k}_\eps^{\rm opt}(H)$, solution to the minimization problem
\begin{equation} \label{eq:optim_J_no-h}
\inf \left\{ J_{\eps, H}(\overline{k}), \quad \overline{k} \in (0,\infty) \right\}.
\end{equation}
Our proof is based on the following two results.

\begin{lemma}[uniform convergence] \label{lem:conv_unif}
  Let $H>0$ be fixed. The minimizers $\overline{k}_\eps^{\rm opt}(H)$ of~\eqref{eq:optim_J_no-h} converge {\em uniformly} to $k^\star$, in the following sense: for any $\eta > 0$, there exists some $\eps_0(\eta,H)$ such that, for any $\eps \leq \eps_0(\eta,H)$, {\em any} minimizer $\overline{k}_\eps^{\rm opt}(H)$ of~\eqref{eq:optim_J_no-h} satisfies $| \overline{k}_\eps^{\rm opt}(H) - k^\star | \leq \eta$.
\end{lemma}

\begin{proof}
Since $\overline{k}_\eps^{\rm opt}$ is a minimizer of $J_{\eps, H}$, we have, using that $k^\star$ is an admissible test coefficient in~\eqref{eq:optim_J}, that
$$
\int_{D \cup D_c} \big| \nabla \overline{u}^H_{\overline{k}_\eps^{\rm opt}, k_\eps} - \nabla u_{\rm ref} \big|^2 = J_{\eps, H}(\overline{k}_\eps^{\rm opt}) \leq J_{\eps, H}(k^\star).
$$
Using that $\overline{u}^H_{\overline{k}_\eps^{\rm opt}, k_\eps} - u_{\rm ref}$ vanishes on $\Gamma$, we deduce that
\begin{equation} \label{eq:nounou3}
\| \overline{u}^H_{\overline{k}_\eps^{\rm opt}, k_\eps} - u_{\rm ref} \|^2_{H^1(D \cup D_c)} \leq C \, J_{\eps, H}(k^\star),
\end{equation}
where $C$ only depends on $D \cup D_c$. Recalling that $\dps \lim_{\eps \to 0} J_{\eps, H}(k^\star) = 0$ (see~\eqref{eq:limit_contradiction}), we deduce that $\overline{u}^H_{\overline{k}_\eps^{\rm opt}, k_\eps}$ converges to $u_{\rm ref}$ in a uniform (with respect to the choice of the minimizer $\overline{k}_\eps^{\rm opt}$) manner.

\medskip

We are now going to show, and this is the main part of the proof, that $\nabla \overline{u}^H_{\overline{k}^{\rm opt}_\eps, k_\eps} / \overline{k}_\eps^{\rm opt}$ converges to $\nabla u_{\rm ref} / k^\star$ uniformly with respect to the choice of the minimizer $\overline{k}_\eps^{\rm opt}$. Turning first to the Lagrange multiplier, and recalling that $\overline{k}_\eps^{\rm opt} > 0$, we can introduce
$$
\theta^H_{\overline{k}_\eps^{\rm opt}, k_\eps} = \frac{\psi^H_{\overline{k}_\eps^{\rm opt}, k_\eps}}{\overline{k}_\eps^{\rm opt}},
$$
and, after dividing by $\overline{k}_\eps^{\rm opt}$, we rewrite the first line of~\eqref{eq:arlequin2} in the form
\begin{equation} \label{eq:nounou}
\forall \overline{v}^H \in V_H^0, \quad \overline{A}_1(\overline{u}^H_{\overline{k}^{\rm opt}_\eps,k_\eps},\overline{v}^H) + {\cal C}(\overline{v}^H,\theta^H_{\overline{k}^{\rm opt}_\eps,k_\eps}) = 0,
\end{equation}
where we recall that the bilinear form $\overline{A}_1$ is defined by~\eqref{eq:def_overline_A_1}. Since $\theta^H_{\overline{k}^{\rm opt}_\eps,k_\eps} \in W_H^{\rm enrich}$, we can represent it as 
\begin{equation}\label{eq:lagrange_multiplier_expansion_eps}
  \theta^H_{\overline{k}^{\rm opt}_\eps,k_\eps} = \tau_\eps \, \big(\psi_0 - \Pi_H(\psi_0) \big) + \widetilde{\theta}^H_\eps,
\end{equation} 
for some $\tau_\eps \in \RR$ (we keep implicit the dependency of $\tau_\eps$ with respect to $H$) and some $\widetilde{\theta}^H_\eps \in W_H$ (which both depend on the choice of the minimizer $\overline{k}^{\rm opt}_\eps$), where we recall that the projection operator $\Pi_H$ is defined by~\eqref{eq:Pih_a}.
We have shown at the end of Section~\ref{sec:k_opt} that $\theta^H_{\overline{k}^{\rm opt}_\eps,k_\eps}$ converges to $\psi_0$. We thus expect $\tau_\eps$ to converge to 1 and $\widetilde{\theta}^H_\eps$ to converge to $\Pi_H(\psi_0)$. This is indeed the case since~\eqref{eq:lagrange_multiplier_expansion_eps} implies that
$$
{\cal C}\big(\theta^H_{\overline{k}^{\rm opt}_\eps,k_\eps},\psi_0 - \Pi_H(\psi_0)\big) = \tau_\eps \, \| \Pi_H(\psi_0) - \psi_0 \|^2_{H^1(D_c)},
$$
and thus, passing to the limit $\eps \to 0$, we obtain
$$
\lim_{\eps \to 0} \tau_\eps = \frac{{\cal C}\big(\psi_0,\psi_0 - \Pi_H(\psi_0)\big)}{\| \Pi_H(\psi_0) - \psi_0 \|^2_{H^1(D_c)}} = 1.
$$
We now establish a bound on $\widetilde{\theta}^H_\eps$. Inserting~\eqref{eq:lagrange_multiplier_expansion_eps} in~\eqref{eq:nounou} and using the orthogonality of any $\overline{v}^H \in V_H^0$ with $\psi_0 - \Pi_H(\psi_0)$, we obtain, for any $\overline{v}^H \in V_H^0$, that
$$
{\cal C}(\overline{v}^H,\widetilde{\theta}^H_\eps)
=
- \overline{A}_1(\overline{u}^H_{\overline{k}^{\rm opt}_\eps,k_\eps},\overline{v}^H)
=
\overline{A}_1(u_{\rm ref} - \overline{u}^H_{\overline{k}^{\rm opt}_\eps,k_\eps},\overline{v}^H) + {\cal C}(\overline{v}^H,\psi_0),
$$
where we have used the variational formulation~\eqref{eq:LMvariational} satisfied by $\psi_0$. We thus deduce that
\begin{equation} \label{eq:nounou2}
\forall \overline{v}^H \in V_H^0, \quad {\cal C}\big(\overline{v}^H,\widetilde{\theta}^H_\eps-\Pi_H(\psi_0)\big) = \overline{A}_1(u_{\rm ref} - \overline{u}^H_{\overline{k}^{\rm opt}_\eps,k_\eps},\overline{v}^H).
\end{equation}
Using arguments similar to those used in Section~\ref{sec:bounds}, we can extend the function $\widetilde{\theta}^H_\eps - \Pi_H(\psi_0)$, which is only defined in $D_c$, over the domain $D$, and thus introduce some $\widetilde{v}^H_\eps$ defined in $D \cup D_c$, satisfying $\widetilde{v}^H_\eps \in V_H^0$, $\widetilde{v}^H_\eps = \widetilde{\theta}^H_\eps - \Pi_H(\psi_0)$ in $D_c$ and $\| \widetilde{v}^H_\eps \|_{H^1(D \cup D_c)} \leq C \| \widetilde{\theta}^H_\eps - \Pi_H(\psi_0) \|_{H^1(D_c)}$ for some constant $C$ independent of $H$ and $\eps$ (and of the choice of the minimizer). 
Taking $\overline{v}^H = \widetilde{v}^H_\eps$ in~\eqref{eq:nounou2}, we thus infer that
\begin{multline*}
\| \widetilde{\theta}^H_\eps - \Pi_H(\psi_0) \|^2_{H^1(D_c)}
=
{\cal C}\big(\widetilde{v}^H_\eps,\widetilde{\theta}^H_\eps - \Pi_H(\psi_0)\big)
=
\overline{A}_1(u_{\rm ref} - \overline{u}^H_{\overline{k}^{\rm opt}_\eps,k_\eps},\widetilde{v}^H_\eps)
\\
\leq
\| u_{\rm ref} - \overline{u}^H_{\overline{k}^{\rm opt}_\eps,k_\eps} \|_{H^1(D \cup D_c)} \, \| \widetilde{v}^H_\eps \|_{H^1(D \cup D_c)},
\end{multline*}
and hence, using~\eqref{eq:nounou3},
\begin{equation}
  \| \widetilde{\theta}^H_\eps - \Pi_H(\psi_0) \|_{H^1(D_c)}
  \leq
  C \| u_{\rm ref} - \overline{u}^H_{\overline{k}^{\rm opt}_\eps,k_\eps} \|_{H^1(D \cup D_c)}
  \leq
  C \sqrt{J_{\eps, H}(k^\star)},
  \label{eq:nounou4}
\end{equation}
for some $C$ independent of $\eps$ and of the choice of the minimizer $\overline{k}_\eps^{\rm opt}$.

\medskip

We now write the second line of~\eqref{eq:arlequin2}. After dividing by $\overline{k}_\eps^{\rm opt}$, it reads
\begin{multline*}
\forall \widecheck{v} \in H^1(D_c \cup D_f), \quad \widecheck{A}_{k_\eps}\left(\frac{\widecheck{u}_{\eps,\overline{k}^{\rm opt}_\eps,k_\eps}}{\overline{k}^{\rm opt}_\eps},\widecheck{v}\right) = {\cal C}(\widecheck{v},\theta^H_{\overline{k}^{\rm opt}_\eps,k_\eps}) \\ = \tau_\eps \, {\cal C}\big(\widecheck{v},\psi_0-\Pi_H(\psi_0)\big) + {\cal C}(\widecheck{v},\widetilde{\theta}^H_\eps).
\end{multline*}
We thus introduce $\widecheck{u}_{\eps,1}$ and $\widecheck{u}_{\eps,2}$, the unique solutions (with vanishing mean in $D_c$) in $H^1(D_c \cup D_f)$ to the problems
\begin{align}
  \forall \widecheck{v} \in H^1(D_c \cup D_f), \quad \widecheck{A}_{k_\eps}(\widecheck{u}_{\eps,1},\widecheck{v}) &= {\cal C}\big(\widecheck{v},\psi_0-\Pi_H(\psi_0)\big),
  \label{eq:nounou5bis}
  \\
  \forall \widecheck{v} \in H^1(D_c \cup D_f), \quad \widecheck{A}_{k_\eps}(\widecheck{u}_{\eps,2},\widecheck{v}) &= {\cal C}(\widecheck{v},\widetilde{\theta}^H_\eps),
  \label{eq:nounou5}
\end{align}
and of course have
\begin{equation} \label{eq:nounou7}
\frac{\widecheck{u}_{\eps,\overline{k}^{\rm opt}_\eps,k_\eps}}{\overline{k}^{\rm opt}_\eps} = \lambda_\eps + \tau_\eps \, \widecheck{u}_{\eps,1} + \widecheck{u}_{\eps,2}
\end{equation}
for some constant $\lambda_\eps$.

To study $\widecheck{u}_{\eps,2}$, we introduce the unique solution $\widecheck{u}_{\eps,3}$ (with vanishing mean in $D_c$) in $H^1(D_c \cup D_f)$ to
\begin{equation} \label{eq:nounou8}
 \forall \widecheck{v} \in H^1(D_c \cup D_f), \quad \widecheck{A}_{k_\eps}(\widecheck{u}_{\eps,3},\widecheck{v}) = {\cal C}\big(\widecheck{v},\Pi_H(\psi_0)\big).
\end{equation}
Subtracting~\eqref{eq:nounou8} from~\eqref{eq:nounou5}, we have
$$
\forall \widecheck{v} \in H^1(D_c \cup D_f), \quad \widecheck{A}_{k_\eps}(\widecheck{u}_{\eps,2}-\widecheck{u}_{\eps,3},\widecheck{v}) = {\cal C}\big(\widecheck{v},\widetilde{\theta}^H_\eps-\Pi_H(\psi_0)\big).
$$
Taking $\widecheck{v} = \widecheck{u}_{\eps,2} - \widecheck{u}_{\eps,3}$ in the above equation and using that $k_\eps$ is bounded from below uniformly in $\eps$, a Poincar\'e-Wirtinger inequality in $H^1(D_c \cup D_f)$ and~\eqref{eq:nounou4}, we obtain that
\begin{equation} \label{eq:nounou6}
  \| \widecheck{u}_{\eps,2}-\widecheck{u}_{\eps,3} \|_{H^1(D_c \cup D_f)} \leq C \| \widetilde{\theta}^H_\eps-\Pi_H(\psi_0) \|_{H^1(D_c)} \leq C \sqrt{J_{\eps, H}(k^\star)},
\end{equation}
for some $C$ independent of $\eps$ and of the choice of the minimizer $\overline{k}_\eps^{\rm opt}$. 

\medskip

We now turn to the third line of~\eqref{eq:arlequin2}. After dividing by $\overline{k}_\eps^{\rm opt}$ and using~\eqref{eq:nounou7}, we have that, for any $\phi^H \in W_H^{\rm enrich}$,
$$
{\cal C}\left(\frac{\overline{u}^H_{\overline{k}^{\rm opt}_\eps, k_\eps}}{\overline{k}_\eps^{\rm opt}} - \left[ \lambda_\eps + \tau_\eps \, \widecheck{u}_{\eps,1} + \widecheck{u}_{\eps,2} \right],\phi^H\right) = 0,
$$
and hence
\begin{align*}
  & {\cal C}\left(\frac{\overline{u}^H_{\overline{k}^{\rm opt}_\eps, k_\eps}}{\overline{k}_\eps^{\rm opt}} - \frac{u_{\rm ref}}{k^\star} - \lambda_\eps - (\tau_\eps-1) \, \widecheck{u}_{\eps,1},\phi^H\right)
  \\
  &=
  {\cal C}\left(\widecheck{u}_{\eps,1} + \widecheck{u}_{\eps,2} - \frac{u_{\rm ref}}{k^\star},\phi^H\right)
  \\
  &=
  {\cal C}\left(\widecheck{u}_{\eps,1} + \widecheck{u}_{\eps,3} - \frac{u_{\rm ref}}{k^\star},\phi^H\right) + {\cal C}\left(\widecheck{u}_{\eps,2} - \widecheck{u}_{\eps,3},\phi^H\right).
\end{align*}
Using the projection operator $\Pi_H^{\rm enrich}$ defined by~\eqref{eq:Pih_b}, we obtain that, for any $\phi^H \in W_H^{\rm enrich}$,
\begin{multline} \label{eq:nounou9}
{\cal C}\left(\frac{\overline{u}^H_{\overline{k}^{\rm opt}_\eps, k_\eps}}{\overline{k}_\eps^{\rm opt}} - \frac{u_{\rm ref}}{k^\star} - \lambda_\eps - (\tau_\eps-1) \, \Pi_H^{\rm enrich}(\widecheck{u}_{\eps,1}),\phi^H\right)
\\
=
{\cal C}\left(\Pi_H^{\rm enrich}(\widecheck{u}_{\eps,4}) - \frac{u_{\rm ref}}{k^\star},\phi^H\right) + {\cal C}\left(\widecheck{u}_{\eps,2} - \widecheck{u}_{\eps,3},\phi^H\right),
\end{multline}
where we have introduced $\widecheck{u}_{\eps,4} = \widecheck{u}_{\eps,1} + \widecheck{u}_{\eps,3}$.

\medskip

Let us now bound the right-hand side of~\eqref{eq:nounou9}. The second term can be bounded using~\eqref{eq:nounou6}. For the first term, we note that $\widecheck{u}_{\eps,4}$ is the unique solution (with vanishing mean in $D_c$) in $H^1(D_c \cup D_f)$ to
\begin{equation} \label{eq:nounou10}
 \forall \widecheck{v} \in H^1(D_c \cup D_f), \quad \widecheck{A}_{k_\eps}(\widecheck{u}_{\eps,4},\widecheck{v}) = {\cal C}(\widecheck{v},\psi_0).
\end{equation}
The function $\widecheck{u}_{\eps,4}$ is therefore independent from the choice of the minimizer $\overline{k}_\eps^{\rm opt}$. In addition, the homogenized limit of the Neumann problem~\eqref{eq:nounou10} reads as follows: when $\eps \to 0$, $\widecheck{u}_{\eps,4}$ converges (strongly in $L^2(D_c \cup D_f)$ and weakly in $H^1(D_c \cup D_f)$) to $\widecheck{u}_{\star,4}$, the unique solution (with vanishing mean in $D_c$) in $H^1(D_c \cup D_f)$ to
\begin{equation} \label{eq:nounou11}
\forall \widecheck{v} \in H^1(D_c \cup D_f), \quad \widecheck{A}_{k^\star}(\widecheck{u}_{\star,4},\widecheck{v}) = {\cal C}(\widecheck{v},\psi_0).
\end{equation}
We hence deduce that $\Pi_H^{\rm enrich}(\widecheck{u}_{\eps,4})$ converges (strongly in $H^1(D_c \cup D_f)$, since $\Pi_H^{\rm enrich}(\widecheck{u}_{\eps,4})$ belongs to the finite dimensional space $W_H^{\rm enrich}$) to $\Pi_H^{\rm enrich}(\widecheck{u}_{\star,4})$. We next observe that, by definition of $\psi_0$, we can solve~\eqref{eq:nounou11} and we know that $\widecheck{u}_{\star,4}(x) = u_{\rm ref}(x)/k^\star = x_1/k^\star$. We hence have that $\Pi_H^{\rm enrich}(\widecheck{u}_{\star,4}) = \widecheck{u}_{\star,4} = u_{\rm ref}/k^\star$. We thus have shown that
\begin{equation} \label{eq:nounou13}
r_{\eps,H} = \left\| \Pi_H^{\rm enrich}(\widecheck{u}_{\eps,4}) - \frac{u_{\rm ref}}{k^\star} \right\|_{H^1(D_c \cup D_f)} \underset{\eps \to 0}{\to} 0,
\end{equation}
with a rate of convergence (with respect to $\eps$) independent of the choice of the minimizer $\overline{k}_\eps^{\rm opt}$. 

\medskip

To manipulate the left-hand side of~\eqref{eq:nounou9}, we write that
\begin{equation} \label{eq:nounou12}
\Pi_H^{\rm enrich}(\widecheck{u}_{\eps,1}) = \alpha_\eps \, \big(\psi_0 - \Pi_H(\psi_0) \big) + \widehat{u}^H_{\eps,1},
\end{equation}
for some $\alpha_\eps \in \RR$ (again, we keep implicit the dependency of $\alpha_\eps$ with respect to $H$) and some $\widehat{u}^H_{\eps,1} \in W_H$. Using~\eqref{eq:nounou5bis}, we observe that
\begin{multline} \label{eq:nounou14}
\alpha_\eps \| \Pi_H(\psi_0) - \psi_0 \|^2_{H^1(D_c)} = {\cal C}\big(\Pi_H^{\rm enrich}(\widecheck{u}_{\eps,1}),\psi_0 - \Pi_H(\psi_0)\big) \\ = {\cal C}\big(\widecheck{u}_{\eps,1},\psi_0 - \Pi_H(\psi_0)\big) = \widecheck{A}_{k_\eps}(\widecheck{u}_{\eps,1},\widecheck{u}_{\eps,1}).
\end{multline}

\medskip

We now write~\eqref{eq:nounou9} with $\phi^H = \psi_0 - \Pi_H(\psi_0)$. Observing that $\dps \frac{\overline{u}^H_{\overline{k}^{\rm opt}_\eps, k_\eps}}{\overline{k}_\eps^{\rm opt}} - \frac{u_{\rm ref}}{k^\star} - \lambda_\eps \in W_H$, we deduce that
\begin{multline*}
{\cal C}\Big( (1-\tau_\eps) \, \Pi_H^{\rm enrich}(\widecheck{u}_{\eps,1}),\psi_0 - \Pi_H(\psi_0)\Big)
\\
=
{\cal C}\left(\Pi_H^{\rm enrich}(\widecheck{u}_{\eps,4}) - \frac{u_{\rm ref}}{k^\star},\psi_0 - \Pi_H(\psi_0)\right) + {\cal C}\big(\widecheck{u}_{\eps,2} - \widecheck{u}_{\eps,3},\psi_0 - \Pi_H(\psi_0)\big).
\end{multline*}
Using~\eqref{eq:nounou12} in the left-hand side and~\eqref{eq:nounou13} and~\eqref{eq:nounou6} in the right-hand side, we obtain
$$
|1-\tau_\eps| \, |\alpha_\eps| \, \| \Pi_H(\psi_0) - \psi_0 \|^2_{H^1(D_c)} \leq \left( r_{\eps,H} + C \sqrt{J_{\eps, H}(k^\star)} \right) \| \Pi_H(\psi_0) - \psi_0 \|_{H^1(D_c)},
$$
and hence, using~\eqref{eq:nounou14},
\begin{equation} \label{eq:nounou17}
|1-\tau_\eps| \leq \left( r_{\eps,H} + C \sqrt{J_{\eps, H}(k^\star)} \right) \ \frac{\| \Pi_H(\psi_0) - \psi_0 \|_{H^1(D_c)}}{\widecheck{A}_{k_\eps}(\widecheck{u}_{\eps,1},\widecheck{u}_{\eps,1})}.
\end{equation}
We note that $\widecheck{A}_{k_\eps}(\widecheck{u}_{\eps,1},\widecheck{u}_{\eps,1}) \neq 0$. Indeed, if this quantity vanishes, then $\widecheck{u}_{\eps,1}$ is a constant, which implies, using~\eqref{eq:nounou5bis}, that ${\cal C}\big(\widecheck{v},\psi_0-\Pi_H(\psi_0)\big) = 0$ for any $\widecheck{v} \in H^1(D_c \cup D_f)$, and hence $\psi_0 = \Pi_H(\psi_0)$, which is not true.

We eventually note that, in view of its definition~\eqref{eq:nounou5bis}, $\widecheck{u}_{\eps,1}$ does not depend on the choice of the minimizer $\overline{k}_\eps^{\rm opt}$. When $\eps \to 0$, the quantity $\widecheck{A}_{k_\eps}(\widecheck{u}_{\eps,1},\widecheck{u}_{\eps,1})$ converges (with a rate independent of the choice of the minimizer $\overline{k}_\eps^{\rm opt}$) to its homogenized limit $\widecheck{A}_{k^\star}(\widecheck{u}_{\star,1},\widecheck{u}_{\star,1}) \neq 0$. In addition, $r_{\eps,H}$ and $J_{\eps, H}(k^\star)$ both converge to zero with rates independent of the choice of the minimizer $\overline{k}_\eps^{\rm opt}$. We thus deduce from~\eqref{eq:nounou17} that
$$
\lim_{\eps \to 0} \tau_\eps = 1 \quad \text{uniformly with respect to the choice of the minimizer $\overline{k}_\eps^{\rm opt}$}.
$$

\medskip

Taking now $\dps \phi^H = \frac{\overline{u}^H_{\overline{k}^{\rm opt}_\eps, k_\eps}}{\overline{k}_\eps^{\rm opt}} - \frac{u_{\rm ref}}{k^\star} - \lambda_\eps - (\tau_\eps-1) \, \Pi_H^{\rm enrich}(\widecheck{u}_{\eps,1})$ in~\eqref{eq:nounou9} and using~\eqref{eq:nounou13} and~\eqref{eq:nounou6} to bound the right-hand side, we obtain
$$
\left\| \frac{\overline{u}^H_{\overline{k}^{\rm opt}_\eps, k_\eps}}{\overline{k}_\eps^{\rm opt}} - \frac{u_{\rm ref}}{k^\star} - \lambda_\eps - (\tau_\eps-1) \, \Pi_H^{\rm enrich}(\widecheck{u}_{\eps,1}) \right\|_{H^1(D_c)} \leq r_{\eps,H} + C \sqrt{J_{\eps, H}(k^\star)},
$$
and hence
$$
\left\| \nabla \left( \frac{\overline{u}^H_{\overline{k}^{\rm opt}_\eps, k_\eps}}{\overline{k}_\eps^{\rm opt}} - \frac{u_{\rm ref}}{k^\star} \right) \right\|_{L^2(D_c)} \leq r_{\eps,H} + C \sqrt{J_{\eps, H}(k^\star)} + |\tau_\eps-1| \, \left\| \Pi_H^{\rm enrich}(\widecheck{u}_{\eps,1}) \right\|_{H^1(D_c)},
$$
which implies that $\dps \frac{\nabla \overline{u}^H_{\overline{k}^{\rm opt}_\eps, k_\eps}}{\overline{k}_\eps^{\rm opt}}$ converges to $\dps \frac{\nabla u_{\rm ref}}{k^\star}$ uniformly with respect to the choice of the minimizer $\overline{k}_\eps^{\rm opt}$. We thus have a similar uniform convergence of $\dps \frac{\| \nabla \overline{u}^H_{\overline{k}^{\rm opt}_\eps, k_\eps} \|_{L^2(D_c)}}{\overline{k}_\eps^{\rm opt}}$ to $\dps \frac{\| \nabla u_{\rm ref} \|_{L^2(D_c)}}{k^\star}$. In view of~\eqref{eq:nounou3}, we also have a uniform convergence of $\| \nabla \overline{u}^H_{\overline{k}^{\rm opt}_\eps, k_\eps} \|_{L^2(D_c)}$ to $\| \nabla u_{\rm ref} \|_{L^2(D_c)}$. These two properties imply a uniform convergence of $\overline{k}_\eps^{\rm opt}$ to $k^\star$. This concludes the proof of Lemma~\ref{lem:conv_unif}.
\end{proof}

\begin{lemma} \label{lem:bound_der_J}
Consider $J_{\eps,H}$ defined by~\eqref{eq:def_J_no-h}, and consider some neighborhood ${\cal K} = (c_-,c_+)$ of $k^\star$ with $c_->0$. Then, for any derivation order $m \geq 0$, there exists $C_m$ such that, for any $\eps$, any $H>0$ and any $\overline{k} \in {\cal K}$,
\begin{equation} \label{eq:bound_J_m}
\left| \frac{\partial^m J_{\eps,H}}{\partial \overline{k}^m}(\overline{k}) \right| \leq C_m.
\end{equation}
\end{lemma}

\begin{proof}
We first show~\eqref{eq:bound_J_m} for $m=0$, i.e. on the function $J_{\eps,H}$ itself. We recall that~\eqref{eq:system_k_star} is the Euler-Lagrange equation of the following minimization problem (which is exactly~\eqref{eq:pb_min_H_enriched} with $h=0$):
$$
\inf \left\{ \begin{array}{c} {\cal E}(\overline{u}^H,\widecheck{u}_\eps), \quad \overline{u}^H \in V_H, \quad \overline{u}^H(x) = x_1 \ \text{on $\Gamma$}, \\ \noalign{\vskip 2pt} \widecheck{u}_\eps \in H^1(D_c \cup D_f), \qquad {\cal C}(\overline{u}^H-\widecheck{u}_\eps,\phi^H) = 0 \ \ \text{for any $\phi^H \in W_H^{\rm enrich}$} \end{array} \right\},
$$
where the energy ${\cal E}$ and the constraint function ${\cal C}$ are defined by~\eqref{eq:def_E} and~\eqref{eq:def_C}. Making the choice $\overline{u}^H(x) = x_1$ in $D \cup D_c$ and $\widecheck{u}_\eps(x) = x_1$ in $D_c \cup D_f$, we obtain
$$
{\cal E}(\overline{u}^H_{\overline{k}, k_\eps},\widecheck{u}_{\eps,\overline{k}, k_\eps}) \leq {\cal E}\left(\overline{u}^H(x) = x_1, \widecheck{u}_\eps(x) = x_1\right),
$$
and therefore
\begin{multline*} 
\int_D \overline{k} \, \left| \nabla \overline{u}^H_{\overline{k}, k_\eps}\right|^2 + \frac{1}{2} \int_{D_c} \Big( \overline{k} \, \left|\nabla \overline{u}^H_{\overline{k}, k_\eps}\right|^2 + k_\eps \, \nabla \widecheck{u}_{\eps, \overline{k}, k_\eps} \cdot \nabla \widecheck{u}_{\eps, \overline{k}, k_\eps} \Big) \\ + \int_{D_f} k_\eps \, \nabla \widecheck{u}_{\eps, \overline{k}, k_\eps} \cdot \nabla \widecheck{u}_{\eps, \overline{k}, k_\eps} \leq \overline{k} \, | D | + \frac{\overline{k}}{2} \, | D_c | + \frac{1}{2} \int_{D_c} k_\eps e_1 \cdot e_1 + \int_{D_f} k_\eps e_1 \cdot e_1.
\end{multline*}
Using that $k_\eps$ is uniformly bounded and coercive and that we have chosen $\overline{k} \in {\cal K}$, we obtain
\begin{equation}\label{eq:bound_m_0}
\int_{D \cup D_c} \, |\nabla \overline{u}^H_{\overline{k}, k_\eps}|^2 \leq C_0,
\end{equation}
for some constant $C_0$ independent of $\eps$, $H$ and $\overline{k} \in {\cal K}$. This directly implies~\eqref{eq:bound_J_m} for $m=0$.

\medskip

We next show~\eqref{eq:bound_J_m} for $m=1$. We begin by writing that
\begin{equation} \label{eq:der_J_1}
\frac{\partial J_{\eps, H}}{\partial \overline{k}}(\overline{k}) = 2 \int_{D \cup D_c} \big( \nabla \overline{u}^H_{\overline{k}, k_\eps} - e_1 \big) \cdot \nabla \overline{u}^{H,1}_{\overline{k}, k_\eps}
\end{equation}
with $\dps \overline{u}^{H,1}_{\overline{k}, k_\eps} = \frac{\partial \overline{u}^H_{\overline{k}, k_\eps}}{\partial \overline{k}}$. To compute this derivative, we recall that $(\overline{u}^H_{\overline{k}, k_\eps},\widecheck{u}_{\eps,\overline{k}, k_\eps},\psi^H_{\overline{k}, k_\eps}) \in V_H^{\rm Dir BC} \times H^1(D_c \cup D_f) \times W_H^{\rm enrich}$ is the unique solution to the variational formulation~\eqref{eq:system_k_star}. Introducing $\dps \widecheck{u}^1_{\eps,\overline{k}, k_\eps} = \frac{\partial \widecheck{u}_{\eps,\overline{k}, k_\eps}}{\partial \overline{k}}$ and $\dps \psi^{H,1}_{\overline{k}, k_\eps} = \frac{\partial \psi^H_{\overline{k}, k_\eps}}{\partial \overline{k}}$ and differentiating~\eqref{eq:system_k_star} with respect to $\overline{k}$, we obtain that $\overline{u}^{H,1}_{\overline{k}, k_\eps} \in V_H^0$, $\widecheck{u}^1_{\eps,\overline{k}, k_\eps} \in H^1(D_c \cup D_f)$ and $\psi^{H,1}_{\overline{k}, k_\eps} \in W_H^{\rm enrich}$ are such that
\begin{equation}\label{eq:system_k_star_m_1}
\begin{cases}
  \forall \overline{v}^H \in V_H^0, \quad & \overline{A}_{\overline{k}}(\overline{u}^{H,1}_{\overline{k}, k_\eps},\overline{v}^H) + {\cal C}(\overline{v}^H,\psi^{H,1}_{\overline{k}, k_\eps}) = - \overline{A}_1(\overline{u}^H_{\overline{k}, k_\eps},\overline{v}^H),
  \\ \noalign{\vskip 3pt}
  \forall \widecheck{v} \in H^1(D_c \cup D_f), \quad & \widecheck{A}_{k_\eps}(\widecheck{u}^1_{\eps,\overline{k}, k_\eps},\widecheck{v}) - {\cal C}(\widecheck{v},\psi^{H,1}_{\overline{k}, k_\eps}) = 0,
  \\ \noalign{\vskip 3pt}
  \forall \phi^H \in W_H^{\rm enrich}, \quad & {\cal C}(\overline{u}^{H,1}_{\overline{k}, k_\eps} - \widecheck{u}^1_{\eps,\overline{k}, k_\eps},\phi^H) = 0,
\end{cases}
\end{equation}
where the bilinear form $\overline{A}_1$, defined by~\eqref{eq:def_overline_A_1}, reads
$$
\overline{A}_1(\overline{u},\overline{v}) = \int_D \nabla \overline{u}(x) \cdot \nabla \overline{v}(x) + \frac{1}{2} \int_{D_c} \nabla \overline{u}(x) \cdot \nabla \overline{v}(x).
$$
Taking $\overline{v}^H = \overline{u}^{H,1}_{\overline{k}, k_\eps}$ and $\widecheck{v} = \widecheck{u}^1_{\eps,\overline{k}, k_\eps}$ in~\eqref{eq:system_k_star_m_1}, adding the first two lines and using the third line with $\phi^H = \psi^{H,1}_{\overline{k}, k_\eps}$, we get
$$
\overline{A}_{\overline{k}}(\overline{u}^{H,1}_{\overline{k}, k_\eps},\overline{u}^{H,1}_{\overline{k}, k_\eps}) + \widecheck{A}_{k_\eps}(\widecheck{u}^1_{\eps,\overline{k}, k_\eps},\widecheck{u}^1_{\eps,\overline{k}, k_\eps}) = - \overline{A}_1(\overline{u}^H_{\overline{k}, k_\eps},\overline{u}^{H,1}_{\overline{k}, k_\eps}),
$$  
and hence
$$
\frac{\overline{k}}{2} \, \| \nabla \overline{u}^{H,1}_{\overline{k}, k_\eps} \|^2_{L^2(D \cup D_c)} \leq \| \nabla \overline{u}^H_{\overline{k}, k_\eps} \|_{L^2(D \cup D_c)} \, \| \nabla \overline{u}^{H,1}_{\overline{k}, k_\eps} \|_{L^2(D \cup D_c)},
$$
which implies
\begin{equation} \label{eq:bound_1_0}
\| \nabla \overline{u}^{H,1}_{\overline{k}, k_\eps} \|_{L^2(D \cup D_c)} \leq \frac{2}{\overline{k}} \, \| \nabla \overline{u}^H_{\overline{k}, k_\eps} \|_{L^2(D \cup D_c)}.
\end{equation}
Collecting this bound with~\eqref{eq:bound_m_0}, we deduce that
\begin{equation}\label{eq:bound_m_1}
\| \nabla \overline{u}^{H,1}_{\overline{k}, k_\eps} \|_{L^2(D \cup D_c)} \leq C_1,
\end{equation}
for some constant $C_1$ independent of $\eps$, $H$ and $\overline{k} \in {\cal K}$. Collecting~\eqref{eq:der_J_1}, \eqref{eq:bound_m_0} and~\eqref{eq:bound_m_1}, we infer~\eqref{eq:bound_J_m} for $m=1$.

\medskip

To proceed with higher-order derivatives, we again differentiate~\eqref{eq:system_k_star_m_1} with respect to $\overline{k}$. We observe that $\dps \overline{u}^{H,2}_{\overline{k}, k_\eps} = \frac{\partial^2 \overline{u}^H_{\overline{k}, k_\eps}}{\partial \overline{k}^2} \in V_H^0$, $\dps \widecheck{u}^2_{\eps,\overline{k}, k_\eps} = \frac{\partial^2 \widecheck{u}_{\eps,\overline{k}, k_\eps}}{\partial \overline{k}^2} \in H^1(D_c \cup D_f)$ and $\dps \psi^{H,2}_{\overline{k}, k_\eps} = \frac{\partial^2 \psi^H_{\overline{k}, k_\eps}}{\partial \overline{k}^2} \in W_H^{\rm enrich}$ are such that
$$
\begin{cases}
  \forall \overline{v}^H \in V_H^0, \quad & \overline{A}_{\overline{k}}(\overline{u}^{H,2}_{\overline{k}, k_\eps},\overline{v}^H) + {\cal C}(\overline{v}^H,\psi^{H,2}_{\overline{k}, k_\eps}) = - 2 \, \overline{A}_1(\overline{u}^{H,1}_{\overline{k}, k_\eps},\overline{v}^H),
  \\ \noalign{\vskip 3pt}
  \forall \widecheck{v} \in H^1(D_c \cup D_f), \quad & \widecheck{A}_{k_\eps}(\widecheck{u}^2_{\eps,\overline{k}, k_\eps},\widecheck{v}) - {\cal C}(\widecheck{v},\psi^{H,2}_{\overline{k}, k_\eps}) = 0,
  \\ \noalign{\vskip 3pt}
  \forall \phi^H \in W_H^{\rm enrich}, \quad & {\cal C}(\overline{u}^{H,2}_{\overline{k}, k_\eps} - \widecheck{u}^2_{\eps,\overline{k}, k_\eps},\phi^H) = 0.
\end{cases}
$$
In a similar fashion as~\eqref{eq:bound_1_0}, we deduce that
$$
\| \nabla \overline{u}^{H,2}_{\overline{k}, k_\eps} \|_{L^2(D \cup D_c)} \leq \frac{4}{\overline{k}} \, \| \nabla \overline{u}^{H,1}_{\overline{k}, k_\eps} \|_{L^2(D \cup D_c)},
$$
and thus, in view of~\eqref{eq:bound_m_1}, that
\begin{equation}\label{eq:bound_m_2}
\| \nabla \overline{u}^{H,2}_{\overline{k}, k_\eps} \|_{L^2(D \cup D_c)} \leq C_2,
\end{equation}
for some constant $C_2$ independent of $\eps$, $H$ and $\overline{k} \in {\cal K}$. Diffentiating~\eqref{eq:der_J_1} yields
\begin{equation} \label{eq:der_J_2}
\frac{\partial^2 J_{\eps, H}}{\partial \overline{k}^2}(\overline{k}) = 2 \int_{D \cup D_c} \big( \nabla \overline{u}^H_{\overline{k}, k_\eps} - e_1 \big) \cdot \nabla \overline{u}^{H,2}_{\overline{k}, k_\eps} + 2 \int_{D \cup D_c} \big| \nabla \overline{u}^{H,1}_{\overline{k}, k_\eps} \big|^2.
\end{equation}
Inserting~\eqref{eq:bound_m_0}, \eqref{eq:bound_m_1} and~\eqref{eq:bound_m_2}, we deduce~\eqref{eq:bound_J_m} for $m=2$. We can of course proceed likewise for larger values of $m$. This concludes the proof of Lemma~\ref{lem:bound_der_J}.
\end{proof}

For each $\eps>0$, and under assumptions~\eqref{eq:J_contradiction} and~\eqref{eq:J_contradiction2}, we know from Theorem~\ref{th:optimization} that there exists at least one optimal constant coefficient $\overline{k}^{\rm opt}_\eps$ minimizing~\eqref{eq:optim_J}. In the limit $\eps \to 0$, assumptions~\eqref{eq:J_contradiction} and~\eqref{eq:J_contradiction2} are satisfied, and we have shown in Theorem~\ref{th:homogenization} that this optimal coefficient converges to $k^\star$. We now show a uniqueness result for small enough $\eps$. As above, we directly take the limit $h \to 0$ and consider minimizers $\overline{k}_\eps^{\rm opt}(H)$ of~\eqref{eq:optim_J_no-h}.

\begin{theorem}\label{th:uniqueness}
Let $k_\eps$ be given by~\eqref{eq:structure-k} for some fixed periodic coefficient $k_{\rm per}$ that satisfies the classical boundedness and coercivity conditions~\eqref{eq:boundedness+coercivity}. We make the regularity assumption~\eqref{eq:holder_continuity} and the geometric assumption~\eqref{eq:boundaries}, and again formally set $h=0$. 

Then, for any $H>0$, there exists some $\eps_0(H)>0$ such that, for any $\eps \leq \eps_0(H)$, there exists a {\em unique} optimal coefficient $\overline{k}_\eps^{\rm opt}(H)$ to the optimization problem~\eqref{eq:optim_J_no-h}. 
\end{theorem}

\begin{proof}
We start by choosing some neighborhood ${\cal K} = (c_-,c_+)$ of $k^\star$ with $c_->0$ (say ${\cal K} = (k^\star/2,2 \, k^\star)$) and we know, in view of Lemma~\ref{lem:conv_unif}, that there exists some $\eps_0(H) > 0$ such that, for any $\eps \leq \eps_0(H)$, all the optimal coefficients $\overline{k}_\eps^{\rm opt}(H)$ (the existence of at least one of those has been shown in Theorem~\ref{th:optimization}) belong to ${\cal K}$. To make notations lighter, we keep the dependency of $\overline{k}_\eps^{\rm opt}(H)$ with respect to $H$ implicit, and thus denote these optimal coefficients by $\overline{k}_\eps^{\rm opt}$. The proof falls in three steps.

\medskip

\noindent {\bf Step 1.} In view of~\eqref{eq:der_J_2}, we have
\begin{equation} \label{eq:der_J_2_opt}
\frac{\partial^2 J_{\eps, H}}{\partial \overline{k}^2}(\overline{k}_\eps^{\rm opt}) = 2 \int_{D \cup D_c} \big( \nabla \overline{u}^H_{\overline{k}_\eps^{\rm opt}, k_\eps} - e_1 \big) \cdot \nabla \overline{u}^{H,2}_{\overline{k}_\eps^{\rm opt}, k_\eps} + 2 \int_{D \cup D_c} \big| \nabla \overline{u}^{H,1}_{\overline{k}_\eps^{\rm opt}, k_\eps} \big|^2.
\end{equation}
We are going to pass to the limit $\eps \to 0$ in that equation.

We have shown in Section~\ref{sec:k_opt} that $\overline{u}^H_{\overline{k}_\eps^{\rm opt}, k_\eps}$ converges (strongly in $H^1(D \cup D_c)$) to $\overline{u}^H_\star$ with $\overline{u}^H_\star(x) = x_1$. 

We now study the limit of $\overline{u}^{H,1}_{\overline{k}_\eps^{\rm opt}, k_\eps}$ when $\eps \to 0$. In view of~\eqref{eq:bound_m_1} and the fact that $\overline{u}^{H,1}_{\overline{k}_\eps^{\rm opt}, k_\eps}(x) = 0$ on $\Gamma$, we see that $\overline{u}^{H,1}_{\overline{k}_\eps^{\rm opt}, k_\eps}$ is bounded in $H^1(D \cup D_c)$. In view of~\eqref{eq:system_k_star_m_1} (written for $\overline{k} = \overline{k}_\eps^{\rm opt}$) and proceeding as in Section~\ref{sec:bounds}, we show that the sequence $\widecheck{u}^1_{\eps,\overline{k}^{\rm opt}_\eps, k_\eps}$ (resp. the sequence $\psi^{H,1}_{\overline{k}^{\rm opt}_\eps, k_\eps}$) is bounded in $H^1(D_c \cup D_f)$ (resp. in $H^1(D_c)$), independently of $\eps$.
We are then in position to use Corollary~\ref{coro:multiplication} and obtain that the limit $(\overline{u}^{H,1}_\star,\widecheck{u}^1_\star,\psi^{H,1}_\star)$ of the solution $(\overline{u}^{H,1}_{\overline{k}^{\rm opt}_\eps, k_\eps},\widecheck{u}^1_{\eps,\overline{k}^{\rm opt}_\eps, k_\eps},\psi^{H,1}_{\overline{k}^{\rm opt}_\eps, k_\eps})$ to~\eqref{eq:system_k_star_m_1} is actually the solution to the following system: find $\overline{u}^{H,1}_{k^\star,k^\star} \in V_H^0$, $\widecheck{u}^1_{k^\star,k^\star} \in H^1(D_c \cup D_f)$ and $\psi^{H,1}_{k^\star,k^\star} \in W_H^{\rm enrich}$ such that
\begin{equation} \label{eq:arlequin_limit2_m_1}
  \begin{cases}
  \forall \overline{v}^H \in V_H^0, \quad & \overline{A}_{k^\star}(\overline{u}^{H,1}_{k^\star,k^\star},\overline{v}^H) + {\cal C}(\overline{v}^H,\psi^{H,1}_{k^\star,k^\star}) = - \overline{A}_1(\overline{u}^H_\star,\overline{v}^H),
  \\ \noalign{\vskip 3pt}
  \forall \widecheck{v} \in H^1(D_c \cup D_f), \quad & \widecheck{A}_{k^\star}(\widecheck{u}^1_{k^\star,k^\star},\widecheck{v}) - {\cal C}(\widecheck{v},\psi^{H,1}_{k^\star,k^\star}) = 0,
  \\ \noalign{\vskip 3pt}
  \forall \phi^H \in W_H^{\rm enrich}, \quad & {\cal C}(\overline{u}^{H,1}_{k^\star,k^\star} - \widecheck{u}^1_{k^\star,k^\star},\phi^H) = 0.
  \end{cases}
\end{equation}
This result will be useful for Step~2 below.

Turning next to $\overline{u}^{H,2}_{\overline{k}_\eps^{\rm opt}, k_\eps}$, we infer from~\eqref{eq:bound_m_2} and the fact that $\overline{u}^{H,2}_{\overline{k}_\eps^{\rm opt}, k_\eps}(x) = 0$ on $\Gamma$ that $\overline{u}^{H,2}_{\overline{k}_\eps^{\rm opt}, k_\eps}$ is bounded in $H^1(D \cup D_c)$. Since it belongs to the finite dimensional space $V_H^0$, it converges to some $\overline{u}^{H,2}_\star$.

We are then in position to pass to the limit $\eps \to 0$ in~\eqref{eq:der_J_2_opt}, and obtain that
\begin{align}
  \lim_{\eps \to 0} \frac{\partial^2 J_{\eps, H}}{\partial \overline{k}^2}(\overline{k}_\eps^{\rm opt})
  &=
  2 \int_{D \cup D_c} \big( \nabla \overline{u}^H_\star - e_1 \big) \cdot \nabla \overline{u}^{H,2}_\star + 2 \int_{D \cup D_c} \big| \nabla \overline{u}^{H,1}_{k^\star,k^\star} \big|^2
  \nonumber
  \\
  &=
  2 \int_{D \cup D_c} \big| \nabla \overline{u}^{H,1}_{k^\star,k^\star} \big|^2,
   \label{eq:der_J_2_opt_lim}
\end{align}
where we have used that $\overline{u}^H_\star(x) = x_1$.

\medskip

\noindent {\bf Step 2.} We now show by contradiction that $\nabla \overline{u}^{H,1}_{k^\star,k^\star}$ does not identically vanish in $D \cup D_c$. If this is the case, using that $\overline{u}^{H,1}_{k^\star,k^\star}$ vanishes on $\Gamma$, we obtain that $\overline{u}^{H,1}_{k^\star,k^\star} = 0$ in $D \cup D_c$. We also infer from the first line of~\eqref{eq:arlequin_limit2_m_1} that
$$
\forall \overline{v}^H \in V_H^0, \quad \overline{A}_1(x_1,\overline{v}^H) + {\cal C}(\overline{v}^H,\psi^{H,1}_{k^\star,k^\star}) = 0.
$$
Arguing as in the proof of Lemma~\ref{lemma:optimization} (see~\eqref{eq:first_line} and~\eqref{eq:psi_decomposition}), we get that
$$
\psi^{H,1}_{k^\star,k^\star} = \tau \, \psi_0 + (1 - \tau) \, \Pi_H(\psi_0),
$$
where we recall that the projection $\Pi_H$ is defined by~\eqref{eq:Pih_a} and where the constant $\tau$ will be determined later. The second line of~\eqref{eq:arlequin_limit2_m_1} yields (compare with~\eqref{eq:u_star}) that
$$
\widecheck{u}^1_{k^\star,k^\star} = \lambda + \frac{\tau}{k^\star} \, \widecheck{u}_1 + \frac{1 - \tau}{k^\star} \, \widecheck{u}_2,
$$
where $\lambda \in \RR$ is an arbitrary constant and where we recall that $\widecheck{u}_1$ and $\widecheck{u}_2$ are defined by~\eqref{eq:u1_u2}.

Using $\phi^H = 1$ in the third line of~\eqref{eq:arlequin_limit2_m_1} and that $\overline{u}^{H,1}_{k^\star,k^\star} = 0$ implies that $\lambda = 0$. More generally, the third line of~\eqref{eq:arlequin_limit2_m_1} implies that
$$
\forall \phi^H \in W_H^{\rm enrich}, \quad {\cal C}\left(\frac{\tau}{k^\star} \, \widecheck{u}_1 + \frac{1 - \tau}{k^\star} \, \widecheck{u}_2, \phi^H \right) = 0,
$$
and thus
$$
\Pi^{\rm enrich}_H \left[ \frac{\tau}{k^\star} \, \widecheck{u}_1 + \frac{1 - \tau}{k^\star} \, \widecheck{u}_2 \right] = 0.
$$
Recalling that $\widecheck{u}_1(x) = x_1$ and thus that $\widecheck{u}_1 \in W_H^{\rm enrich}$, we deduce that
$$
\frac{\tau}{k^\star} \, \widecheck{u}_1 + \frac{1 - \tau}{k^\star} \, \Pi^{\rm enrich}_H \left( \widecheck{u}_2 \right) = 0.
$$
We have shown in the proof of Lemma~\ref{lemma:optimization} that $\widecheck{u}_1$ and $\Pi^{\rm enrich}_H \left( \widecheck{u}_2 \right)$ are linearly independent. This implies that $\dps \frac{\tau}{k^\star} = \frac{1 - \tau}{k^\star} = 0$, which leads to a contradiction.

\medskip

\noindent {\bf Step 3.} Writing the Taylor expansion of $J_{\eps, H}$ around one particular minimizer, that we denote $\overline{k}_\eps^{{\rm opt},1}$ (and which belongs to ${\cal K}$, see the beginning of the proof), we obtain that, for any $\overline{k} \in {\cal K}$, there exists some $\widetilde{k} \in {\cal K}$ such that
\begin{multline} \label{eq:nounou15}
J_{\eps, H}(\overline{k}) = J_{\eps, H}(\overline{k}_\eps^{{\rm opt},1}) + \frac{(\overline{k}-\overline{k}_\eps^{{\rm opt},1})^2}{2} \, \frac{\partial^2 J_{\eps,H}}{\partial \overline{k}^2}(\overline{k}_\eps^{{\rm opt},1}) \\ + \frac{(\overline{k}-\overline{k}_\eps^{{\rm opt},1})^3}{6} \, \frac{\partial^3 J_{\eps,H}}{\partial \overline{k}^3}(\widetilde{k}),
\end{multline}
where we have used the fact that $\dps \frac{\partial J_{\eps,H}}{\partial \overline{k}}(\overline{k}_\eps^{{\rm opt},1}) = 0$ since $\overline{k}_\eps^{{\rm opt},1}$ is a minimizer of $J_{\eps,H}$ in the open set $(0,\infty)$.

Using Lemma~\ref{lem:bound_der_J}, we know that there exists $C_3$ such that $\dps \left| \frac{\partial^3 J_{\eps,H}}{\partial \overline{k}^3}(\overline{k}) \right| \leq C_3$ for any $\eps$ and any $\overline{k} \in {\cal K}$. In view of~\eqref{eq:der_J_2_opt_lim} and of the fact that $\nabla \overline{u}^{H,1}_{k^\star,k^\star}$ does not identically vanish in $D \cup D_c$, and up to choosing a smaller value for $\eps_0(H)$ (that we initially chose, we recall, at the beginning of the proof), we know that, for any $\eps \leq \eps_0(H)$, we have 
$$
\frac{\partial^2 J_{\eps, H}}{\partial \overline{k}^2}(\overline{k}_\eps^{{\rm opt},1})
\geq
\int_{D \cup D_c} \big| \nabla \overline{u}^{H,1}_{k^\star,k^\star} \big|^2
> 0.
$$
We now choose some $\eta > 0$ such that $(k^\star-\eta,k^\star+\eta) \subset {\cal K}$ and such that
\begin{equation} \label{eq:nounou16}
\frac{1}{2} \int_{D \cup D_c} \big| \nabla \overline{u}^{H,1}_{k^\star,k^\star} \big|^2 - \frac{\eta}{3} \, C_3 > 0.
\end{equation}
Using again Lemma~\ref{lem:conv_unif}, we obtain that, upon choosing a smaller value for $\eps_0(H)$, all the minimizers of $J_{\eps,H}$ (and thus in particular $\overline{k}_\eps^{{\rm opt},1}$) belong to $(k^\star-\eta,k^\star+\eta)$ for any $\eps \leq \eps_0(H)$. Using~\eqref{eq:nounou15}, we see that, for any $\overline{k} \in (k^\star-\eta,k^\star+\eta)$,
\begin{align*}
  & J_{\eps, H}(\overline{k}) - J_{\eps, H}(\overline{k}_\eps^{{\rm opt},1})
  \\
  &=
  (\overline{k}-\overline{k}_\eps^{{\rm opt},1})^2 \left( \frac{1}{2} \, \frac{\partial^2 J_{\eps,H}}{\partial \overline{k}^2}(\overline{k}_\eps^{{\rm opt},1}) + \frac{\overline{k}-\overline{k}_\eps^{{\rm opt},1}}{6} \, \frac{\partial^3 J_{\eps,H}}{\partial \overline{k}^3}(\widetilde{k}) \right)
  \\
  &\geq
  (\overline{k}-\overline{k}_\eps^{{\rm opt},1})^2 \left( \frac{1}{2} \int_{D \cup D_c} \big| \nabla \overline{u}^{H,1}_{k^\star,k^\star} \big|^2 - \frac{2 \eta}{6} \, C_3 \right),
\end{align*}
where we have used that $|\overline{k}-\overline{k}_\eps^{{\rm opt},1}| \leq |\overline{k}-k^\star| + |k^\star-\overline{k}_\eps^{{\rm opt},1}| \leq 2 \eta$. In view of our choice~\eqref{eq:nounou16} of $\eta$, we obtain that the unique minimizer of $J_{\eps,H}$ in $(k^\star-\eta,k^\star+\eta)$ is $\overline{k}_\eps^{{\rm opt},1}$. We have also pointed out above that {\em all the minimizers} of $J_{\eps,H}$ belong to $(k^\star-\eta,k^\star+\eta)$. This shows that $J_{\eps,H}$ has a unique minimizer when $\eps \leq \eps_0(H)$. This concludes the proof of Theorem~\ref{th:uniqueness}.
\end{proof}

\appendix

\section{Matrix-valued coefficients}\label{seq:matrix}

We now consider the case when the homogenized coefficient $k^\star$, and therefore the constant coefficient $\overline{k}$ upon which we optimize, is matrix-valued and symmetric: $\overline{k} = \begin{bmatrix} \overline{k}_{11} & \overline{k}_{12} \\ \overline{k}_{21} & \overline{k}_{22} \end{bmatrix}$ with $\overline{k}_{12} = \overline{k}_{21}$. Similarly to the scalar-case, we are going to show that the corresponding minimization problem (analogous to~\eqref{eq:optim_J}, and that we introduce below, see~\eqref{eq:optim_J_matrix}) based on the solution to the problem coupling the heterogeneous coefficient $k_\eps$ with the constant coefficient $\overline{k}$ yields an optimal value $\overline{k}^{\rm opt}_\eps$ which itself converges (in some sense made precise below) to $k^\star$ when $\eps \to 0$.

\medskip

This appendix is organized as follows. We first revisit below (and adjust) the definition of the enriched Lagrange multiplier space, in the current matrix-valued case. Next, in Section~\ref{sec:opt_matrix}, we also revisit and adjust the optimization strategy aiming at computing an approximation of the homogenized coefficient $k^\star$ associated to the highly oscillatory coefficient $k_\eps$. The main difference with the scalar case is that we enforce in the optimization problem some a priori lower and upper bounds on $\overline{k}$, for reasons discussed below. In the following Section~\ref{sec:tech_lemmas_matrix}, we establish a result for homogeneous materials which is the equivalent of Lemma~\ref{lemma:optimization} for the matrix-valued case (see Lemma~\ref{lemma:optimization_matrix}). This result is critical to prove that the method indeed converges to the homogenized coefficient. The existence of a minimizer $\overline{k}^{\rm opt}_\eps$ to the optimization problem~\eqref{eq:optim_J_matrix} (for a fixed value of $\eps$) is investigated in Section~\ref{sec:well_matrix} (see Theorem~\ref{th:optimization_matrix}). We eventually consider the limit $\eps \to 0$ and show that some components of the optimal matrix $\overline{k}^{\rm opt}_\eps$ indeed converge to the corresponding components of the homogenized matrix $k^\star$: this is our main result in the matrix-valued case, Theorem~\ref{th:homogenization_matrix} in Section~\ref{sec:homogenization_matrix}. The homogenized matrix $k^\star$ can then be completely determined by considering several optimization problems, as explained in Remark~\ref{matrix_dirichlet_x_2}.

\medskip

We begin this appendix by first highlighting that the enriched space for the Lagrange multiplier should be appropriately defined in this matrix-case. To that aim, we proceed as in the scalar case and consider again the solution $(\overline{u}^H,\widecheck{u}^h_\eps,\psi^H)$ to~\eqref{eq:arlequin0}, assume that $\overline{k} = k^\star$, and take the limit $h \to 0$, $\eps \to 0$ and $H \to 0$. Then the Lagrange multiplier may be explicitly determined (see~\cite[Section~3.1]{ref_olga_comp}): it is the solution to
$$
\begin{cases}
  - \Delta \psi + \psi = 0 \quad \text{in $D_c$},
  \\ \noalign{\vskip 3pt}
  \dps \nabla \psi \cdot n_{\Gamma_c} = \frac{1}{2} \, (k^\star e_1) \cdot n_{\Gamma_c} \ \ \text{on $\Gamma_c$}, \qquad \nabla \psi \cdot n_{\Gamma_f} = -\frac{1}{2} \, (k^\star e_1) \cdot n_{\Gamma_f} \ \ \text{on $\Gamma_f$}.
\end{cases}
$$
We of course wish to enrich the Lagrange multiplier space by functions independent of $k^\star$. Using the linearity of the above problem, we introduce the solutions $\psi_{0,j}$, $j=1,2$, to
\begin{equation} \label{eq:psi_0j}
\! \! \begin{cases}
  - \Delta \psi_{0,j} + \psi_{0,j} = 0 \quad \text{in $D_c$},
  \\ \noalign{\vskip 3pt}
  \dps \nabla \psi_{0,j} \cdot n_{\Gamma_c} = \frac{1}{2} \, e_j \cdot n_{\Gamma_c} \ \ \text{on $\Gamma_c$}, \quad \nabla \psi_{0,j} \cdot n_{\Gamma_f} = -\frac{1}{2} \, e_j \cdot n_{\Gamma_f} \ \ \text{on $\Gamma_f$},
\end{cases}
\end{equation}
and we get that $\psi = (e_1^T k^\star e_1) \, \psi_{0,1} + (e_2^T k^\star e_1) \, \psi_{0,2}$. Note that the function $\psi_{0,1}$ is identical to the enrichment defined by~\eqref{eq:psi_0}.

In the spirit of~\eqref{eq:def_W_H_enrich}, we therefore enrich the classical finite element space $W_H$ and consider
\begin{equation} \label{eq:def_W_H_enrich_matrix}
W_H^{\rm enrich} = W_H + \text{Span $\psi_{0,1}$} + \text{Span $\psi_{0,2}$}
\end{equation}
instead of $W_H$. In practice, and similarly to the scalar-case, we use accurate finite element approximations of $\psi_{0,j}$, $j=1,2$. We again introduce the finite element space
$$
W_h = \left\{ \phi^h \in H^1(D_c), \quad \text{$\phi^h$ is piecewise affine on the fine mesh ${\cal T}_h$} \right\},
$$
and define $\psi_{0,j}^h \in W_h$ as the solution to the following variational formulation (compare with~\eqref{eq:psi_0_FV_h}):
$$
\forall \phi^h \in W_h, \qquad {\cal C}(\psi_{0,j}^h,\phi^h) = \frac{1}{2} \int_{\Gamma_c} (e_j \cdot n_{\Gamma_c}) \, \phi^h - \frac{1}{2} \int_{\Gamma_f} (e_j \cdot n_{\Gamma_f}) \, \phi^h.
$$
In what follows, we therefore consider the enriched space
\begin{equation} \label{eq:def_enriched_matrix}
W_{H,h}^{\rm enrich} = W_H + \text{Span $\psi_{0,1}^h$} + \text{Span $\psi_{0,2}^h$}
\end{equation}
instead of $W_H^{\rm enrich}$. We again discretize the coupled problem~\eqref{eq:pb_min} by~\eqref{eq:pb_min_H_enriched} (which corresponds to the variational formulation~\eqref{eq:arlequin}), where $W_{H,h}^{\rm enrich}$ is now given by~\eqref{eq:def_enriched_matrix}. As when $\overline{k}$ is scalar-valued, the system~\eqref{eq:arlequin} has a unique solution for any positive definite symmetric matrix $\overline{k}$.

\subsection{Optimization upon the coefficient $\overline{k}$} \label{sec:opt_matrix}

We start again from the optimality system~\eqref{eq:arlequin} of the minimization problem~\eqref{eq:pb_min_H_enriched}: find $\overline{u}^H \in V_H^{\rm Dir BC}$, $\widecheck{u}^h_\eps \in V_h$ and $\psi^H \in W_{H,h}^{\rm enrich}$ such that
\begin{equation} \label{eq:arlequin_matrix}
  \begin{cases}
  \forall \overline{v}^H \in V_H^0, \quad & \overline{A}_{\overline{k}}(\overline{u}^H,\overline{v}^H) + {\cal C}(\overline{v}^H,\psi^H) = 0,
  \\ \noalign{\vskip 3pt}
  \forall \widecheck{v}^h \in V_h, \quad & \widecheck{A}_{k_\eps}(\widecheck{u}^h_\eps,\widecheck{v}^h) - {\cal C}(\widecheck{v}^h,\psi^H) = 0,
  \\ \noalign{\vskip 3pt}
  \forall \phi^H \in W_{H,h}^{\rm enrich}, \quad & {\cal C}(\overline{u}^H-\widecheck{u}^h_\eps,\phi^H) = 0.
  \end{cases}
\end{equation}
Instead of~\eqref{eq:optim_J}, we consider, for some positive constants $c_-$ and $c_+$ fixed throughout this appendix (with $0 < c_- < c_+$), the minimization problem
\begin{equation} \label{eq:optim_J_matrix}
  I_{\eps, H, h} = \inf \left\{ J_{\eps, H, h}(\overline{k}), \quad \overline{k} \in {\cal M}(c_-,c_+) \right\},
\end{equation}
with
\begin{equation} \label{eq:def_cal_M}
{\cal M}(c_-,c_+) = \left\{ \text{symmetric $2 \times 2$ matrix $\overline{k}$ such that $c_- \leq \overline{k} \leq c_+$} \right\}
\end{equation}
and where $J_{\eps, H, h}(\overline{k})$ is defined by~\eqref{eq:def_J} (as for the scalar-valued case) from the solution to~\eqref{eq:arlequin_matrix}.

\begin{remark}
In the minimization set for~\eqref{eq:optim_J_matrix}, and in sharp contrast to the scalar case addressed in Section~\ref{sec:scalar}, we prescribe some explicit minimal and maximal ellipticity constants. If we remove this constraint from~\eqref{eq:optim_J_matrix}, it is unclear to us how to prove the existence of an optimal coefficient $\overline{k}^{\rm opt}_\eps$ for a fixed value of $\eps>0$. Stated otherwise, it is unclear to us how to extend the arguments of Sections~\ref{sec:k_n} and~\ref{sec:non_0} to the matrix-valued case.
\end{remark}

\begin{remark}\label{matrix_dirichlet_x_2}
  Note that, in sharp contrast to the case when $\overline{k}$ is scalar-valued, we can at best hope to identify the vector $k^\star \, e_1$ in the limit $\eps \to 0$, and not the entire matrix $k^\star$. Indeed, momentarily replacing $k_\eps$ by its homogenized limit $k^\star$ and assuming that $\overline{k} \, e_1 = k^\star \, e_1$ (and ignoring any space discretization), we find that the solution to~\eqref{eq:arlequin_matrix} satisfies $\overline{u}(x) = x_1$ in $D \cup D_c$, and hence $J(\overline{k}) = 0$ in that case, thus reaching the minimum in~\eqref{eq:optim_J_matrix}. 

  In the case of matrix-valued coefficients $\overline{k}$, we therefore need to consider both~\eqref{eq:arlequin_matrix} (that is~\eqref{eq:pb_min_H_enriched}) and~\eqref{eq:pb_min_H_x2} in order to recover all the components of the matrix $k^\star$. In this Appendix~\ref{seq:matrix}, we focus on~\eqref{eq:arlequin_matrix} and thus on recovering $k^\star \, e_1$.
\end{remark}

\subsection{A useful technical result: the specific case of homogeneous materials} \label{sec:tech_lemmas_matrix}

As in the scalar-case, we begin our analysis by a useful auxiliary result. We first point out that the homogenization results stated in Lemma~\ref{th:multiplication} (and also Corollary~\ref{coro:multiplication}, assuming there that $\overline{k}_0$ is symmetric positive definite) still hold in the matrix case. We have never used in their proof the fact that $k^\star$ or $\overline{k}$ was scalar-valued. We now turn to extending Lemma~\ref{lemma:optimization} to the matrix case.

To that purpose, we study the following system (where, as in~\eqref{eq:arlequin_limit_lemma}, we have set $h=0$): for any $\overline{k}_a$ and $\overline{k}_b$ in ${\cal M}(c_-,c_+)$, find $\overline{u}^H_{\overline{k}_a, \overline{k}_b} \in V_H^{\rm Dir BC}$, $\widecheck{u}_{\overline{k}_a, \overline{k}_b} \in H^1(D_c \cup D_f)$ and $\psi^H_{\overline{k}_a, \overline{k}_b} \in W_H^{\rm enrich}$ such that
\begin{equation} \label{eq:arlequin_limit_lemma_matrix}
  \begin{cases}
  \forall \overline{v}^H \in V_H^0, \quad & \overline{A}_{\overline{k}_a}(\overline{u}^H_{\overline{k}_a, \overline{k}_b},\overline{v}^H) + {\cal C}(\overline{v}^H,\psi^H_{\overline{k}_a, \overline{k}_b}) = 0,
  \\ \noalign{\vskip 3pt}
  \forall \widecheck{v} \in H^1(D_c \cup D_f), \quad & \widecheck{A}_{\overline{k}_b}(\widecheck{u}_{\overline{k}_a, \overline{k}_b},\widecheck{v}) - {\cal C}(\widecheck{v},\psi^H_{\overline{k}_a, \overline{k}_b}) = 0,
  \\ \noalign{\vskip 3pt}
  \forall \phi^H \in W_H^{\rm enrich}, \quad & {\cal C}(\overline{u}^H_{\overline{k}_a, \overline{k}_b}-\widecheck{u}_{\overline{k}_a, \overline{k}_b},\phi^H) = 0.
  \end{cases}
\end{equation}

\begin{lemma}\label{lemma:optimization_matrix}
  We consider~\eqref{eq:arlequin_limit_lemma_matrix} for some constant matrices $\overline{k}_a$ and $\overline{k}_b$ in ${\cal M}(c_-,c_+)$.
  
  If $\overline{k}_a \, e_1 = \overline{k}_b \, e_1$, then the solution to~\eqref{eq:arlequin_limit_lemma_matrix} is $\overline{u}^H_{\overline{k}_a, \overline{k}_b}(x) = x_1$ in $D \cup D_c$, $\widecheck{u}_{\overline{k}_a, \overline{k}_b}(x) = x_1$ in $D_c \cup D_f$ and $\dps \psi^H_{\overline{k}_a, \overline{k}_b} = \sum_{j=1}^2 (e_j^T \, \overline{k}_a e_1) \, \psi_{0,j}$ in $D_c$, where $\psi_{0,j}$ are the Lagrange multiplier functions defined by~\eqref{eq:psi_0j}.

  Conversely, if $(\overline{u}^H_{\overline{k}_a, \overline{k}_b},\widecheck{u}_{\overline{k}_a, \overline{k}_b},\psi^H_{\overline{k}_a, \overline{k}_b})$ is a solution to~\eqref{eq:arlequin_limit_lemma_matrix} with $\overline{u}^H_{\overline{k}_a, \overline{k}_b}(x) = x_1$ in $D \cup D_c$, then $\widecheck{u}_{\overline{k}_a, \overline{k}_b}(x) = x_1$ in $D_c \cup D_f$, $\dps \psi^H_{\overline{k}_a, \overline{k}_b} = \sum_{j=1}^2 (e_j^T \, \overline{k}_a e_1) \, \psi_{0,j}$ in $D_c$ and $\overline{k}_a \, e_1 = \overline{k}_b \, e_1$.
\end{lemma}

\begin{proof}
We start by the first assertion and assume $\overline{k}_a \, e_1 = \overline{k}_b \, e_1$. We immediately get the result, recalling that the system~\eqref{eq:arlequin_limit_lemma_matrix} has a unique solution and noticing that $\dps \left( \overline{u}^H_{\overline{k}_a,\overline{k}_b}(x),\widecheck{u}_{\overline{k}_a,\overline{k}_b}(x),\psi^H_{\overline{k}_a,\overline{k}_b}(x) \right) = \left( x_1,x_1, \sum_{j=1}^2 (e_j^T \, \overline{k}_a e_1) \, \psi_{0,j}(x) \right)$, where $\psi_{0,j}$ are defined by~\eqref{eq:psi_0j}, is a solution to~\eqref{eq:arlequin_limit_lemma_matrix}.

We now turn to the second assertion and hence assume that $\overline{u}^H_{\overline{k}_a,\overline{k}_b}(x) = x_1$ in $D \cup D_c$. The first line of~\eqref{eq:arlequin_limit_lemma_matrix} reads as
\begin{equation}\label{eq:first_line_matrix}
  \forall \overline{v}^H \in V_H^0, \qquad \overline{A}_{\overline{k}_a}(x_1,\overline{v}^H) + {\cal C}(\overline{v}^H,\psi^H_{\overline{k}_a,\overline{k}_b}) = 0.
\end{equation}
Since $\psi^H_{\overline{k}_a,\overline{k}_b} \in W_H^{\rm enrich} = W_H + \text{Span $\psi_{0,1}$} + \text{Span $\psi_{0,2}$}$, we can represent it as 
\begin{equation}\label{eq:lagrange_multiplier_expansion_matrix}
  \psi^H_{\overline{k}_a,\overline{k}_b} = \widetilde{\psi}^H + \sum_{j=1}^2 \tau_j \, \psi_{0,j},
\end{equation} 
for some $\tau_j \in \RR$ and some $\widetilde{\psi}^H \in W_H$. 

We infer from~\eqref{eq:first_line_matrix} and~\eqref{eq:lagrange_multiplier_expansion_matrix} that
$$
\forall \overline{v}^H \in V_H^0, \qquad {\cal C}(\overline{v}^H,\widetilde{\psi}^H) = - \overline{A}_{\overline{k}_a}(x_1,\overline{v}^H) - \sum_{j=1}^2 \tau_j \, {\cal C}(\overline{v}^H,\psi_{0,j}),
$$
which provides us with an expression of $\widetilde{\psi}^H$ in terms of the $\tau_j$ and the entries of $\overline{k}_a$, using the linearity of the problem: 
\begin{equation} \label{eq:covid6_matrix}
\widetilde{\psi}^H = \sum_{j=1}^2 (e_j^T \, \overline{k}_a e_1) \, \widetilde{\psi}_{2,j}^H - \sum_{j=1}^2 \tau_j \, \widetilde{\psi}_{1,j}^H,
\end{equation}
with $\widetilde{\psi}_{1,j}^H \in W_H$ and $\widetilde{\psi}_{2,j}^H \in W_H$ uniquely defined by
\begin{equation} \label{eq:LMtech_matrix}
  \begin{array}{l}
    \forall \overline{v}^H \in V_H^0, \qquad {\cal C}(\overline{v}^H,\widetilde{\psi}_{1,j}^H) = {\cal C}(\overline{v}^H,\psi_{0,j}),
    \\ \noalign{\vskip 3pt}
    \forall \overline{v}^H \in V_H^0, \qquad {\cal C}(\overline{v}^H,\widetilde{\psi}_{2,j}^H) = -\overline{A}_{e_j}(\overline{v}^H),
  \end{array}
\end{equation}
where the linear form $\overline{A}_{e_j}$ is defined by
$$
\overline{A}_{e_j}(\overline{v}) = \int_D e_j \cdot \nabla \overline{v}(x) + \frac{1}{2} \int_{D_c} e_j \cdot \nabla \overline{v}(x).
$$

In the sequel, we again use the $H^1$-orthogonal projection operators to the coarse finite element spaces $\Pi_H : H^1(D_c) \to W_H$ and $\Pi_H^{\rm enrich} : H^1(D_c) \to W^{\rm enrich}_H$ defined by~\eqref{eq:Pih_a} and~\eqref{eq:Pih_b}, where $W^{\rm enrich}_H$ is now of course defined by~\eqref{eq:def_W_H_enrich_matrix}.

In view of the first line of~\eqref{eq:LMtech_matrix} and of the definition of $\Pi_H$, we have that $\widetilde{\psi}^H_{1,j} = \Pi_H(\psi_{0,j})$. We next observe that the Lagrange multipliers $\psi_{0,j}$ defined by~\eqref{eq:psi_0j} satisfy
$$
\forall \overline{v} \in H^1(D \cup D_c) \ \ \text{with $\overline{v} = 0$ on $\Gamma$}, \quad \overline{A}_{e_j}(\overline{v}) = - {\cal C}(\overline{v},\psi_{0,j}).
$$
Gathering this relation with the second line of~\eqref{eq:LMtech_matrix}, we see that $\widetilde{\psi}^H_{2,j}$ satisfies the same equation as $\widetilde{\psi}^H_{1,j}$. We thus have $\widetilde{\psi}^H_{2,j} = \widetilde{\psi}^H_{1,j} = \Pi_H(\psi_{0,j})$. Inserting this relation in~\eqref{eq:covid6_matrix}, we deduce that $\dps \widetilde{\psi}^H = \sum_{j=1}^2 (e_j^T \, \overline{k}_a e_1 - \tau_j) \, \Pi_H(\psi_{0,j})$, and thus
\begin{equation}\label{eq:psi_decomposition_matrix}
  \psi^H_{\overline{k}_a,\overline{k}_b} = \sum_{j=1}^2 \tau_j \, \psi_{0,j} + \sum_{j=1}^2 (e_j^T \, \overline{k}_a e_1 - \tau_j) \, \Pi_H(\psi_{0,j}),
\end{equation}
where the coefficients $\tau_j$ will be determined later.

\medskip

We now turn to the second line of~\eqref{eq:arlequin_limit_lemma_matrix}. Let us introduce $\widecheck{u}_{1,j}(\overline{k}_b)$ and $\widecheck{u}_{2,j}(\overline{k}_b)$ in $H^1(D_c \cup D_f)$ such that $\dps \int_{D_c} \widecheck{u}_{1,j}(\overline{k}_b) = \int_{D_c} \widecheck{u}_{2,j}(\overline{k}_b) = 0$ and 
\begin{equation}\label{eq:u1_u2_matrix}
  \begin{array}{ll}
    \forall \widecheck{v} \in H^1(D_c \cup D_f), \qquad & \widecheck{A}_{\overline{k}_b}\big(\widecheck{u}_{1,j}(\overline{k}_b),\widecheck{v}\big) = {\cal C}(\widecheck{v},\psi_{0,j}),
    \\ \noalign{\vskip 3pt}
    \forall \widecheck{v} \in H^1(D_c \cup D_f), \qquad & \widecheck{A}_{\overline{k}_b}\big(\widecheck{u}_{2,j}(\overline{k}_b),\widecheck{v}\big) = {\cal C}\big(\widecheck{v},\Pi_H(\psi_{0,j})\big).
  \end{array}
\end{equation}
In contrast to the case when $\overline{k}_b$ is scalar, the functions $\widecheck{u}_{2,j}(\overline{k}_b)$ depend on $\overline{k}_b$ in a complex manner. On the other hand, an explicit expression for $\widecheck{u}_{1,j}(\overline{k}_b)$ can be easily determined. Indeed, we observe that,
$$
\forall \widecheck{v} \in H^1(D_c \cup D_f), \quad {\cal C}(\widecheck{v},\psi_{0,j}) = \frac{1}{2} \int_{D_c} e_j \cdot \nabla \widecheck{v}(x) + \int_{D_f} e_j \cdot \nabla \widecheck{v}(x).
$$
The functions $\widecheck{u}_{1,j}(\overline{k}_b)$ defined by $\widecheck{u}_{1,j}(\overline{k}_b)(x) = [(\overline{k}_b)^{-1} \, e_j] \cdot x$ on $D_c \cup D_f$ have a vanishing mean over $D_c$ and are such that $\overline{k}_b \nabla \widecheck{u}_{1,j}(\overline{k}_b) = e_j$. They are hence the unique solution to the first line of~\eqref{eq:u1_u2_matrix}.

Inserting the expression~\eqref{eq:psi_decomposition_matrix} for $\psi^H_{\overline{k}_a, \overline{k}_b}$ in the second line of~\eqref{eq:arlequin_limit_lemma_matrix}, we obtain
$$
  \widecheck{u}_{\overline{k}_a, \overline{k}_b} = \lambda + \sum_{j=1}^2 \tau_j \, \widecheck{u}_{1,j}(\overline{k}_b) + \sum_{j=1}^2 (e_j^T \, \overline{k}_a e_1 - \tau_j) \, \widecheck{u}_{2,j}(\overline{k}_b),
  $$
where $\lambda \in \RR$ is an arbitrary constant.

\medskip

To identify the constants $\lambda$ and $\tau_j$, we use the third line of~\eqref{eq:arlequin_limit_lemma_matrix}, that reads as: for any $\phi^H \in W_H^{\rm enrich}$,
$$
{\cal C}\left(\lambda + \sum_{j=1}^2 \tau_j \, \widecheck{u}_{1,j}(\overline{k}_b) + \sum_{j=1}^2 (e_j^T \, \overline{k}_a e_1 - \tau_j) \, \widecheck{u}_{2,j}(\overline{k}_b) - \overline{u}^H_{\overline{k}_a,\overline{k}_b}, \phi^H \right) = 0.
$$
Taking $\phi^H = 1$ and using that the mean over $D_c$ of $\overline{u}^H_{\overline{k}_a,\overline{k}_b}$, $\widecheck{u}_{1,j}$ and $\widecheck{u}_{2,j}$ vanishes, we get $\lambda = 0$. Now taking
$$
\phi^H = \Pi^{\rm enrich}_H\left(\sum_{j=1}^2 \tau_j \, \widecheck{u}_{1,j}(\overline{k}_b) + \sum_{j=1}^2 (e_j^T \, \overline{k}_a e_1 - \tau_j) \, \widecheck{u}_{2,j}(\overline{k}_b) - \overline{u}^H_{\overline{k}_a,\overline{k}_b}\right),
$$
we obtain that
$$
\Pi^{\rm enrich}_H\left(\sum_{j=1}^2 \tau_j \, \widecheck{u}_{1,j}(\overline{k}_b) + \sum_{j=1}^2 (e_j^T \, \overline{k}_a e_1 - \tau_j) \, \widecheck{u}_{2,j}(\overline{k}_b) - \overline{u}^H_{\overline{k}_a,\overline{k}_b}\right) = 0,
$$
which reads, since $\widecheck{u}_{1,j}(\overline{k}_b)$ and $\overline{u}^H_{\overline{k}_a,\overline{k}_b}$ belong to $W_H^{\rm enrich}$, as 
$$
\sum_{j=1}^2 \tau_j \, \widecheck{u}_{1,j}(\overline{k}_b) + \sum_{j=1}^2 (e_j^T \, \overline{k}_a e_1 - \tau_j) \, \Pi^{\rm enrich}_H\left(\widecheck{u}_{2,j}(\overline{k}_b)\right) - \overline{u}^H_{\overline{k}_a,\overline{k}_b} = 0 \quad \text{in $D_c$}.
$$
Using the explicit expression of $\widecheck{u}_{1,j}(\overline{k}_b)$ and $\overline{u}^H_{\overline{k}_a,\overline{k}_b}$, we next observe that, taking $\beta \in \RR^2$ such that $(\overline{k}_b)^{-1} \, \beta = e_1$, we have $\dps \overline{u}^H_{\overline{k}_a,\overline{k}_b} = \sum_{j=1}^2 \beta_j \, \widecheck{u}_{1,j}(\overline{k}_b)$ in $D_c$. The above relation hence implies that
\begin{equation} \label{eq:vacances}
  \sum_{j=1}^2 (\tau_j-\beta_j) \, \widecheck{u}_{1,j}(\overline{k}_b) + \sum_{j=1}^2 (e_j^T \, \overline{k}_a e_1 - \tau_j) \, \Pi^{\rm enrich}_H\left(\widecheck{u}_{2,j}(\overline{k}_b)\right) = 0 \quad \text{in $D_c$}.
\end{equation}

We now claim that the four functions $\widecheck{u}_{1,j}(\overline{k}_b)$ and $\Pi^{\rm enrich}_H (\widecheck{u}_{2,j}(\overline{k}_b))$, $j=1,2$, are linearly independent on $D_c$. In order to prove this claim, we argue by contradiction and assume that there exist real numbers $\gamma_{1,j}$ and $\gamma_{2,j}$ such that 
\begin{equation}\label{eq:collinearity_matrix}
  \sum_{j=1}^2 \gamma_{1,j} \, \widecheck{u}_{1,j}(\overline{k}_b) + \sum_{j=1}^2 \gamma_{2,j} \, \Pi^{\rm enrich}_H\left(\widecheck{u}_{2,j}(\overline{k}_b)\right) = 0 \quad \text{in $D_c$}.
\end{equation}
Let us introduce
$$
\widecheck{w} = \sum_{j=1}^2 \gamma_{1,j} \, \widecheck{u}_{1,j}(\overline{k}_b) + \sum_{j=1}^2 \gamma_{2,j} \, \widecheck{u}_{2,j}(\overline{k}_b).
$$
For any $\widecheck{v} \in H^1(D_c \cup D_f)$, we compute, using~\eqref{eq:u1_u2_matrix}, that
\begin{align*}
  \widecheck{A}_{\overline{k}_b}(\widecheck{w},\widecheck{v})
  &=
  \sum_{j=1}^2 \gamma_{1,j} \, \widecheck{A}_{\overline{k}_b}(\widecheck{u}_{1,j}(\overline{k}_b),\widecheck{v}) + \sum_{j=1}^2 \gamma_{2,j} \, \widecheck{A}_{\overline{k}_b}(\widecheck{u}_{2,j}(\overline{k}_b),\widecheck{v})
  \\
  &=
  \sum_{j=1}^2 \gamma_{1,j} \, {\cal C}(\widecheck{v},\psi_{0,j}) + \sum_{j=1}^2 \gamma_{2,j} \, {\cal C}(\widecheck{v},\Pi_H(\psi_{0,j}))
  \\
  &=
  {\cal C}\left(\widecheck{v},\sum_{j=1}^2 \gamma_{1,j} \, \psi_{0,j} + \sum_{j=1}^2 \gamma_{2,j} \, \Pi_H(\psi_{0,j}) \right).
\end{align*}
Taking $\widecheck{v} = \widecheck{w}$ in the equation above, we obtain that
\begin{align} 
  \widecheck{A}_{\overline{k}_b}(\widecheck{w},\widecheck{w})
  &=
  {\cal C}\left(\widecheck{w},\sum_{j=1}^2 \gamma_{1,j} \, \psi_{0,j} + \sum_{j=1}^2 \gamma_{2,j} \, \Pi_H(\psi_{0,j}) \right)
  \nonumber
  \\
  &=
  {\cal C}\left(\Pi_H^{\rm enrich}(\widecheck{w}),\sum_{j=1}^2 \gamma_{1,j} \, \psi_{0,j} + \sum_{j=1}^2 \gamma_{2,j} \, \Pi_H(\psi_{0,j}) \right),
  \label{eq:colliniearity2_matrix}
\end{align}
where the last equality stems from the definition of the projection operator $\Pi_H^{\rm enrich}$. We next observe that $\Pi_H^{\rm enrich}(\widecheck{w}) = 0$, because of~\eqref{eq:collinearity_matrix} and the fact that $\Pi_H^{\rm enrich}(\widecheck{u}_{1,j}(\overline{k}_b)) = \widecheck{u}_{1,j}(\overline{k}_b)$ (recall that $\widecheck{u}_{1,j}(\overline{k}_b)$ is a linear function on $D_c \cup D_f$ and thus belongs to $W_H^{\rm enrich}$). The right-hand side of~\eqref{eq:colliniearity2_matrix} thus vanishes. By definition of the bilinear form $\widecheck{A}_{\overline{k}_b}$, this implies that $\widecheck{w} = \widehat{\lambda}$ on $D_c \cup D_f$ for some constant $\widehat{\lambda}$. Since the average of $\widecheck{w}$ over $D_c$ vanishes, we obtain $\widehat{\lambda}=0$ and thus $\widecheck{w}=0$.

We thus infer from~\eqref{eq:u1_u2_matrix} that, for any $\widecheck{v} \in H^1(D_c \cup D_f)$,
$$
{\cal C}\left(\widecheck{v},\sum_{j=1}^2 \gamma_{1,j} \, \psi_{0,j} + \sum_{j=1}^2 \gamma_{2,j} \, \Pi_H(\psi_{0,j})\right) = \widecheck{A}_{\overline{k}_b}(\widecheck{w},\widecheck{v}) = 0.
$$
This yields $\dps \sum_{j=1}^2 \gamma_{1,j} \, \psi_{0,j} + \sum_{j=1}^2 \gamma_{2,j} \, \Pi_H(\psi_{0,j}) = 0$, and thus that $\dps \sum_{j=1}^2 \gamma_{1,j} \, \psi_{0,j} \in W_H$. The functions $\psi_{0,j}$ are solutions to~\eqref{eq:psi_0j} and thus cannot be, in general, equal to finite element functions. We hence obtain that $\dps \sum_{j=1}^2 \gamma_{1,j} \, \psi_{0,j} = 0$, and thus $\gamma_{1,j} = 0$ for $j=1,2$. We then get $\dps \Pi_H\left( \sum_{j=1}^2 \gamma_{2,j} \, \psi_{0,j} \right) = 0$, and thus, for any $\varphi^H \in W_H$,
\begin{multline*}
0
=
\left(\Pi_H\left( \sum_{j=1}^2 \gamma_{2,j} \, \psi_{0,j} \right),\varphi^H\right)_{H^1(D_c)} = \left(\sum_{j=1}^2 \gamma_{2,j} \, \psi_{0,j},\varphi^H\right)_{H^1(D_c)}
\\ =
\sum_{j=1}^2 \gamma_{2,j} \, {\cal C}(\psi_{0,j},\varphi^H)
=
\frac{1}{2} \int_{\Gamma_c} (\overline{e} \cdot n_{\Gamma_c}) \, \varphi^H - \frac{1}{2} \int_{\Gamma_f} (\overline{e} \cdot n_{\Gamma_f}) \, \varphi^H,
\end{multline*}
where $\dps \overline{e} = \sum_{j=1}^2 \gamma_{2,j} \, e_j$ and where we have used the variational formulation satisfied by $\psi_{0,j}$ in the last equality. Since the value of $\varphi^H$ can be chosen independently on $\Gamma_c$ and $\Gamma_f$, this implies that, for any $\varphi^H \in W_H$, we have $\dps \int_{\Gamma_c} \varphi^H \, \overline{e} \cdot n_{\Gamma_c} = 0 = \int_{\Gamma_f} \varphi^H \, \overline{e} \cdot n_{\Gamma_f}$. If $\overline{e} \neq 0$, then $\overline{e} \cdot n_{\Gamma_c}$ does not identically vanishes on $\Gamma_c$, and we obtain a contradiction. We thus have $\overline{e} = 0$, and thus $\gamma_{2,j} = 0$ for $j=1,2$.

We hence have shown that $\gamma_{1,j} = \gamma_{2,j} = 0$ for $j=1,2$, which concludes the proof of our claim that the functions $\widecheck{u}_{1,j}(\overline{k}_b)$ and $\Pi^{\rm enrich}_H (\widecheck{u}_{2,j}(\overline{k}_b))$ are linearly independent on $D_c$.

\medskip

We now return to~\eqref{eq:vacances} and deduce that, for $j=1,2$, we have $\beta_j = \tau_j = e_j^T \, \overline{k}_a e_1$. Recalling that $\beta = \overline{k}_b e_1$, we hence have shown that $\overline{k}_b e_1 = \overline{k}_a e_1$. This concludes the proof of Lemma~\ref{lemma:optimization_matrix}.
\end{proof}

\subsection{Well-posedness of the optimization problem upon $\overline{k}$ at fixed $\eps$} \label{sec:well_matrix}

As in Section~\ref{sec:opt}, we aim at showing the existence of a minimizer $\overline{k}^{\rm opt}_\eps$ of~\eqref{eq:optim_J_matrix} for a fixed value of $\eps>0$. For that purpose, we again consider a minimizing sequence $\{ \overline{k}^n \}_{n \in \NN}$ of the optimization problem~\eqref{eq:optim_J_matrix}, that is a sequence $\overline{k}^n \in {\cal M}(c_-,c_+)$ that satisfies the inequality
$$
I_{\eps, H, h} \leq J_{\eps, H, h}(\overline{k}^n) = \int_{D \cup D_c} \big| \nabla \overline{u}^H_{\overline{k}^n, k_\eps} - e_1 \big|^2 \leq I_{\eps, H, h} + \frac{1}{n},
$$
where $(\overline{u}^H_{\overline{k}^n, k_\eps},\widecheck{u}^h_{\eps,\overline{k}^n, k_\eps}, \psi^H_{\overline{k}^n, k_\eps}) \in V_H^{\rm Dir BC} \times V_h \times W_{H,h}^{\rm enrich}$ is the solution to~\eqref{eq:arlequin_matrix} for the tentative constant coefficient $\overline{k} = \overline{k}^n$.

\medskip

In contrast to the scalar case, the set ${\cal M}(c_-,c_+)$ in which we look for the constant matrix $\overline{k}$ is compact. We thus immediately obtain that, up to the extraction of a subsequence, $\overline{k}^n$ converges to some $\overline{k}^{\rm opt}_\eps \in {\cal M}(c_-,c_+)$. Moreover, since $\overline{k}^n$ remains isolated from 0 and $\infty$ (in the sense of symmetric matrices), the map $\overline{k}^n \mapsto \overline{u}^H_{\overline{k}^n, k_\eps}$ is continuous from ${\cal M}(c_-,c_+)$ to $H^1(D \cup D_c)$, and thus $J_{\eps, H, h}(\overline{k}^n)$ converges to $J_{\eps, H, h}(\overline{k}^{\rm opt}_\eps)$. This shows that $\overline{k}^{\rm opt}_\eps$ is a minimizer of~\eqref{eq:optim_J_matrix}. We hence have shown the following result. 

\begin{theorem}\label{th:optimization_matrix}
  Let $k_\eps$ be a symmetric matrix that satisfies the classical boundedness and coercivity conditions~\eqref{eq:boundedness+coercivity}. 
  
  For any fixed $\eps>0$, $H > 0$ and $h > 0$, and for any positive constants $c_+ > c_- > 0$, the optimization problem~\eqref{eq:optim_J_matrix} has at least one minimizer $\overline{k}^{\rm opt}_\eps(H,h) \in {\cal M}(c_-,c_+)$.
\end{theorem}

\subsection{Homogenized limit}\label{sec:homogenization_matrix}

We are now going to pass to the limit $\eps \to 0$. Since $h$ has to be chosen much smaller than $\eps$, this implies that we also have $h \to 0$. For simplicity, and as in Section~\ref{sec:homogenization}, we hereafter fix $h = 0$. For each $\eps>0$, we know from Theorem~\ref{th:optimization_matrix} (which also holds true if we set $h=0$) that there exists at least one optimal constant matrix $\overline{k}^{\rm opt}_\eps$ (which minimizes~\eqref{eq:optim_J_matrix}) with the corresponding solution $(\overline{u}^H_{\overline{k}^{\rm opt}_\eps,k_\eps},\widecheck{u}^h_{\eps,\overline{k}^{\rm opt}_\eps,k_\eps},\psi^H_{\overline{k}^{\rm opt}_\eps,k_\eps}) \in V_H^{\rm Dir BC} \times H^1(D_c \cup D_f) \times W_H^{\rm enrich}$ to the system~\eqref{eq:arlequin2} (where of course $W_H^{\rm enrich}$ is defined in~\eqref{eq:arlequin2} by~\eqref{eq:def_W_H_enrich_matrix}).

We have assumed in~\eqref{eq:structure-k} that the sequence $k_\eps$ is such that $k_\eps = k_{\rm per}(\cdot/\eps)$ for some fixed periodic function $k_{\rm per}$. This periodicity assumption implies that the homogenized coefficient $k^\star$ exists and is constant, a fact that we are going to use below. Similarly to the scalar case, we perform our analysis in the periodic setting but believe that it actually carries over to more general cases (random stationary setting, \dots).

\begin{theorem}\label{th:homogenization_matrix}
Let $k_\eps$ be given by~\eqref{eq:structure-k} for some fixed periodic coefficient $k_{\rm per}$ that satisfies the classical boundedness and coercivity conditions~\eqref{eq:boundedness+coercivity}. We make the regularity assumption~\eqref{eq:holder_continuity} and the geometric assumption~\eqref{eq:boundaries}. We furthermore assume that the constants $c_+ > c_- > 0$ have been chosen in~\eqref{eq:def_cal_M} so that the homogenized coefficient $k^\star$ belongs to ${\cal M}(c_-,c_+)$.

Let $\overline{k}_\eps^{\rm opt}(H,h)$ be an optimal coefficient (the existence of which is provided by Theorem~\ref{th:optimization_matrix}). Then the vector $\overline{k}_\eps^{\rm opt}(H,h) \, e_1$ converges to $k^\star \, e_1$ when $h$ and $\eps$ go to $0$: for any $H>0$, we have
\begin{equation} \label{eq:main_result2_matrix}
\lim_{\eps \to 0} \lim_{h \to 0} \overline{k}_\eps^{\rm opt}(H,h) \, e_1 = k^\star \, e_1.
\end{equation}
\end{theorem}
We have assumed that $k_\eps = k_{\rm per}(\cdot/\eps)$ where $k_{\rm per}$ satisfies the boundedness and coercivity conditions~\eqref{eq:boundedness+coercivity}. This implies that $c_1 \leq k^\star \leq c_2$ in the sense of symmetric matrices. We can hence choose $c_+ = c_2$ and $c_- = c_1$ in~\eqref{eq:def_cal_M} and we then indeed obtain that $k^\star \in {\cal M}(c_-,c_+)$.

\medskip

In the proof below, we do not make explicit in the notation the fact that the optimal coefficient $\overline{k}^{\rm opt}_\eps$ depends on $H$.

\begin{proof}
  Since $\overline{k}_\eps^{\rm opt}$ belongs to the compact set ${\cal M}(c_-,c_+)$, we know that, up to the extraction of a subsequence, $\overline{k}_\eps^{\rm opt}$ converges to some $\overline{k}^{\rm opt}_0 \in {\cal M}(c_-,c_+)$ when $\eps \to 0$.

The constants $c_-$ and $c_+$ have been chosen so that the homogenized coefficient $k^\star$ is an admissible test coefficient in~\eqref{eq:optim_J_matrix}. Hence, by definition of $I_{\eps, H}$ (which is $I_{\eps, H, h}$ where we have formally set $h=0$), we have in particular
\begin{multline}\label{eq:u_bound_eps_matrix}
  \int_{D \cup D_c} \Big| \nabla \overline{u}^H_{\overline{k}_\eps^{\rm opt}, k_\eps} - e_1 \Big|^2 = J_{\eps, H}(\overline{k}_\eps^{\rm opt}) = I_{\eps, H} \\ \leq J_{\eps, H}(k^\star) = \int_{D \cup D_c} \Big| \nabla \overline{u}^H_{k^\star, k_\eps} - e_1 \Big|^2,
\end{multline}
where $\left(\overline{u}^H_{k^\star, k_\eps},\widecheck{u}_{\eps,k^\star, k_\eps}, \psi^H_{k^\star, k_\eps}\right)$ is the solution to~\eqref{eq:system_k_star} with the constant coefficient $\overline{k} = k^\star$. Since we know that $\left( \overline{u}^H_{k^\star, k_\eps}, \widecheck{u}_{\eps, k^\star, k_\eps}\right)$ is the minimizer of~\eqref{eq:pb_min_H_enriched} (with $h=0$) with $\overline{k} = k^\star$, we can compare its energy with that of the particular choice $\left(\overline{u}^H(x) = x_1, \widecheck{u}_\eps(x) = x_1\right)$. Writing that
$$
{\cal E}\left( \overline{u}^H_{k^\star, k_\eps}, \widecheck{u}_{\eps, k^\star, k_\eps}\right) \leq {\cal E}\left(\overline{u}^H(x) = x_1, \widecheck{u}_\eps(x) = x_1\right),
$$
we obtain, similarly as in~\eqref{eq:u_k_bound}, that $\nabla \overline{u}^H_{k^\star, k_\eps}$ is bounded in $L^2(D \cup D_c)$ uniformly in $H$ and $\eps$. We then infer from~\eqref{eq:u_bound_eps_matrix} that $\overline{u}^H_{\overline{k}_\eps^{\rm opt}, k_\eps}$ is bounded in $H^1(D \cup D_c)$ uniformly in $H$ and $\eps$. Proceeding as in Section~\ref{sec:bounds}, we deduce that $\widecheck{u}_{\eps,\overline{k}^{\rm opt}_\eps,k_\eps}$ (resp. $\psi^H_{\overline{k}^{\rm opt}_\eps,k_\eps}$) is bounded in $H^1(D_c \cup D_f)$ (resp. $W_H^{\rm enrich}$) uniformly in $H$ and $\eps$.

We thus know that, up to the extraction of a subsequence, $\overline{u}^H_{\overline{k}_\eps^{\rm opt}, k_\eps}$ converges strongly in $H^1(D \cup D_c)$ to some $\overline{u}^H_\star$ when $\eps \to 0$, that $\widecheck{u}_{\eps,\overline{k}^{\rm opt}_\eps,k_\eps}$ converges weakly in $H^1(D_c \cup D_f)$ to some $\widecheck{u}_\star$ and that $\psi^H_{\overline{k}^{\rm opt}_\eps,k_\eps}$ converges strongly in $H^1(D_c)$ to some $\psi^H_\star$. We are thus in position to use Corollary~\ref{coro:multiplication}, which shows that the limit $(\overline{u}^H_\star,\widecheck{u}_\star,\psi^H_\star)$ is actually the solution $(\overline{u}^H_{\overline{k}^{\rm opt}_0,k^\star},\widecheck{u}_{\overline{k}^{\rm opt}_0,k^\star},\psi^H_{\overline{k}^{\rm opt}_0,k^\star})$ to~\eqref{eq:arlequin_limit2}.

\medskip

We now claim that the limit $\overline{u}^H_\star$ of $\overline{u}^H_{\overline{k}_\eps^{\rm opt}, k_\eps}$ satisfies
\begin{equation} \label{eq:lim_u_k^opt_bis_matrix}
  \overline{u}^H_\star(x) = x_1 \quad \text{in $D \cup D_c$}.
\end{equation}
Consider, as above, the solution $\left(\overline{u}^H_{k^\star, k_\eps},\widecheck{u}_{\eps,k^\star, k_\eps}, \psi^H_{k^\star, k_\eps}\right)$ to~\eqref{eq:system_k_star} with the constant coefficient $\overline{k} = k^\star$. We have shown that $\overline{u}^H_{k^\star, k_\eps}$ is bounded in $H^1(D \cup D_c)$ uniformly in $H$ and $\eps$, and thus converges (when $\eps \to 0$) to some $\overline{u}^H_{k^\star, \star}$. Proceeding again as in Section~\ref{sec:bounds}, we deduce that $\widecheck{u}_{\eps,k^\star,k_\eps}$ (resp. $\psi^H_{k^\star,k_\eps}$) is bounded in $H^1(D_c \cup D_f)$ (resp. $W_H^{\rm enrich}$) uniformly in $H$ and $\eps$. We are thus in position to use Lemma~\ref{th:multiplication}, which shows that the limit when $\eps \to 0$ of $\left(\overline{u}^H_{k^\star, k_\eps},\widecheck{u}_{\eps,k^\star, k_\eps}, \psi^H_{k^\star, k_\eps}\right)$ satisfies~\eqref{eq:system_contradiction} with $\overline{k} = k^\star$. In view of the first statement of Lemma~\ref{lemma:optimization_matrix}, we deduce that $\overline{u}^H_{k^\star, \star}(x) = x_1$ in $D \cup D_c$. Passing to the limit $\eps \to 0$ in~\eqref{eq:u_bound_eps_matrix}, we get
\begin{multline*}
\int_{D \cup D_c} \Big| \nabla \overline{u}^H_\star - e_1 \Big|^2 = \lim_{\eps \to 0} \int_{D \cup D_c} \Big| \nabla \overline{u}^H_{\overline{k}_\eps^{\rm opt}, k_\eps} - e_1 \Big|^2 \\ \leq \lim_{\eps \to 0} \int_{D \cup D_c} \Big| \nabla \overline{u}^H_{k^\star, k_\eps} - e_1 \Big|^2 = \int_{D \cup D_c} \Big| \nabla \overline{u}^H_{k^\star,\star} - e_1 \Big|^2 = 0,
\end{multline*}
hence~\eqref{eq:lim_u_k^opt_bis_matrix}.

\medskip

We hence have obtained that the constant matrix $\overline{k}^{\rm opt}_0$ is such that the solution $(\overline{u}^H_{\overline{k}^{\rm opt}_0,k^\star},\widecheck{u}_{\overline{k}^{\rm opt}_0,k^\star},\psi^H_{\overline{k}^{\rm opt}_0,k^\star})$ to~\eqref{eq:arlequin_limit2} satisfies $\overline{u}^H_{\overline{k}^{\rm opt}_0,k^\star}(x) = x_1$ in $D \cup D_c$. Using the second statement of Lemma~\ref{lemma:optimization_matrix}, we deduce that $\overline{k}^{\rm opt}_0 \, e_1 = k^\star \, e_1$, which is exactly~\eqref{eq:main_result2_matrix}. This concludes the proof of Theorem~\ref{th:homogenization_matrix}.   
\end{proof}

\section*{Acknowledgments}

The work of the authors is partially supported by ONR under grant N00014-20-1-2691 and by EOARD under grant FA8655-20-1-7043. The last two authors acknowledge the continuous support from these two agencies.

\bibliographystyle{plain}
\bibliography{biblio_standard_layout_FL}

\end{document}